\pdfoutput=1

\documentclass{amsart}

\usepackage{amsmath,amsfonts,amsthm,amssymb}
\usepackage[alphabetic]{amsrefs}
\usepackage[OT4]{fontenc}
\usepackage{float}
\usepackage{enumerate}
\usepackage{verbatim}

\usepackage[utf8]{inputenc}
\usepackage[polish,english]{babel}
\usepackage{tikz}
\usetikzlibrary{shapes.geometric}
\usepackage{xcolor}
\usetikzlibrary{shapes.multipart,shapes.arrows}
\usetikzlibrary{patterns,arrows,decorations.pathreplacing}
\usetikzlibrary{decorations.markings,arrows}

\usepackage{url}
\usepackage[all]{xy}
\usepackage{float}
\usepackage{enumerate}

\newtheorem{theorem}{Theorem}[section]

\newtheorem{prob}[theorem]{Problem}
\newtheorem{lemma}[theorem]{Lemma}
\newtheorem{lem}[theorem]{Lemma}
\newtheorem{theo}[theorem]{Theorem}

\newtheorem{prop}[theorem]{Proposition}

\newtheorem{cor}[theorem]{Corollary}

\newtheorem{maintheo}{Theorem}[part]

\theoremstyle{definition}
\newtheorem{rem}[theorem]{Remark}

\newtheorem{defi}[theorem]{Definition}

\newcommand{\parti}{\mathrm{\tilde{\partial}}}

\newcommand{\cay}{\widetilde{X}}

\theoremstyle{definition}
\newtheorem{definition}[theorem]{Definition}

\theoremstyle{remark}

\numberwithin{equation}{section}

\title{Bent walls for random groups in the square and hexagonal model}
\address{Institute of Mathematics, Polish Academy of Science,
Warsaw, {\'S}niadeckich 8}
\email{tomaszo@impan.pl}
\author{Tomasz Odrzyg{\'o}{\'z}d{\'z}}

\thanks{This paper was created as a result of the research project UMO-2015/18/M/ST1/00050 financed by the Polish National Science Center.}

\begin{document}
\maketitle

\begin{abstract}
We consider two random group models: the hexagonal model and the square model, defined as the quotient of a free group by a random set of reduced words of length four and six respectively. Our first main result is that in this model there exists sharp density threshold for Kazhdan's Property (T) and it equals $\frac{1}{3}$. Our second main result is that for densities $< \frac{3}{8}$ a random group in the square model with overwhelming probability does not have Property (T). Moreover, we provide a new version of the Isoperimetric Inequality  that concerns non-planar diagrams and we introduce new geometrical tools to investigate random groups: trees of loops, diagrams collared by a tree of loops and specific codimension one structures in the Cayley complex, called \textit{bent hypergraphs}.

\end{abstract}
\selectlanguage{english}
\renewcommand{\thepage}{\roman{page}}

\selectlanguage{english}

\section{Introduction}\label{Chapter:Introduction}
\renewcommand{\thepage}{\arabic{page}}
\setcounter{page}{1}

In \cite{odrz} the author introduced the square model for random group and prove that such random groups for densities $<\frac{1}{3}$, with overwhelming probability, do not satisfy Property (T), and that for densities $< \frac{1}{2}$ are infinite and hyperbolic. In his next paper \cite{cubu} the author proved that for densities $<\frac{3}{10}$ random groups in the square model act properly and cocompactly on a CAT(0) cube complex. This result was obtained independently (with better constant) by Duong in her Ph.D. thesis \cite{yen} where she proved that random groups in the square model act properly and cocompactly on a CAT(0) cube complex for densities $< \frac{1}{3}$. 

In this paper we investigate the square model and the hexagonal model. Both these models can be generalized to the same formal definition, called the $k$-gonal model, which was first introduced in \cite{angular} under the name of standard $k$-angular model.

\begin{defi}[$k$-gonal model]\label{Definition:k_gonal_model}
For a natural $n>2$ let $S_n$ be the set of $n$ letters. Choose a \textit{density} $d \in (0,1)$. Let $W_n$ be the set of all cyclically reduced words of length $k$ over $S_n$ (we allow inverses of letters). Let $R_n$ be a subset of $W_n$ having $(2n-1)^{kd}$ elements, chosen at random with the uniform distribution among all such sets. We define a \textit{random group in the $\mathcal{G}(n,k,d)$ model} to be a group given by the presentation $\left< S_n | R_n \right>$. The model $\mathcal{G}(n,k,d)$ will be also called the \textit{$k$-gonal model on $n$ generators at density $d$}.

For $k=3$ we call this model the \textit{triangular model} or \textit{\.Zuk model},
for $k=4$ we call it the \textit{square model} and for $k=6$ we call it the \textit{hexagonal model}.

We say that a property $\mathcal{P}$  holds with overwhelming probability (w.o.p.) in the $k$-gonal model at density $d$ if the probability that a random group in the $\mathcal{G}(n,k,d)$ model satisfies $\mathcal{P}$ tends to 1 as $n$ goes to infinity. 

We say that $d_{\mathcal{P}}$ is the \textit{sharp threshold} for property $\mathcal{P}$ iff the following two conditions are satisfied:
\begin{itemize}
\item for densities $< d_{\mathcal{P}}$ a random group in the $k$-gonal model w.o.p. does not have $\mathcal{P}$,
\item for densities $> d_{\mathcal{P}}$ a random group in the $k$-gonal model w.o.p. has $\mathcal{P}$.
\end{itemize}
\end{defi}

Several results are known to hold for $k$-gonal models for all values of $k \geq 2$: for densities $> \frac{1}{2}$ a random group is w.o.p. trivial or $\mathbb{Z}/{2 \mathbb{Z}}$ and for densities $<\frac{1}{2}$ a random groups is w.o.p. hyperbolic, infinite and torsion-free (see \cite[Theorem 1.]{angular}). Moreover, for densities $\frac{1}{k}$ a random group in the $k$-gonal model is w.o.p. free, while for densities $> \frac{1}{k}$ w.o.p. is not isomorphic to a nontrivial free group (see \cite[Theorem 2.]{angular}). 

The first main result of this paper is the following theorem.

\begin{maintheo}\label{MainTheorem:Sharp_hexagonal}
The sharp threshold for property \emph{(T)} in the hexagonal model equals $\frac{1}{3}$.
\end{maintheo}

To prove Theorem \ref{MainTheorem:Sharp_hexagonal} we need to show that w.o.p. random groups in the hexagonal model have property (T) for densities $ > \frac{1}{3}$ and do not have property (T) for densities $< \frac{1}{3}$. The statement for densities $> \frac{1}{3}$ can be quickly obtained by the previous results in the field. We present the proof in Section \ref{Section:T_in_hex_positive}.

The result for densities $<\frac{1}{3}$ is obtained in Section \ref{Section:Lack_of_T_hex} by using results from Sections \ref{Section:Tree_of_loops}, \ref{Section:Trees_hexagonal} and \ref{Section:Bent_walls_hexagonal}. Also, at the beginning of Section \ref{Section:Trees_hexagonal} we explain how the value $\frac{1}{3}$ arises naturally as a candidate for this sharp threshold in the hexagonal model. 

The main idea behind the proof of lack of property (T) is to use an action of a random group on a space with walls to find a non-trivial action of this group on a CAT(0) cube complex. In the context of random groups, the general idea of constructing such actions by building walls comes from the works of Ollivier and Wise \cite{ow11} and was applied with some modifications by Przytycki and Mackay in \cite{przyt}. The main difficulty lies in finding an appropriate system of walls, which is suited to the geometry of the Cayley complex of a random group. 

The second main result of this paper is the following theorem.

\begin{maintheo}\label{MainTheorem:T_square}
In the square model at densities $<\frac{3}{8}$ a random group w.o.p. does not have property \emph{(T)}.
\end{maintheo}

The proof of Theorem \ref{MainTheorem:T_square} can be found in Section \ref{Section:Bent_walls_square}. To show Theorem \ref{MainTheorem:T_square} we use the same techniques (with some small adjustments) as in the proof of Theorem \ref{MainTheorem:Sharp_hexagonal}. We also explain, at the beginning of Section \ref{Section:Bent_walls_square}, why the density $\frac{3}{8}$ is the natural limit for our methods.    

The motivation for considering $k$-gonal models for some small values of $k$, comes from the following problem.

\begin{prob}\label{Problem:Sharp_T_Gromov}
Investigate for which densities Property \emph{(T)} holds w.o.p. in the Gromov model \emph{(see \cite{gro93} for definition, and \cite{zuk03, kot} for further discussion)}.
\end{prob}

This paper can be seen as a small step towards the solution of Problem \ref{Problem:Sharp_T_Gromov}. The $k$-gonal model seems to be similar to the Gromov model, if the value of $k$ is large enough. Therefore, to find the solution to Problem \ref{Problem:Sharp_T_Gromov} it may be helpful to first understand how the sharp threshold for property (T) behaves asymptotically with $k$ tending to infinity. This paper may be considered as the beginning of this approach, that is, we provide a new sharp threshold for $k=6$, and  a new estimation for $k=4$. 

In Figure \ref{Figure:General_view} we present some part of what is already known about the values of the sharp threshold $d_{\text{(T)}}$ for property (T) in $k$-gonal models. We include our results from this paper. Note that $d_{\text{(T)}}$ is not a monotonic function of $k$ because of Theorem \ref{MainTheorem:Sharp_hexagonal} and Theorem \ref{MainTheorem:T_square}. 

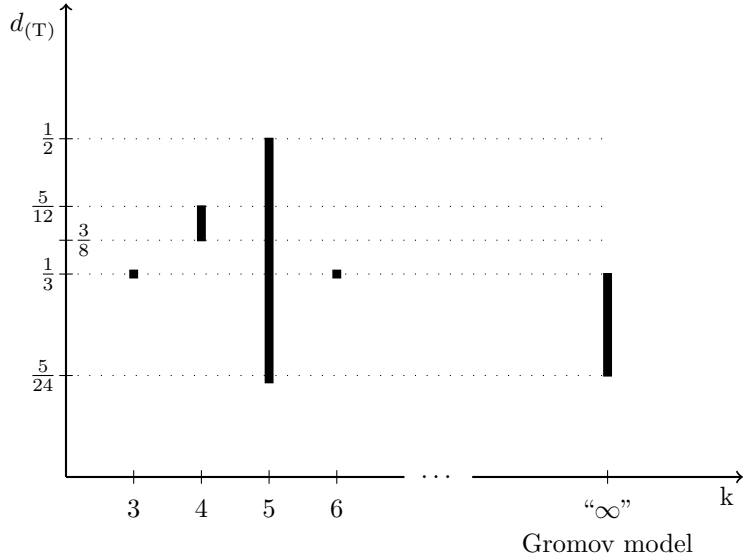
\begin{figure}[h]
\centering
\begin{tikzpicture}[scale=0.9] 
  
 \draw[thick] (0,0) -- (5,0);
 \draw[->, thick] (6,0) -- (10,0);
 \node [below left] at (10,0) {k};
 \draw[->, thick] (0,0) -- (0,7);
 \node [below left] at (0,7) {$d_{\text{(T)}}$};

 \draw (-0.1,5) -- (0.1,5); \node [left] at (0,5) {$\frac{1}{2}$};
 \draw (-0.1,4) -- (0.1,4); \node [left] at (0,4) {$\frac{5}{12}$};
   \draw (-0.1,3.5) -- (0.1,3.5); \node [right] at (0,3.5) {$\frac{3}{8}$};
  \draw (-0.1,3) -- (0.1,3); \node [left] at (0,3) {$\frac{1}{3}$};
   \draw (-0.1,1.5) -- (0.1,1.5); \node [left] at (0,1.5) {$\frac{5}{24}$};

 \draw (1,-0.1) -- (1,0.1); \node [below] at (1,-0.2) {3};
 \draw (2,-0.1) -- (2,0.1); \node [below] at (2,-0.2) {4};
 \draw (3,-0.1) -- (3,0.1); \node [below] at (3,-0.2) {5};`
 \draw (4,-0.1) -- (4,0.1); \node [below] at (4,-0.2) {6};
 \node at (5.5,0) {$\dots$};
 \draw (8,-0.1) -- (8,0.1); \node [below] at (8,-0.2) {``$\infty$''}; \node [below] at (8,-0.7) {Gromov model};
 
  \draw [loosely dotted] (0,4) -- (8,4);
  \draw [loosely dotted] (0,3.5) -- (8,3.5);
  \draw [loosely dotted] (0,3) -- (8,3);
  \draw [loosely dotted] (0,1.5) -- (8,1.5);
  \draw [loosely dotted] (0,5) -- (8,5);

\draw [fill=black, style={thick}] (0.95,2.95) rectangle (1.05,3.05);
 \draw [fill=black, style={thick}] (1.95,3.5) rectangle (2.05,4);
 \draw [fill=black, style={thick}] (2.95,1.4) rectangle (3.05,5);
\draw [fill=black, style={thick}] (3.95,2.95) rectangle (4.05,3.05);

  \draw [fill=black, style={thick}] (7.95,1.5) rectangle (8.05,3);

\end{tikzpicture}
\caption{Possible values of sharp threshold $d_{\text{(T)}}$ for property (T) in the $k$-gonal models. Thick segments mark the possible value of the threshold if the exact value is not known.}
\label{Figure:General_view}
\end{figure}

\subsection*{Acknowledgements}
I am very thankful to Piotr Przytycki who was my Assistant Doctoral Supervisor for many mathematical consultations and discussions that led to the results contained in this paper.

\section{Non-planar Isoperimetric Inequality}\label{Chapter:Nonplanar_isoperimetric}

The goal of this section is to generalize the well known ``Isoperimetric Inequality'' for random groups to some class of non-planar 2-dimensional complexes. First, let us recall the original statement:

\begin{theo}[{\cite[Theorem 2]{some}}]\label{Theorem:Isoperimetric_inequality} For any $\varepsilon > 0$, in the Gromov model at density $d < \frac{1}{2}$, with overwhelming probability all reduced van Kampen diagrams associated to the group presentation satisfy

\begin{equation} \label{Equation:Isoperimetric_inequality}
|\partial D| \geq l(1 - 2d -  \varepsilon) |D|.
\end{equation}
Here $\partial D$ denotes the set of boundary edges of the diagram $D$ and $|D|$ denotes the number of 2-cells of $D$.
\end{theo}

A corollary of Theorem \ref{Theorem:Isoperimetric_inequality} is that in the Gromov model at densities $<\frac{1}{2}$ a random group is w.o.p. hyperbolic (see \cite[Theorem 1]{shape} and \cite{gro93}). 

Let us now introduce some definitions.

\begin{defi}
Suppose $Y$ is a finite $2$-complex, not necessarily a disc diagram.
\begin{itemize}
\item The \textit{generalized boundary length} of $Y$, denoted $|\tilde{\partial} Y|$, is $$ |\tilde{\partial} Y| := \sum_{e\in Y^{(1)}} (2 - \deg(e)),$$
\item and the \textit{cancellation} in $Y$ is
 $$\mathrm{Cancel}(Y) := \sum_{e\in Y^{(1)}} (\deg(e)-1),$$
\end{itemize}
where $\mathrm{deg(e)}$ is the number of times that $e$ appears as the image of an edge of the attaching map of a $2$–cell of $Y$.
\end{defi}

We denote the generalized boundary length of a complex $Y$ by $|\parti Y|$ to be consistent with the notation in the literature concerning planar diagrams. Note that for every planar diagram $Y$ we have $|\tilde{\partial} Y| = |\partial Y|$. Moreover, for $k$-gonal diagrams (i.e. such that each of its faces is a $k$-gon) we have $\mathrm{Cancel}(Y)=\frac{1}{2}(k|Y|-|\parti Y|)$. 

\begin{defi}
We say that a finite 2-complex $Y$ is \emph{fulfilled} by a set of relators $R$ if there is a combinatorial map from~$Y$ to the presentation complex $\cay /G$ that is locally injective around edges (but not necessarily around vertices).
\end{defi}

In particular, any subcomplex of the Cayley complex $\cay$ is fulfilled by~$R$.

\begin{defi}
Let $D$ be a disc diagram and let $\gamma$ be an injected edge-path with ends on the boundary of $D$. We say that $\gamma$ is \textit{partitioning $D$ into two connected components} if there exist two planar closed connected diagrams $D'$ and $D''$ such that $D' \cup D'' = D$ and $D' \cap D'' \subseteq \gamma$. 

Let $Y$ be a finite 2-complex. We say that a pair of connected complexes $Y'$ and $Y''$ is a \textit{partition of $Y$} if $Y' \cup Y'' = Y$ and $Y' \cap Y''$ contains no 2-cells. 
\end{defi}

\begin{defi}[diagram with $K$-small hull]\label{Definition:Diagram_with_small_hull} For $K > 0$ let $Y$ be a 2–complex that is a union of a disc diagram $Z$ and not necessarily connected complex $H$ (called \textit{hull of $Y$}), such that $Y$ satisfies the following condition (``easy cutting condition''):

\begin{itemize}
\item Let $\gamma$ be an edge-path in $Z$ with two endpoints on the boundary of $Z$ that is partitioning $Z$ into two connected components: $Z'$ and $Z''$. Then, there exists a partition of $Y$ into two components: $Y'$ and $Y''$ such that $Z' \subseteq Y'$, $Z'' \subseteq Y''$, and moreover: 

\begin{equation}\label{Equation:Hull_definition}
|\tilde{\partial} Y| \geq |\tilde{\partial} Y'| + |\tilde{\partial} Y''| - K|\gamma| - K.
\end{equation}
\end{itemize}
We call such $Y$ a \textit{diagram with $K$-small hull}. The disc diagram $Z$ is called the \textit{disc basis of $Y$}.
\end{defi}

Equation (\ref{Equation:Hull_definition}) can be interpreted in the following way: if we cut a disc basis of a diagram with $K$-small hulls along an edge-path $\gamma$, we need to perform at most $K|\gamma| + K$ extra cuttings to split the diagram into two connected components. Note that a disc diagram is a diagram with 2-small hull. 

Until the end of this section, let $G$ denote a random group in the $k$-gonal model with the presentation $\left<S | R \right>$ and let $\cay$ denote the Cayley complex of $G$ with respect to this presentation.  

\begin{theo}[Generalized Isoperimetric Inequality]\label{Theorem:Generalized_Isoperimetric_Inequality}
In the $k$-gonal model at density $d < \frac{1}{2}$,  for each $K$ and $\varepsilon > 0$, the following statement holds w.o.p.: there is no diagram $Y$ with $K$-small hull, that is fullfilable by $R$ and satisfies:

\begin{equation}
\emph{Cancel}(Y) > k(d + \frac{\varepsilon}{2})|Y|,
\end{equation}
or equivalently all diagrams $Y$ with $K$-small hull and fullfilable by $R$ satisfy
\begin{equation}
|\tilde{\partial} Y| > k(1 - 2d - \varepsilon)|Y|.
\end{equation}

\end{theo}

Our strategy to show Theorem \ref{Theorem:Generalized_Isoperimetric_Inequality} is to first recall a ``local'' version of it, i.e. with the additional limit on the number of 2-cells in a diagram, and then to show that this locality assumption can be omitted. 

\begin{lem}[local version of the Generalized Isoperimetric Inequality]\label{Lemma:Local_generalized_isoperimetric}
In the $k$-gonal model at density $d < \frac{1}{2}$, for each $K, \varepsilon >0$ w.o.p. there is no 2-complex $Y$ with $|Y| \leq K$ fulfilled by $R$ and satisfying
\begin{equation}\label{Inequality:Local_diagrams}
\emph{Cancel}(Y) > k(d + \varepsilon)|Y|.
\end{equation}
\end{lem}

The proof of Lemma \ref{Lemma:Local_generalized_isoperimetric} for a special case of $k=4$ can be found in \cite[Section 2.1]{cubu}. General proof of Lemma \ref{Lemma:Local_generalized_isoperimetric} can be obtained by repeating all steps in the proof of \cite[Lemma 2.7]{cubu}, but with the length of relators equal $k$. Hence, we will not repeat the argument here.

We start the proof of Theorem \ref{Theorem:Generalized_Isoperimetric_Inequality} with reformulating \cite[Lemma 11]{some} by replacing the length of relator $l$ by $k$ and simplifying some constants.

\begin{lemma}[reformulation of {\cite[Lemma 11]{some}}]\label{cut}
Let $G = \left<S | R \right>$ be a finite presentation in which all elements of $R$ have length k. Suppose that for some constant $C'>0$ every van Kampen diagram $D$ of this presentation satisfies:

$$ |\partial D| \geq k C' |D|.$$

Then every van Kampen diagram $D$ can be partitioned into two diagrams $D'$, $D''$ by cutting it along a path of length at most $2 k \frac{\log(|D|)}{C'}$ with endpoints on the boundary of $D$ such that each of $D'$ and $D''$ contains at least one quarter of the boundary of $D$.
\end{lemma}

We will now prove two propositions ``approximating'' Theorem \ref{Theorem:Generalized_Isoperimetric_Inequality}.

\begin{prop}\label{p1}
Let $G = \left<S | R \right>$ be a finite presentation such that all elements of $R$ are reduced words of length $k$. Suppose that for some constant $C'$ all van Kampen diagrams $D$ with respect to this presentation satisfy

$$|\partial D| \geq C' k |D|.$$

Choose any $K, \varepsilon > 0$. Take $A$ large enough to satisfy $K ( \frac{2k}{C'} \log (\frac{7 A}{6 C'}) + 1) < k \varepsilon A$ and $\frac{1}{C'}\log(\frac{7 A}{6 C'}) < \frac{1}{16} A$. Suppose that for some $C >0$ all diagrams $Y$ with $K$-small hull, the disc basis of which has the boundary length at most $kA$, satisfy:

$$|\parti Y| \geq C k |Y|.$$

Then all diagrams $Y$ with $K$-small hull, the disc basis of which has the boundary length at most $\frac{7}{6} k A$, satisfy:

$$|\parti Y| \geq (C - \varepsilon) k|Y|.$$
\end{prop}

\begin{proof}
Let $Y$ be the diagram with $K$-small hull such that its disc basis $Z$ has the boundary length between $k A$ and $\frac{7}{6} k A $. By Lemma \ref{cut} we can perform a partition of $Z$ along a path $\gamma$ of length at most $\frac{2 k}{C'} \log (|Z|)$ into two reduced disc diagrams $Z'$ and $Z''$ such that $|\partial Z'| > \frac{1}{4} |\partial Z|$ and $|\partial Z''| > \frac{1}{4} |\partial Z|$.

Then $|\partial Z'| \leq \frac{3}{4} |\partial Z| + |\gamma| \leq \frac{3}{4} |\partial Z| + \frac{2k}{C'}\log(|Z|) \leq \frac{3}{4} |\partial Z| + \frac{2k}{C'}\log(\frac{|\partial Z|}{k C'}) \leq \frac{7}{8}k A + \frac{2k}{C'}\log(\frac{7A}{6C'}) \leq  k A$. Analogously, $|\partial Z''| \leq k A$. Since $Y$ is a diagram with $K$-small hull, we can perform a partition of $Y$ into two diagrams $Y'$ and $Y''$ with $K$-small hull, such that $|\parti Y| \geq |\parti Y'| + |\parti Y''| - K (\frac{2k}{C'} \log (|Z|) + 1)$.

We know that $|\partial Z'|, |\partial Z''| \leq kA$, so by our assumption we obtain:

$$|\parti Y'| \geq C k |Y'|$$
$$|\parti Y''| \geq C k |Y''|.$$

Moreover, by the assumption on van Kampen diagrams, we know that $|Z| \leq \frac{|\parti Z|}{C' k}$, so $\frac{1}{C'} \log (|Z|) < \frac{1}{C'} \log (\frac{7 A}{6 C'})$. Therefore,

$$|\parti Y| \geq |\parti Y'| + |\parti Y''| - K (\frac{2k}{C'} \log (|Z|) +  1) \geq $$
$$\geq C k (|Y'| + |Y''|) - K (\frac{2k}{C'}  \log (\frac{7 A}{6 C'}) +  1).$$

We have chosen $A$ large enough so that $K ( \frac{2k}{C'} \log (\frac{7 A}{6 C'}) + 1) < k\varepsilon A$, thus, we can continue estimation

$$|\parti Y| \geq k C |Y| - k \varepsilon A.$$

Observe that $|\partial Z| \leq k|Z| \leq k|Y|$. Additionally, by the assumption on $Y$, we know that $|\partial Z| \geq k A$, so $k A \leq |\partial Z| \leq k |Y|$, thus $A \leq |Y|$. 

Therefore, we can eventually estimate:
$$|\parti Y| \geq k C |Y| - k \varepsilon A  \geq k(C - \varepsilon)|Y|.$$
\end{proof}
The last approximation to Theorem \ref{Theorem:Generalized_Isoperimetric_Inequality} is the following

\begin{prop}\label{p2}
Let $G = \left<S | R \right>$ be a finite presentation such that all elements of $R$ are reduced words of length $k$. Suppose that for some constant $C'$ all van Kampen diagrams $D$ with respect to this presentation satisfy

\begin{equation}\label{Equation:All_van_Kampen}
|\partial D| \geq k C' |D|.
\end{equation}

Choose any $K, \varepsilon > 0$. Take $A$ large enough to satisfy $K ( \frac{2k}{C'} \log (\frac{7 A}{6 C'}) + 1) < k \varepsilon A$ and $\frac{1}{C'}\log(\frac{7 A}{6 C'}) < \frac{1}{16} A$. Suppose that for some $C >0$ all diagrams $Y$ with $K$-small hull, the disc basis of which has the boundary length at most $k A$, satisfy:

\begin{equation}\label{Equation:All_small_hull}
|\parti Y| \geq k C |Y|.
\end{equation}

Then all diagrams $Y$ with $K$-small hull satisfy:

$$|\parti Y| \geq k(C - \varepsilon)|Y|.$$
\end{prop}

\begin{proof}
The assumptions of this proposition and Proposition \ref{p1} are the same. Hence by the statement of Proposition \ref{p1} we conclude that the assumptions of Proposition \ref{p1} are fulfilled with the new parameters: $A_1 = \frac{7}{6}A, \varepsilon_1 = \varepsilon (\frac{6}{7})^{\frac{1}{2}}$ and $C_1 = C - \varepsilon$ instead of $A, \varepsilon, C$ and with the same $C'$ (these new parameters indeed satisfy $k \varepsilon_1 A_1 > K ( \frac{2k}{C'} \log (\frac{7 A_1}{6 C'}) + 1)$ and $\frac{1}{C'}\log(\frac{7 A_1}{6 C'}) < \frac{1}{16} A_1$). By induction, every diagram with $K-small$ hull such that its disc basis has boundary of length at most $k A \left(\frac{7}{6} \right)^k$ satisfies

$$|\parti Y| \geq \left(C - \varepsilon \sum_{i = 0}^{k-1} \left(\frac{6}{7} \right)^{\frac{i}{2}} \right)k|Y| $$
and we end the proof by the inequality $\sum_{i=0}^{\infty} \left(\frac{6}{7} \right)^{\frac{i}{2}} < 14$.
\end{proof}

Now, we will recall the fact known as the ``local-global principle'' or \\
Gromov–Cartan–Hadamard Theorem. This theorem occurs in literature in many various formulations. The variant best suited to our context is \cite[Theorem 60]{inv}, here we present a simplified version of it: 

\begin{theo}\label{Theorem:Gromov_Cartan_Hadamard}
Let $G = \left<S | R \right>$ be a finite presentation such that all elements of $R$ are reduced words of length $k$. Let $C > 0$. Choose $\varepsilon > 0$. Suppose that every reduced van Kampen diagram with respect to this presentation $D$, having at most $10^{50} C^{-3} \varepsilon^{-2}$ faces, satisfies:

$$|\partial D| \geq C |D|.$$

Then every reduced van Kampen diagram with respect to this presentation satisfies:

$$|\partial D| \geq (C - \varepsilon) |D|.$$
\end{theo}

We can finally prove the ``Generalized Isoperimetric Inequality'':

\begin{proof}[Proof of Theorem  \ref{Theorem:Generalized_Isoperimetric_Inequality}]
Combining Lemma \ref{Lemma:Local_generalized_isoperimetric} with Theorem \ref{Theorem:Gromov_Cartan_Hadamard} we obtain that all van Kampen diagrams with respect to the presentation of random group in the $k$-gonal model w.o.p. satisfy Equation (\ref{Equation:All_van_Kampen}) for $C'=(1-2d-\varepsilon')$ and for arbitrarily small $\varepsilon'$. By Lemma \ref{Lemma:Local_generalized_isoperimetric} applied to diagrams of size $\leq A = A(C',K,\varepsilon)$ we know that for any $K, \varepsilon$ w.o.p. all diagrams with $K$-small hull satisfy Equation (\ref{Equation:All_small_hull}) with $C=(1-2d-\varepsilon)$. Hence, the assumptions of the Proposition \ref{p2} are satisfied, which gives the statement of Theorem \ref{Theorem:Generalized_Isoperimetric_Inequality}.
\end{proof}

\begin{cor}[{\cite[Theorem 2.]{angular}}]\label{Corollary:Torsionfree}
For $k \geq 3$ a random group in the $k$-gonal model at density $< \frac{1}{2}$ is w.o.p. hyperbolic, torsion-free group of geometric dimension at most 2.
\end{cor}

\begin{proof}
The proof can be found in Section 3 of \cite{angular}.
\end{proof}

Later we will use the following lemma.

\begin{lem}\label{Lemma:Pairity_of_the_boundary}
Let $\mathcal{D}$ be any 2-complex fulfilled by a set of relators of a random group in the $k$-gonal model, for even values of $k$. Then $|\tilde{\partial}(\mathcal{D})|$ (the generalized boundary length) is an even number.
\end{lem}

\begin{proof}
The statement results from the fact that every 2-cell of $\mathcal{D}$ has even number of edges and every identification of two 1-cells reduces the generalized boundary length by an even number.
\end{proof}

The following definition will be very useful in the latter part of the paper.

\begin{defi}
Let $D$ be a diagram and $c$ its 2-cell. An \textit{external} edge of the 2-cell $c$ is an edge of $D$ that is adjacent only to $c$, and to no other 2-cell of $D$. We say that a 2-cell $c$ \textit{contributes at most $n$} to the generalized boundary length of a diagram $D$ if $c$ has at most $n$  external edges. 
\end{defi}

\section{Preliminaries}\label{Chapter:Preliminaries}

In this section we will introduce some useful geometrical tools: hypergraph, $n$-cycle of intra-segments, and a diagram collared by $n$-cycle of intra segments. 

Let us recall the definition of a hypergraph, introduced in \cite{ow11}. The term "hypergraph" refers to the fact that what we consider is 1-dimensional hyperplane in a square or hexagonal complex. Hypergraph is also a graph theory term meaning a graph with multi-edges, but we do not consider these objects in this paper. 
 
\begin{defi}[based on {\cite[Definition 2.1]{ow11}}]\label{Definition:Hypergraph}
Let $\cay$ be the Cayley complex of a random group $G$ in the $k$-gonal model and suppose that $k$ is even. We define a graph $\Gamma$ in the following way: let the vertices of $\Gamma$ be the set $V$ of midpoints of edges of $\cay$. We join $x,y \in V$ with an edge iff $x$ and $y$ correspond to the antipodal midpoints of edges of a 2-cell of $\cay$ (if there are many such 2-cells we add as many edges between $x$ and $y$). A connected component of $\Gamma$ is called a \textit{hypergraph}.

There is a natural map $\varphi : \Gamma \to \cay$ that sends each vertex of $\Gamma$ to the corresponding midpoint of an edge of $\cay$ and each edge of $\Gamma$ to an injected path joining antipodal points of an appropriate 2-cell (we assume that the interior of this path is contained in only one 2-cell of $\cay$). We can choose $\varphi$ in a way that images of all edges of $\Gamma$ joining antipodal midpoints of a given 2-cell $c$ intersect at a single point $x_c$ that lies in the interior of $c$. Such point $x_c$ will be called the \textit{middle of $c$}. Moreover, suppose that the set of middles of 2-cells of $\cay$ is $G$-invariant. 

We define a \textit{hypergraph segment} or a \textit{segment of a hypergraph} to be an immersed finite path in a hypergraph. 

We say that two edges of a hypergraph \textit{intersect} if their images under the natural map $\varphi$ intersect.

We say that an edge $e$ of a hypergraph is \textit{contained} in a 2-cell $c$ of $\cay$ iff $c$ intersects with the interior of $\varphi(e)$.

For a subgraph $\Gamma'$ of $\Gamma$ we define a 2-complex, called the \textit{unfolded carrier} $A$ of $\Gamma'$ in the following way: for every edge $e$ of $\Gamma'$ let $c_e$ be the 2-cell of $\cay$ containing $e$. For every 2-cell $c_e$ we consider the isomorphic copy $c'_e$ of $c_e$. Now, we take the disjoint union of 2-cells $c'_e$ for all $e$ in $\Gamma'$ and we glue them as follows: if two edges $e_1$ and $e_2$ share a common endpoint $v$, then we identify $c_{e_1}'$ and $c_{e_2}'$ along the 1-cell corresponding to the vertex $v$. 

A \textit{ladder} is the unfolded carrier of a segment. The \textit{carrier} of a hypergraph $\Lambda$ is the subcomplex of $\cay$ consisting of all 2-cells containing edges of $\Lambda$.    
\end{defi}

Later we will introduce modified hypergraphs (that do not always join antipodal points of 2-cells), so we will sometimes refer to Definition \ref{Definition:Hypergraph} as to the definition of a \textit{standard hypergraph}.  

\begin{defi}[intra-segment]\label{Definition:Intra_segment}
Let $\cay$ be the Cayley complex of a random group $G$ in the $k$-gonal model, and suppose that $k$ is even. Let $V$ be the set of midpoints of 1-cells of $\cay$ and let $V_{mid}$ be the set of middles of 2-cells of $\cay$.

For $n \geq 0$ let $\gamma = (s, x_1, x_2, \dots, x_n, e)$ be a sequence of elements of $V_{mid} \cup V$ such that: $s$ and $e$ are the middles of some 2-cells $c_s$ and $c_e$ respectively, $x_i \in V$ for $1 \leq i \leq n$, $x_{1}$ belongs to the boundary of $c_s$, $x_n$ belongs to the boundary of $c_e$  and for every $1 \leq i \leq n-1$ the points $x_i$ and $x_{i+1}$ are antipodal midpoints of some 2-cell of $\cay$ (we allow $s=e$, thus $c_s = c_e$). 

Let $\lambda_{int}$ be a graph defined as follows: the set of vertices of $\lambda_{int}$ is the set of elements of $\gamma$, and we join two vertices $x,y$ of $\lambda_{int}$ by an edge if they are two consecutive elements of $\gamma$. Such $\lambda_{int}$ is called an \textit{intra-segment}. 

There is a natural map $\varphi: \lambda_{int} \to \cay$ sending each vertex of $\lambda_{int}$ to the corresponding point in $\cay$ and each edge $e=\{v_1, v_2\}$ of $\lambda_{int}$ to an edge joining $\varphi(v_1)$ with $\varphi(v_2)$. We choose $\varphi$ in a way that for every edge $e$ of $\lambda_{int}$ with  both ends in $V$ the image of $e$ under $\varphi$ is the same as the image of $e$ under the natural map introduced in Definition \ref{Definition:Hypergraph}.

The map $\varphi$ extends, in a natural way, to a map from a finite union of intra-segment to $\cay$.

We say that an edge $e$ of an intra-segment is \textit{contained} in a 2-cell $c$ if $\varphi(e)$ intersects with the interior of $c$. 
\end{defi}

In other words: an intra-segment $\lambda$ is a path in $\cay$ such that its first and last edge join the middle of a 2-cell with its boundary, and all other edges of $\lambda$ join antipodal points of 2-cells of $\cay$. 

\begin{defi}
Let $\lambda_1$ and $\lambda_2$ be two intra-segments. Suppose that an extreme vertex $v$ of $\lambda_1$ (that is the first or the last vertex of $\lambda_1$) is equal to an extreme vertex of $\lambda_2$. Let $c$ be a 2-cell of $\cay$ whose middle is $v$. We say that the intra-segments $\lambda_1$ and $\lambda_2$ \textit{prolong each other} (or \textit{one of them prolongs the other one}) if there exist a vertex $x$ of $\lambda_1$ adjacent to $v$ and a vertex $y$ of $\lambda_2$ adjacent to $v$ such that $x$ and $y$ are antipodal midpoints of edges of $c$.
\end{defi}

Informally, $\lambda_1$ and $\lambda_2$ prolong each other if some two extreme edges of $\lambda_1$ and $\lambda_2$ ``glue up'' to an edge joining antipodal midpoints of a 2-cell of $\cay$. For our later purpose, we need also a procedure of ``cropping'' a hypergraph segment to an intra-segment. 

\begin{defi}
Let $\lambda$ be a hypergraph segment of length at least two and let $(x_1, x_2, \dots, x_n)$ be its consecutive vertices. Let $c_1$ and $c_n$ be the 2-cells containing the first and the last edge of $\lambda$ respectively (it may happen that $c_1 = c_n$). Denote by $m_1$ and $m_n$ the middles of $c_1$ and $c_n$ respectively. We define the \textit{intra-segment of }$\lambda$ to be the intra-segment with the following sequence of vertices $(m_1, x_2, \dots, x_{n-1}, m_n)$. 
\end{defi}

\subsection{$n$-cycles and diagrams collared by $n$-cycles}\label{Section:n_cycles}

Recall that $\varphi$ denotes the map sending hypergraphs, intra-segments or unions of intra-segments into the Cayley complex of a random group (see Definition \ref{Definition:Hypergraph} and Definition \ref{Definition:Intra_segment}).

\begin{defi}[$n$-cycle of intra-segments]\label{Definition:n_cycle}
We define a \textit{$1$-cycle of intra-segments} to be an intra-segment $\lambda$ such that $\varphi(\lambda)$ is an injected circle into $\cay$.

For $n \geq 2$ a \textit{$n$-cycle of intra-segments
} is a sequence of intra-segments: $(\lambda_1, \lambda_2, \dots \lambda_n )$ such that: 
\begin{itemize}
\item for $1 \leq i \leq n-1$ the last vertex of $\lambda_i$ coincides with the first vertex of $\lambda_{i+1}$ and the last vertex of $\lambda_n$ coincides with the first vertex of $\lambda_1$,

\item for $1 \leq i \leq n-1$ the intra-segment $\lambda_{i+1}$ does not prolong $\lambda_i$ and $\lambda_1$ does not prolong $\lambda_n$,
\item 
$\varphi(\lambda_1 \cup \lambda_2 \cup \dots \cup \lambda_n)$ is an injected circle into $\cay$.
\end{itemize} 

Let us denote by $\mathcal{C}$ the $n$-cycle of intra-segments defined above. If the value of $n$ is unknown (or not important) we call it a \textit{multi-cycle of intra-segments}. We say that an edge or a vertex \textit{belongs to $\mathcal{C}$} if it belongs to one of intra-segments $\lambda_1, \lambda_2, \dots, \lambda_n$.

We define the \textit{ladder of a $n$-cycle of intra-segments $\mathcal{C}$} in the following way: for each 2-cell $c$ of $\cay$ that contains an edge $e$  of $\mathcal{C}$ we consider the isomorphic copy $c'$ of $c$. We take the disjoint union of these 2-cells and we glue them in the following way: if two of these 2-cells $c_1$, $c_2$ contain the same vertex of $\mathcal{C}$, as the midpoint of their boundary edges: $e_1 \subset c_1$ and $e_2 \subset c_2$, we then identify $e_1$ with $e_2$. 

We say that a vertex $v$ of $\mathcal{C}$ is \textit{contained} in a 2-cell $c'$ of the ladder of $\mathcal{C}$ if $c'$ is an isomorphic copy of a 2-cell $\tilde{c}$ of $\cay$ such that $v$ lies in the middle of $\tilde{c}$.

For $n \geq 2$ a \textit{weld} is a vertex of $n$-cycle of intra-segments that is both the first vertex of one intra-segment and the last vertex of another intra-segment. For $n=1$ we define a \textit{weld} to be the first vertex of $1$-cycle of intra-segments.
\end{defi}

Definition \ref{Definition:n_cycle} is illustrated in Figure \ref{Figure:n_cycles}.

\begin{figure}[h]
\centering
\begin{tikzpicture}[scale=0.3]

\tikzset{
  hexagon/.style={signal,signal to=east and west}
}

\draw  plot[smooth, tension=.7] coordinates {(-20,23) (-12,23) (-10,18) (-15,14) (-20,15) (-22,19) (-21,22) (-20,23)};

\node[regular polygon, regular polygon sides=6, shape aspect=0.5, minimum width=0.6cm, minimum height=0.3cm,   draw=black!50!black,fill=gray,fill opacity=0.25] at (-4.9,22.4) {};

\draw  plot[smooth, tension=.7] coordinates {(-4.9,22.4) (-2.9,23.4) (2.1,22.4) (4.1,19.4) (3.1,15.4) (-0.9,13.4) (-5.9,15.4) (-5.9,19.4) (-3.9,21.4) (-4.9,22.4)};

\draw  plot[smooth, tension=.7] coordinates {(8.1,16.4) (8.6,18.2) (12.1,18.4) (14.1,21.1) (16.1,22.4)};
\node[regular polygon, regular polygon sides=6, shape aspect=0.5, minimum width=0.6cm, minimum height=0.3cm,   draw=black!50!black,fill=gray,fill opacity=0.25] (v1) at (16.1,22.4) {};

\draw  plot[smooth, tension=.7] coordinates {(v1) (18.1,21.4) (19.7,18.9) (18.1,17.4)};
\node[regular polygon, regular polygon sides=6, shape aspect=0.5, minimum width=0.6cm, minimum height=0.3cm,   draw=black!50!black,fill=gray,fill opacity=0.25] (v1) at (18.1,17.4) {};

\draw  plot[smooth, tension=.7] coordinates {(v1) (19.3,16.7) (19.9,15.7) (19.1,14.4) (15.1,14.4) (11.1,15.4) (9.1,15.9) (8.1,16.4)};
\node[regular polygon, regular polygon sides=6, shape aspect=0.5, minimum width=0.6cm, minimum height=0.3cm,   draw=black!50!black,fill=gray,fill opacity=0.25] (v1) at (8.1,16.4) {};

\end{tikzpicture}
\caption{Two different 1-cycles and one 3-cycle of hypergraph intra-segments presented in the context of the hexagonal model.}
\label{Figure:n_cycles}
\end{figure}
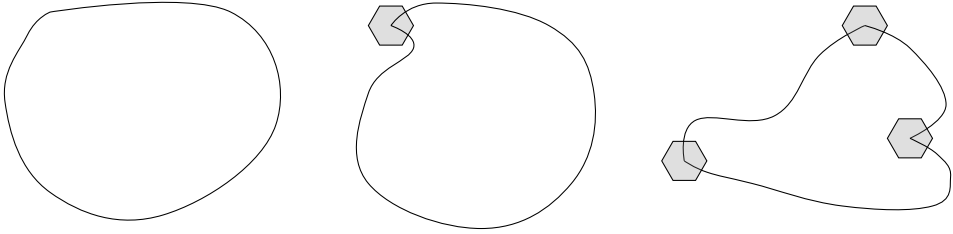

Now we will show how to construct a diagram, that is  ``filling'' a given intra-segment to a disc diagram. 

\begin{defi}\label{Definition:Collared_n_cycle_unreduced}

Let $\mathcal{C} = (\lambda_1, \lambda_2, \dots, \lambda_n)$ be a $n$-cycle of intra-segments for some $n \geq 1$. Let $L$ be the ladder of $\mathcal{C}$ and let $P$ be an edge-path $P \to L$ that represents a generator of $\pi_1(L)$. Let $A$ be a disc diagram with the boundary path $P$.

A \textit{diagram collared by $\mathcal{C}$} is the complex $\mathcal{D}$ obtained as a union $\mathcal{D}:=L \cup_{P} A$. A 2-cell of $L$ containing a weld is called a \textit{corner (of $\mathcal{D}$)}. We say that a corner of $\mathcal{D}$ is a \textit{shell corner}, if at least half of its edges are external edges (that is edges adjacent to only one 2-cell). The disc diagram $A$ is called the \textit{disc basis of} $\mathcal{D}$. 

It may happen that $A$ has no 2-cells or no edges (for example, if $L$ is a disc diagram itself).

\end{defi}

Definition \ref{Definition:Collared_n_cycle_unreduced} is illustrated in Figure \ref{Figure:Diagram_collared_by_n_cycle}.

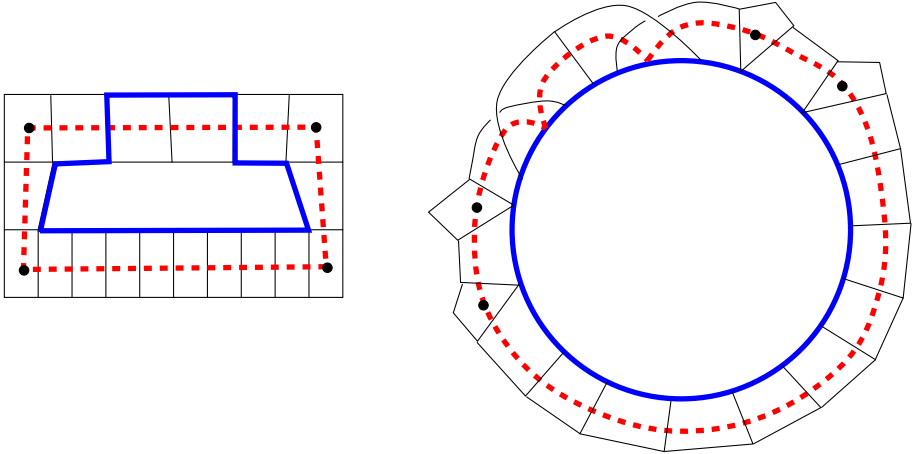
\begin{figure}[h]
\centering
\begin{tikzpicture}[scale=0.9]

\draw [blue, draw,line width=2pt] (-11,16) ellipse (2.5 and 2.5);
\draw [red, draw,line width=2pt, dashed] plot[smooth, tension=.7] coordinates {(-13,17.5) (-13.0257,18.2168) (-12.1396,18.8585) (-11.5,18.5)};
\draw [red, draw,line width=2pt, dashed] plot[smooth, tension=.7] coordinates {(-11.5,18.5) (-11,19) (-10.2323,18.9926) (-9.3919,18.6403) (-8.5,18) (-8.0562,16.639) (-8.0806,15.0643) (-8.8416,13.8037) (-10.534,13.0504) (-12.3535,13.3463) (-13.6722,14.4387) (-14.0553,15.9685) (-13.6158,17.4524) (-13,17.5)};
\draw (-12.3011,18.1586) node (v8) {} -- (-12.8702,18.9258) node (v7) {};
\draw  plot[smooth, tension=.7] coordinates {(v7) (-11.807,19.3097) (-10.7085,18.5067)};
\draw  plot[smooth, tension=.7] coordinates {(-12.8643,18.934) (-13.7187,17.9103) (-13.3455,16.7463)};
\draw  plot[smooth, tension=.7] coordinates {(v8)};

\draw  plot[smooth, tension=.7] coordinates {(-11.3453,19.1509) (-10.7765,19.3634) (-10.1431,19.2634)};
\draw (-10.1445,19.2628) node (v15) {} -- (-10.1212,18.3294) node (v14) {};
\draw  plot[smooth, tension=.7] coordinates {(-11.9403,18.3175) (-11.9482,18.708) (-11.7131,18.9632) (-11.5307,19.0892)};
\draw  plot[smooth, tension=.7] coordinates {(-12.7056,17.8332) (-12.9925,17.905) (-13.5486,17.8153) (-13.6801,17.7316)};
\draw  plot[smooth, tension=.7] coordinates {(-13.8176,17.624) (-14.003,17.3848) (-14.0867,17.0081) (-14.1345,16.7511)};
\draw (-14.1285,16.751) node (v11) {} -- (-13.4828,16.3684);

\draw (v14);
\draw (v14);
\draw (-10.1461,19.2589) -- (-9.6076,19.3608) -- (-9.3377,19.016) -- (-10.1255,18.3375);

\draw (-14.1256,16.7449) -- (-14.7335,16.2601) -- (-14.3024,15.8447) -- (-13.4702,16.3655);
\draw (-14.2909,15.8409) -- (-14.2608,15.2216) -- (-13.3914,15.1805);
\draw (-14.2322,15.1989) -- (-14.3681,14.7721) -- (-14.0107,14.3523) -- (-13.3927,15.1988);
\draw (-14.0189,14.3365) -- (-13.3018,13.5417) -- (-12.7373,14.1883);
\draw (-13.2878,13.5316) -- (-12.4969,12.9853) -- (-12.0937,13.7499);
\draw (-9.3549,18.9941) -- (-8.6848,18.4996) -- (-9.2249,17.7134);
\draw (-8.7023,18.4832) -- (-8.0709,18.473) -- (-7.9738,18.0128) -- (-9.219,17.7285);
\draw (-7.963,18) -- (-7.7617,17.1987) -- (-8.68,16.962);
\draw (-7.7625,17.2074) -- (-7.6147,16.0968) -- (-8.4776,16.0585);
\draw (-7.6055,16.0959) -- (-7.723,14.9905) -- (-8.6177,15.2819);
\draw (-7.7224,14.9919) -- (-8.1343,14.0807) -- (-8.9598,14.594);
\draw (-12.4985,12.9836) -- (-11.2535,12.7208) -- (-11.153,13.4687);
\draw (-11.2517,12.7219) -- (-9.9414,12.833) -- (-10.2473,13.6027);
\draw (-9.9357,12.8333) -- (-8.9384,13.3708) -- (-9.5085,13.989);
\draw (-8.9305,13.3708) -- (-8.1267,14.0836);
\draw (-20.2876,16.9926) -- (-20.314,17.9884);
\draw (-19.4757,16.9995) -- (-19.506,17.9895);
\draw (-18.5218,16.9916) -- (-18.5718,17.9974);
\draw (-20.5,15) -- (-20.5,16);
\draw (-20,15) -- (-20,16);
\draw (-19.5,15) -- (-19.5,16);
\draw (-19,16) -- (-19,15);
\draw (-18.5,15) -- (-18.5,16);
\draw (-18,15) -- (-18,16);
\draw (-17.5,15) -- (-17.5,16);
\draw (-17,15) -- (-17,16);
\draw (-16.5,15) -- (-16.5,16);
\draw  (-21,16) node (v1) {} rectangle (-16,15);
\draw  (-21,18) rectangle (-16,17);
\draw (-17.6198,16.9925) -- (-17.6138,17.9976);
\draw (-16.8408,16.9879) -- (-16.7914,17.9903);
\draw (v1);
\draw (-21,17) -- (-21,16);
\draw (-16,17) -- (-16,16);
\draw (-20.5,16) -- (-20.2794,16.993);
\draw (-16.8413,16.9913) -- (-16.5,16);
\draw  [red, draw,line width=2pt, dashed] (-20.6543,17.5051) -- (-16.3952,17.5167) -- (-16.245,15.4536) -- (-20.7179,15.4131) -- (-20.6543,17.5051);
\draw [blue, draw,line width=2pt] (-20.4707,15.9657) -- (-20.2462,16.9735) -- (-19.4534,17.0069) -- (-19.4868,17.9908) -- (-17.5908,17.9956) -- (-17.5908,16.9878) -- (-16.8314,16.983) -- (-16.5066,15.9943) -- (-20.4755,15.9849);
\draw  [black,fill=black] (-20.6332,17.5042) circle (0.07);
\draw  [black,fill=black] (-20.7088,15.3996) circle (0.07);
\draw  [black,fill=black] (-16.2276,15.4389) circle (0.07);
\draw  [black,fill=black] (-16.3938,17.5137) circle (0.07);

\draw  [black,fill=black] (-9.9066,18.8803) circle (0.07);
\draw  [black,fill=black] (-14.0189,16.3276) circle (0.07);
\draw  [black,fill=black] (-13.9221,14.8835) circle (0.07);
\draw  [black,fill=black] (-8.6216,18.1222) circle (0.07);

\end{tikzpicture}
\caption{Diagram collared by 4-cycle of intra-segments in the square model. On the left we present the ladder of the 4-cycle of intra-segments: 4-cycle is marked by thick dashed line and the path $P$ is marked with thick continuous line. On the right we present the diagram collared by this 4-cycle of intra-segments.}
\label{Figure:Diagram_collared_by_n_cycle}
\end{figure}

The following lemma is inspired by \cite[Lemma 3.8]{ow11}.

\begin{lem}\label{Definition:Collared_n_cycle_unreduced_exists}

For every $n$-cycle of hypergraphs $\mathcal{C}$ there exists a diagram collared by $\mathcal{C}$. 
\end{lem}

\begin{proof}
The proof of Lemma \ref{Definition:Collared_n_cycle_unreduced_exists} is identical to the proof of \cite[Lemma 3.8]{ow11}. We quote it here for completeness.

Let $L$ be the ladder of $\mathcal{C}$. Choose a simple cycle $P$ in $L$ that generates $\pi_1(L)$. By the fact that $\cay$ is simply-connected, there exists a disc $F$ with the boundary path $P$. The desired diagram $\mathcal{D}$ is constructed as a union $\mathcal{D} = A \cup_{P} L$.
\end{proof}

\begin{defi}
For a 2-complex $Y$ fullfiled by a set of relators of a random group in the $k$-gonal model a \textit{reduction pair} is a pair of adjacent 2-cells of $Y$ that are mapped onto the same 2-cell of $\cay$ under the natural combinatorial map. 
\end{defi}

Note that a diagram collared by a $n$-cycle of intra-segments may not be planar and may contain reduction pairs. Ollivier and Wise in \cite{ow11} started from \cite[Definition 3.6]{ow11}, that is similar to our Definition \ref{Definition:Collared_n_cycle_unreduced}, and they performed a procedure of removing reduction pairs and making such diagram planar. We can use the non-planar version of the Isoperimetric Inequality, so we do not need planarity, however we still need to remove all reduction pairs. 

Our way of reducing collared diagrams is slightly different from the one presented in the proof of \cite[Lemma 3.9]{ow11}.

\begin{lem}\label{Lemma:Collared_n_cycles_reduced_exists}
Let $\mathcal{D}$ be a diagram collared by an $n$-cycle of hypergraphs $\mathcal{C}$. Then there exists a reduced diagram $\mathcal{D}^*$ collared by the same $n$-cycle of hypergraphs, thus having the same set of corners as $\mathcal{D}$.
\end{lem}

\begin{proof}
Let $L$ be the ladder of $\mathcal{C}$, $A$ be the disc glued to this ladder (as in Definition \ref{Definition:Collared_n_cycle_unreduced}) and $P$ be the path along which $A$ is glued to $L$ (it means that $P$ is also the boundary path of $A$). We will now present a procedure of removing a reduction pair, which can be repeated inductively. 

A priori, there are three possible types of reduction pairs:
\begin{enumerate}[1)]
\item \label{Types_of_reduction_pairs_LL} Both 2-cells of a pair lie in $L$.
\item \label{Types_of_reduction_pairs_AA} Both 2-cells of a pair lie in $A$.
\item \label{Types_of_reduction_pairs_AL} One 2-cell of a pair lies in $L$ and the other one in $A$.
\end{enumerate}

Case \ref{Types_of_reduction_pairs_LL}). If there exists a reduction pair consisting of two 2-cells lying in $L$, it implies that this two 2-cells are mapped into the same 2-cell $c_{\cay}$ of $\cay$. Therefore, $\varphi(\mathcal{C})$ has at least one double point. It contradicts with the fact that an image of a multi-cycle of intra-segments under the natural combinatorial map $\varphi$ is an injected circle into $\cay$.

Case \ref{Types_of_reduction_pairs_AA}). If a reduction pair consists of two 2-cells lying in $A$, then we can remove it in the same way as for van Kampen diagrams: we remove both 2-cells and we glue the newly created boundary to itself to eliminate the gap.

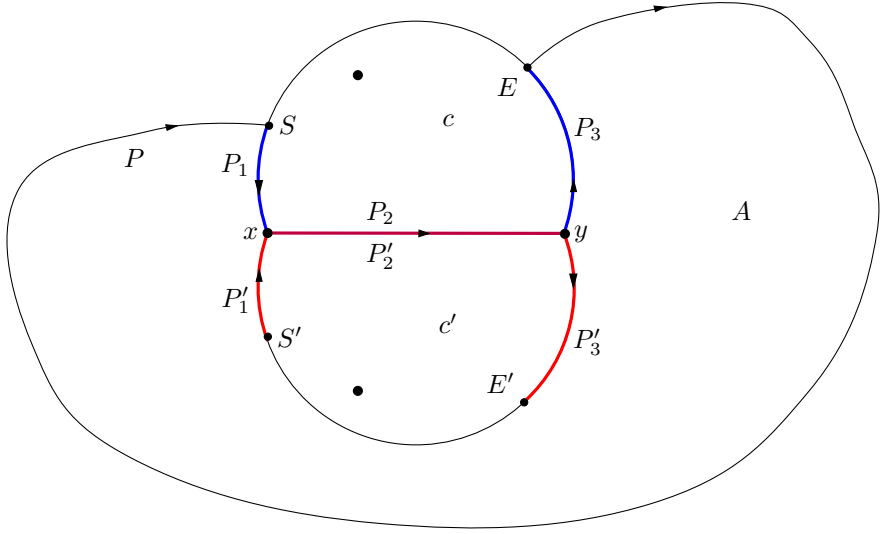
\begin{figure}[h]
\centering
\begin{tikzpicture}[scale=0.6]

\draw  (-3.5,0) arc (-159.9994:-380:3.5);
\draw [blue, draw,line width=1.2pt]  (-3.5,0) arc (-159.9994:-200:3.5);
\draw [blue, draw,line width=1.2pt] (3.0547,-0.0342) arc (-20.3647:44.6928:3.5);

\draw (-3.5,0) node (v1) {} arc (159.9994:380:3.5);
\draw  [red, draw,line width=1.2pt] (-3.5,0) node (v1) {} arc (159.9994:200:3.5);
\draw [red, draw,line width=1.2pt](3.0675,-0.0162) arc (21.0632:-47.4873:3.3873);

\draw [purple, draw,line width=1.2pt]  (-3.5,0) -- (3.0675,-0.0084);

\draw [black,fill=black] (-3.4682,2.3798) circle (0.08);
\draw [black,fill=black] (-3.4949,-2.2989) circle (0.08);
\draw [black,fill=black] (2.2573,3.6701) circle (0.08);
\draw [black,fill=black] (2.1846,-3.7545) circle (0.08);

\draw  plot[smooth, tension=.7] coordinates {(-3.454,2.3882) (-6.0553,2.2766) (-9,1) (-8.7398,-2.42) (-6.5,-5) (-1.0247,-6.4394) (5,-6) (8.5,-3.5) (10,0) (9.4628,2.4612) (8.5266,4.2375) (7.0419,5.0891) (3.8966,4.6876) (2.2514,3.6606)};

\draw [black,fill=black](-5.7563,2.4) -- (-5.5058,2.3778) -- (-5.7269,2.277) -- cycle;
\draw [black,fill=black](5.0627,5.0316) -- (5.3132,5.0094) -- (5.0921,4.9086) -- cycle;
\draw [black,fill=black](-0.1451,0.0647) -- (-0.1486,-0.0817) -- (0.0729,-0.0103) -- cycle;
\draw [black,fill=black](-3.7719,1.163) -- (-3.6135,1.163) -- (-3.6906,0.8584) -- cycle;
\draw [black,fill=black](3.1946,-0.9093) -- (3.353,-0.9093) -- (3.2759,-1.2139) -- cycle;
\draw [black,fill=black](-3.7532,-1.0719) -- (-3.6224,-1.0703) -- (-3.6852,-0.8176) -- cycle;
\draw [black,fill=black](3.2136,0.9175) -- (3.3444,0.9191) -- (3.2816,1.1718) -- cycle;
\draw [black,fill=black] (-1.5,3.5) circle (0.1);
\draw [black,fill=black] (-1.5,-3.5) circle (0.1);
\draw [black,fill=black] (-3.5,0) circle (0.1);
\draw [black,fill=black] (3.0876,-0.0182) circle (0.1);

\node at (7,0.5) {$A$};
\node at (0.5,2.5) {$c$};
\node at (0.5,-2) {$c'$};
\node[right] at (-3.4569,2.3922) {$S$};
\node[below left] at (2.255,3.651) {$E$};
\node [right] at (-3.4929,-2.3032) {$S'$};

\node[above left] at (2.1878,-3.7458) {$E'$};
\node [left] at (-3.4969,-0.002) {$x$};
\node [right] at (3.0876,-0.0182) {$y$};
\node[left] at (-3.6876,1.4986) {$P_1$};
\node[left] at (-3.6786,-1.4856) {$P_1'$};
\node[above] at (-1,0) {$P_2$};
\node[below] at (-1,0) {$P'_2$};

\node[right] at (3.0702,2.3068) {$P_3$};
\node[right] at (3.0611,-2.4107) {$P_3'$};
\node[below right] at (-6.8953,2.0656) {$P$};

\end{tikzpicture}
\caption{The reduction pair $\{c, c' \}$ in $\mathcal{D}$, with the edge-paths: $P$, $P_{c} = P_1 \cup P_2 \cup P_3$ and $P_{c'} = P'_1 \cup P'_2 \cup P'_3$.}
\label{Figure:Reducing_paths_introduction}
\end{figure}

Case \ref{Types_of_reduction_pairs_AL}). Let $\{c, c'\}$ be a reduction pair such that $c \in L$ and $c' \in A$. Let $P_c$ be the fragment of $P$
along which $c$ is glued to the diagram $A$. Note that $P_c$ contains at least one edge as otherwise $c$ could not form a reduction pair with a 2-cell from $A$. Let $P_{c'}$ be the  edge path on the boundary of $c'$ that corresponds to $P_c$. Let $S$ and $E$ be the beginning end the end of $P_c$ respectively. Let $S'$ and $E'$ be the beginning and the end of $P_{c'}$ respectively. Let $x$ be the vertex where $P_{c}$ and $P_{c'}$ meet for the first time and let $y$ be the last common vertex of $P_{c}$ and $P_{c'}$. These definitions are illustrated in Figure \ref{Figure:Reducing_paths_introduction}.

Consider paths $P_2 \subset P_{c}$ and $P_2' \subset P_{c'}$ bounded by points $x$ and $y$ (see Figure \ref{Figure:Reducing_paths_introduction}). We will show that $P_2 = P_2'$. Note that $P_2$ and $P_2'$ are labeled by the same letters (by the fact that they correspond to each other). If $P_2 \neq P_2'$, there exists a disc diagram with the boundary word that reduces itself to the trivial word, so by making identification of its boundary edges we can transform it to a diagram with no boundary. By Theorem \ref{Theorem:Generalized_Isoperimetric_Inequality} (Generalized Isoperimetric Inequality) such diagrams w.o.p. do not exist in the $k$-gonal models for densities $< \frac{1}{2}$ and all $k \geq 3$.

Now, consider the disc diagram $A - c'$. We have two possibilities:

\begin{enumerate}[a)]
\item \label{Cases_SEAL_equal} $S$ = $S'$ and $E = E'$
\item \label{Cases_SEAL_not_equal} $S \neq S'$ or $E \neq E'$ 
\end{enumerate}

Situation \ref{Cases_SEAL_equal}). In that case we define the new disc diagram as $A^* = A - c'$.

Situation \ref{Cases_SEAL_not_equal}). If $S \neq S'$ then there exists a fragment of the boundary path of $A - c'$ that equals to the fragment of $P_{c}$ bounded by the points $S$ and $x$, call it $P_1$. Also, in that case, there exists another fragment of the boundary path of $c'$ that equals to the fragment of $P_{c'}$ bounded by the points $S'$ and $x$, call it $P_1'$. By the fact that $P_2 = P_2'$, we know that $P_1 \cup (P_1')^{-1}$ is a connected fragment of the boundary path of $A - c'$. The path $P_1 \cup (P_1')^{-1}$ reads off the trivial word, so we reduce the boundary path of $A - c'$ by identifying edges of $P_1$ with the corresponding edges of $(P_1')^{-1}$ (see Figure \ref{Figure:Diagram_A_star_creation}). 

If $E \neq E'$ we perform the analogous procedure of boundary reducing on the diagram to obtain the diagram in which $E$ is identified with $E'$. 

The resulting diagram will be called $A^*$. By the construction, $A^*$ is a disc diagram. The way in which we created diagram $A^*$ is illustrated in Figure \ref{Figure:Diagram_A_star_creation}. 

\begin{figure}[h]
\centering
\begin{tikzpicture}[scale=0.55]

\draw [blue, draw,line width=1.2pt]  (-3.5,0) arc (-159.9994:-200:3.5);
\draw [blue, draw,line width=1.2pt] (3.0547,-0.0342) arc (-20.3647:44.6928:3.5);

\draw (-3.5,0) node (v1) {} arc (159.9994:380:3.5);
\draw  [red, draw,line width=1.2pt] (-3.5,0) node (v1) {} arc (159.9994:200:3.5);
\draw [red, draw,line width=1.2pt](3.0675,-0.0162) arc (21.0632:-47.4873:3.3873);

\draw [black,fill=black] (-3.4682,2.3798) circle (0.08);
\draw [black,fill=black] (-3.4949,-2.2989) circle (0.08);
\draw [black,fill=black] (2.2573,3.6701) circle (0.08);
\draw [black,fill=black] (2.1846,-3.7545) circle (0.08);

\draw  plot[smooth, tension=.7] coordinates {(-3.454,2.3882) (-6.0553,2.2766) (-9,1) (-8.7398,-2.42) (-6.5,-5) (-1.0247,-6.4394) (5,-6) (8.5,-3.5) (10,0) (9.4628,2.4612) (8.5266,4.2375) (7.0419,5.0891) (3.8966,4.6876) (2.2514,3.6606)};

\draw [black,fill=black](-5.7563,2.4) -- (-5.5058,2.3778) -- (-5.7269,2.277) -- cycle;
\draw [black,fill=black](5.0627,5.0316) -- (5.3132,5.0094) -- (5.0921,4.9086) -- cycle;

\draw [black,fill=black](-3.7719,1.163) -- (-3.6135,1.163) -- (-3.6906,0.8584) -- cycle;
\draw [black,fill=black](3.1946,-0.9093) -- (3.353,-0.9093) -- (3.2759,-1.2139) -- cycle;
\draw [black,fill=black](-3.7532,-1.0719) -- (-3.6224,-1.0703) -- (-3.6852,-0.8176) -- cycle;
\draw [black,fill=black](3.2136,0.9175) -- (3.3444,0.9191) -- (3.2816,1.1718) -- cycle;
\draw [black,fill=black] (-3.5085,0.0095) circle (0.1);
\draw [black,fill=black] (3.0876,-0.0182) circle (0.1);

\node at (7,0.5) {$A$};

\node[right] at (-3.4569,2.3922) {$S$};
\node[below left] at (2.255,3.651) {$E$};
\node [right] at (-3.4929,-2.3032) {$S'$};

\node[above left] at (2.1878,-3.7458) {$E'$};
\node [left] at (-3.5361,-0.0143) {$x$};
\node [right] at (3.0876,-0.0182) {$y$};
\node[left] at (-3.6876,1.4986) {$P_1$};
\node[left] at (-3.6786,-1.4856) {$P_1'$};

\node[right] at (3.0702,2.3068) {$P_3$};
\node[right] at (3.0611,-2.4107) {$P_3'$};
\node[below right] at (-6.8953,2.0656) {$P$};
\draw (-3.4516,-0.0341) ;

\begin{scope}[shift={(0,-12)}]

\draw [->, draw,line width=1.2pt] (0,4) -- (0,2);

\draw  node (v1) {} arc (123.5445:2.7156:4.6752);

\draw [black,fill=black] (-1.5,0) circle (0.08);
\draw [black,fill=black] (4.0929,-0.0077) circle (0.08);

\draw  plot[smooth, tension=.7] coordinates {(-1.5,0) (-4.5,2.5) (-9,1) (-8.7634,-2.1439) (-5.3281,-3.9972) (0.2402,-4.5993) (5.0096,-3.9268) (8.2948,-2.6947) (10,-0.5) (10,1.5) (9.4846,3.535) (7,5) (3.2472,3.7931) (2,0)};

\draw [black,fill=black](-5.8867,2.5082) -- (-5.6362,2.486) -- (-5.8573,2.3852) -- cycle;
\draw [black,fill=black](5.6565,4.9899) -- (5.9262,4.9534) -- (5.6859,4.8669) -- cycle;
\draw [black,fill=black] (-4.5,0) circle (0.1);
\draw [black,fill=black] (1.9886,-0.0304) circle (0.1);

\node[above] at (7,0.5) {$A^*$};

\node at (-0.4972,0.4918) {$S=S'$};
\node at (2.9981,-0.608) {$E=E'$};

\node  at (-5,0) {$x$};
\node at (4.5,0) {$y$};
\node[above] at (-3,0) {$P_1$};

\node[above] at (2.9415,0.0049) {$P_3$};
\node[below right] at (-6.9246,2.0948) {$P$};
\draw [purple, draw,line width=1.2pt](-4.5,0) -- (-1.5,0);
\draw (-1.5,0) arc (-145.867:-34.5993:2.118);
\draw [purple, draw,line width=1.2pt] (2.014,0.0147) -- (4.0745,0.0045);
\node[below] at (0.2548,-0.9314) {$P_{opp}$};
\draw [black,fill=black] (0.5,-1) -- (0.4774,-0.8573) -- (0.7599,-0.8901);

\end{scope}

\end{tikzpicture}
\caption{Creation of the diagram $A^*$.}
\label{Figure:Diagram_A_star_creation}
\end{figure}
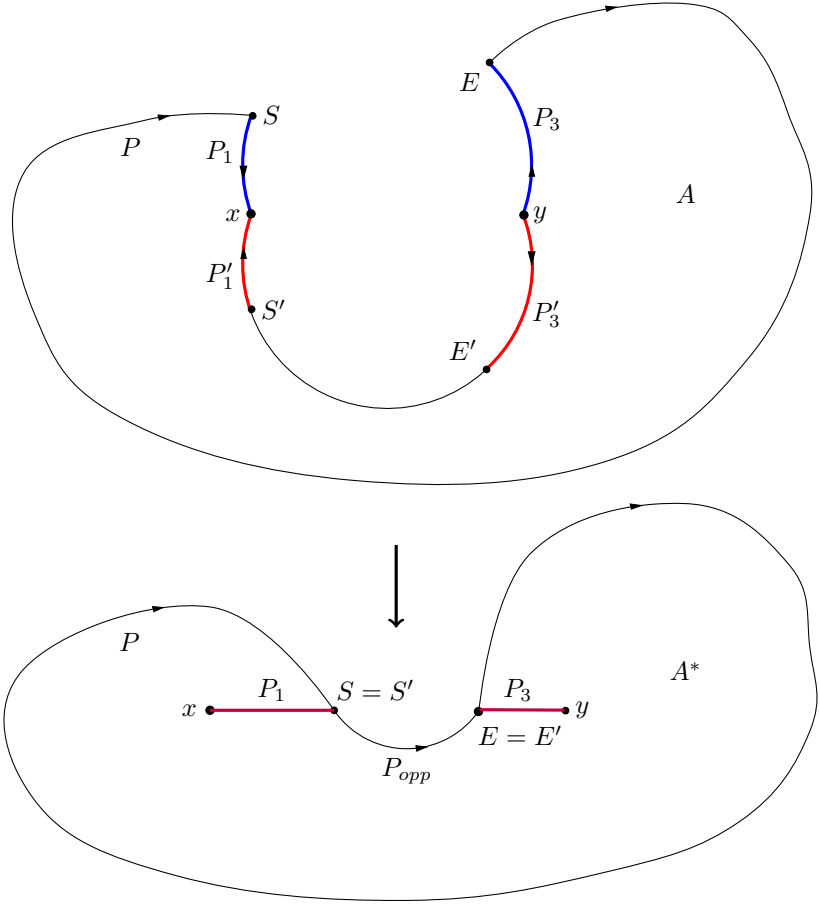

Let $P_{opp}$ be the boundary path of $c$ that is bounded by the points $S$ and $E$ but is different than $P_1 \cup P_2 \cup P_3$ (the path $P_{opp}$ shares only the beginning and the end with $P_{c}$). 

We define the edge path $P^*$ in $L$ to be the edge-path $P$ with the fragment $P_1 \cup P_2 \cup P_3$ replaced by the fragment $P_{opp}$. We define the diagram $D^*$ to be the gluing of $L$ to $A^*$ along the path $P^*$.

In both situations a) and b) the new diagram $\mathcal{D}^*$ contains strictly fewer 2-cells than $\mathcal{D}$ and has the same collaring $n$-cycle of intra-segments, so the same corners. Therefore, we can repeat this procedure inductively to remove all reduction pairs. 
\end{proof}

Since now, we will assume that every diagram collared by $n$-cycle of intra-segments is reduced.

\section{Property (T) in the hexagonal model at densities $> \frac{1}{3}$}\label{Section:T_in_hex_positive}

In this section, we will prove the following proposition.

\begin{prop}\label{Proposition:T_in_hex_positive}
For densities $> \frac{1}{3}$, a random group in the hexagonal model w.o.p. has property \emph{(T)}.
\end{prop}

We will be essentially mimicking the proof of \cite[Thorem B]{kot}. However, our proof is even easier, since the transition from the triangular model to the hexagonal is more straightforward than from the triangular model to the Gromov model. We start by recalling the following definition and theorem.

\begin{defi}[{\cite[Definition 3.12]{kot}}]
For $d \in (0,1)$ a group in the \textit{positive triangular model} is given by the presentation $\left<S | R \right>$, where $|S|=n$ and $R$ is a set of $(2n-1)^{3d}$  words over $S$ that do not consist of elements of $S^{-1}$ (positive words) chosen randomly with uniform distribution among all such sets of words. 

We say that some group property holds \textit{with overwhelming probability} in this model if the probability that a random group has this property tends to 1 as $n \to \infty$.  
\end{defi}

Now, let us recall the following theorem.

\begin{theo}[{\cite[Theorem 3.14]{kot}}]
For densities $> \frac{1}{3}$ a random group in the positive triangular model with overwhelming probability has property \emph{(T)}.
\end{theo}

\begin{proof}[Proof of Proposition \ref{Proposition:T_in_hex_positive}]
Let $S_n$ be the set of size $n$. Denote by $W_n^{6}$ the set of cyclically reduced of words of length six over $S_n$. Let $W_n^{6,+}$ be the set of words of length six over $S_n$ that use only elements of $S_n$ but not their inverses, i.e. $W_n^{6,+}$ is the set of positive words of length six over $S_n$. 

Denote by $R_n^d$ the set of $(2n-1)^{6d}$ elements of $W_n^6$ chosen at random with uniform distribution. Note that $|W_n^{6,+}| > \frac{1}{2^6}|W_n^6|$. Therefore, the expected number of elements of $R_n^d$ that belong to $W_n^{6,+}$ is larger than $\frac{1}{2^6}|W_n^6| > 2(2n-1)^{6d'}$, where the last inequality holds for $d' < d$ and $n$ sufficiently large. Therefore, with overwhelming probability, $R_n^d$ contains at least $(2n-1)^{6d'}$ elements of $W_n^{6,+}$ for any $0 < d' < d$ (this can be obtained by using for example Chernoff bound). Let $R_n^{d,+}$ be the set $R_n^d \cap W_n^{6,+}$.  

Fix $d > \frac{1}{3}$ and choose any $\frac{1}{3} < d' < d$. Let $G^+$ be a group given by the presentation $\left< S_n | R_n^{d,+} \right>$. Let $W_n^{2,+}$ be the set of positive words of length 2 over $S_n$, that is words of length 2 consisting only of elements of $S_n$ (and none inverses of elements of $S_n$). Note that every element of $R_n^{d,+}$ corresponds, in a natural way, to a positive word of length 3 over $W_n^{2,+}$. We will denote the set of this words of length 3 by $R_{tri, n}^{d}$. 

Let $\Gamma$ be a group given by the presentation $\Gamma = \left<W_n^{2,+} | R_{tri, n}^{d} \right>$. Let $\phi : \Gamma \to G^{+}$ be the natural homomorphism, sending each generator of $\Gamma$ to a word of length 2 in $G^{+}$ labeling this generator. 

Now, we will prove that $\phi(\Gamma)$ is a subgroup of finite index in $G^{+}$. To prove this, it suffices to show that every word over $S_n \cup S_n^{-1}$ is equal, in the free group, to a word of form $u$ or $tu$, where $u \in \phi(\Gamma)$ and $t \in S_n \cup S_n^{-1}$.

First, note that every word of length 2 over $S_n \cup S_n^{-1}$ belongs to $\phi(\Gamma)$, because words consisting of only elements of $S_n$ or $S_n^{-1}$ are the images of generators of $\Gamma$ or their inverses and words of form $s_1 s_2^{-1}$, where $s_1, s_2 \in S_n$ can be obtained as $s_1 s_2^{-1} = s_1 x (s_2 x)^{-1}$, for any $x \in S_n$. Therefore, if a word over $S_n \cup S_n^{-1}$ has an even length, it belongs to $\phi(\Gamma)$, and if it has an odd length, then it is of form $t u$, where $u \in \phi(\Gamma)$ and $t \in S_n \cup S_n^{-1}$.

Let $G$ be a group given by the presentation $\left<S_n | R_n^d \right>$, i.e., $G$ is a random group in the hexagonal model at density $d$. Note that there exists a epimorphism $G^{+} \to G$. By the previous observations, we can w.o.p. choose groups $G^+ = \left<S_n | R_n^{+,d} \right>$ and $\Gamma = \left< S_n | R_{tri, n}^{d'} \right>$ in a way that there exits epimorphism from $\Gamma$ onto a finite index subgroup of $G^{+}$. The choice of the presentation of $\Gamma$ is not unique, however we can perform it in a way that $\Gamma$ is chosen in the same way as a random group in the positive triangular model at density $\frac{1}{3} < d' <d$ (the appropriate approach is described in the proof of Theorem A in \cite{kot}). Hence, with overwhelming probability, $G^+$ has Property (T). This ends the proof, since Property (T) is preserved by epimorphisms and finite index extensions (see for example \cite{bdl}). 
\end{proof}

\section{Tree of loops and tree of diagrams}\label{Section:Tree_of_loops}

Now we will introduce a tool for investigating the structure of long intra-segments. In the first approach, we will consider intra-segments with possibly many self-intersections but with no triple points.  

\begin{definition}[tree of loops]\label{Def:Tree_of_loops}
A \textit{simple cycle} is an unoriented graph that is a closed walk with no repetition of edges and no repetition of vertices. We allow a simple cycle to consist of one vertex and one edge. For a union of simple-cycles consisting of $\{P_1, P_2, \dots,  P_n\}$ we define the \textit{dual graph to} it in the following way: the set $\{P_1, P_2, \dots, P_n \}$ is the set of vertices and we join two vertices with an edge if they correspond to two simple cycles having nonempty intersection.  

A union $\mathcal{T}$ of simple cycles from the set $\mathcal{S} = \{P_1, P_2, \dots, P_n \}$, together with the set $\mathcal{S}$, is called a \textit{tree of loops} if it satisfies the following conditions:
\begin{enumerate}
\item \label{Def:Tree_of_loops1} $\mathcal{T}$ is connected,
\item \label{Def:Tree_of_loops2} for every  $1 \leq i \neq  j \leq n$ the simple cycles $P_i$ and $P_j$ have at most one common vertex and no common edges,
\item \label{Def:Tree_of_loops3} the dual graph to $\mathcal{T}$ is a tree.
\end{enumerate}
\end{definition}

Definition \ref{Def:Tree_of_loops} is illustrated in Figure \ref{Figure:Tree_of_loops}.

\begin{figure}[h]
\centering
\begin{tikzpicture}[scale=0.8]

\draw (-2,-1.5) -- (1,-1.5) -- (2,0) -- (0,2) -- (-2,0.5) -- (-1,-2);
\draw  plot[smooth, tension=.7] coordinates {(-2,-1.5)};
\draw  plot[smooth, tension=.7] coordinates {(-2,-1.5) (-2.7678,-1.6514) (-3.4681,-2.2128) (-3.3309,-3.089) (-2.4446,-3.6245) (-1.3031,-3.397) (-0.8544,-2.7625) (-1.0088,-1.9986)};
\draw (-2,0.5) -- (-2.5,1.5) -- (-3.5,1) -- (-3,0) -- cycle;
\draw (0,2) -- (-1,3.5) -- (1,3.5) -- cycle;
\draw (2,0) -- (2.5,1) -- (3.5,1.5) -- (4.5,1) -- (5,0) -- (4.5,-1) -- (3.5,-1.5) -- (2.5,-1) -- cycle;
\draw (2.5,2) -- (3.5,1.5) -- (4.5,2);
\draw  plot[smooth, tension=.7] coordinates {(2.4941,2.0004) (2.1942,2.2867) (2.3088,2.8987) (3.1436,3.141) (4.2074,3.1175) (4.872,2.8175) (4.9103,2.3389) (4.5072,2.0013)};

\draw[fill=black]  (-1.2018,-1.5225) circle (0.05);
\draw[fill=black]  (3.5,1.5) circle (0.05);
\draw[fill=black]  (2,0) circle (0.05);
\draw[fill=black]  (0,2) circle (0.05);
\draw[fill=black]  (-2,0.5) circle (0.05);
\node at (0,0) {$P_1$};
\node at (0,3) {$P_2$};
\node at (-2.8045,0.6791) {$P_3$};
\node at (-2.2803,-2.555) {$P_4$};
\node at (3.5,0) {$P_5$};
\node at (3.5,2.5) {$P_6$};
\draw[fill=black]  (9.5,0.5) circle (0.05);
\draw[fill=black]  (9.5,2.5) circle (0.05);
\draw[fill=black]  (7.5,1) circle (0.05);
\draw[fill=black]  (8.5,-1) circle (0.05);
\draw[fill=black]  (11,2) circle (0.05);
\draw[fill=black]  (11,0.5) circle (0.05);
\draw (9.5,0.5) -- (11,0.5);
\draw (9.5,2.5) -- (9.5,0.5);
\draw (7.5,1) -- (9.5,0.5);
\draw (9.5,0.5) -- (8.5,-1);
\draw (11,2) -- (11,0.5);
\node at (9.6758,0.039) {$P_1$};
\node at (9.5,3) {$P_2$};
\node at (7,1) {$P_3$};
\node at (8,-1.5) {$P_4$};
\node at (11.5,0) {$P_5$};
\node at (11,2.5) {$P_6$};

\end{tikzpicture}
\caption{A tree of loops consisting of six simple cycles and its dual graph.}
\label{Figure:Tree_of_loops}
\end{figure}
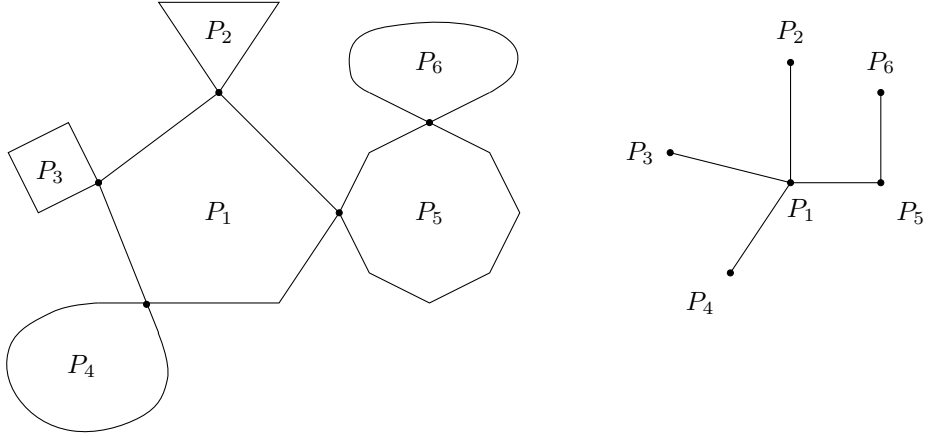

Now, we will show how to generalize the notion of a diagram collared by a $n$-cycle of intra-segments to the notion of a diagram collared by a tree of loops. We will start with a few definitions. Recall that $\cay$ denotes the Cayley complex of a random group. 

\begin{defi}\label{Definition:f}
Let $A$ be a subset of $[0,1]$, such that $0, 1 \in A$. An \textit{admissible function} $f : A \to \cay$ is any continuous function such that: 
\begin{itemize}
\item $f(0)=f(1)$, 
\item $f$ has no triple points (i.e. there are no three distinct points $x,y,z \in [0,1]$ such that $f(x)=f(y)=f(z)$),
\item $f$ has finitely many double points (i.e. distinct points $x,y \in [0,1]$ such that $f(x)=f(y)$). 
\end{itemize}

A \textit{sub-loop} of an admissible function $f$ is a segment  $[x,y] \subseteq [0,1]$ such that $f(x)=f(y)$ and $x \neq y$. The set of sub-loops of  $f$ will be denoted by $\mathcal{I}(f)$.  
\end{defi}

Now, we will introduce a partial order on the set $\mathcal{I}(f)$: 
\begin{defi}
For $i, j \in \mathcal{I}(f)$ we say that $i$ is \textit{smaller} then $j$ if $i$ is a subset of $j$. 
\end{defi}

\begin{rem}
There is only finitely many sub-loops of an admissible function $f$, so there exists a minimal element of $\mathcal{I}(f)$.
\end{rem}

Now, we will define the operation of removing a sub-loop of $f$: 
\begin{defi}\label{Definition:Removing_a_subloop}
Let $f: A \to \cay$ be an admissible function and let $[x,y]$ be a sub-loop of $f$. A function $f$ restricted to $A \setminus [x,y]$ is called $f$ \textit{with removed sub-loop} $[x,y]$.
\end{defi}

\begin{figure}[h]
\centering
\begin{tikzpicture}[scale=0.60]

\tikzset{ decoration={markings,
 mark=at position .280  with {\arrow[red,line width=2pt]{>}},
 mark=at position .660  with {\arrow[red,line width=2pt]{>}},
 } }

  \draw[
     black,postaction={decorate} ]
    plot[smooth, tension=.7] coordinates {(-3.3633,1.8707) (-3.7513,2.9965) (-4.3945,3.7699) (-5.2601,4.4551) (-4.6557,5.2201)};
    
  \draw[
     black, very thick,postaction={decorate} ]
    plot[smooth, tension=.7] coordinates {(-4.6607,5.2272) (-4.1703,5.5624) (-3.3793,5.8996) (-2.7767,5.6282) (-2.6319,5.0368) (-3.2303,4.573) (-4.3426,5.0824) (-4.6502,5.2248)};

\draw[
     black,postaction={decorate} ]
plot[smooth, tension=.7] coordinates {(-4.6599,5.2234) (-4.9804,5.4327) (-5.4145,6.0628) (-5.3059,6.9122) (-4.0982,7.4156) (-3.3994,6.7554)};

\draw[
     black,postaction={decorate} ]
 plot[smooth, tension=.7] coordinates {(-3.394,6.7468) (-2.8992,5.252) (-2.8371,4.1216) (-3.7733,4.0944) (-3.8044,5.1822) (-3.5596,6.2583) (-3.4004,6.7477)};

\draw[
     black,postaction={decorate} ]
plot[smooth, tension=.7] coordinates {(-3.3998,6.7417) (-2.8059,7.5785) (-1.6559,7.1538) (-0.4763,5.3782) (-0.4843,3.042) (-3.3587,1.8581)};

\draw[
     black, postaction={decorate} ]
plot[smooth, tension=.7] coordinates {(-3.3543,1.856) (-4.0011,1.6713) (-4.3445,1.1154) (-3.976,0.5807) (-3.2158,0.7831) (-3.3568,1.8548)};

\draw[->, thick] (1.5,3) -- (4.5,3);

\draw (-6,-1.5) -- (0.5,-1.5);
\draw[fill = black]  (-4.1876,0.741) circle (0.05);
\draw (-6,-1.35) -- (-6,-1.65);
\draw (0.5,-1.35) -- (0.5,-1.65);
\node at (-6,-2) {$0$};
\node at (-4.9,0.5) {$f(0)$};
\node at (0.5,-2) {$1$};
\draw[very thick] (-4,-1.5) -- (-2.5,-1.5);
\node at (-4,-2) {$x$};
\node at (-2.5,-2) {$y$};
\node at (8,-2) {$x$};
\draw  plot[smooth, tension=.7] coordinates {(-2,-1) (-1.5,0) (-1.5,1)};
\draw (-1.6031,0.8162) -- (-1.5,1) -- (-1.3248,0.8616);
\node at (-1.1102,-0.0649) {$f$};

\begin{scope}[shift={(12,0)}]

  \draw[
     black,postaction={decorate} ]
    plot[smooth, tension=.7] coordinates {(-3.3633,1.8707) (-3.7513,2.9965) (-4.3945,3.7699) (-5.2601,4.4551) (-4.6557,5.2201)};
    
\draw[
     black,postaction={decorate} ]
plot[smooth, tension=.7] coordinates {(-4.6599,5.2234) (-4.9804,5.4327) (-5.4145,6.0628) (-5.3059,6.9122) (-4.0982,7.4156) (-3.3994,6.7554)};

\draw[
     black,postaction={decorate} ]
 plot[smooth, tension=.7] coordinates {(-3.394,6.7468) (-2.8992,5.252) (-2.8371,4.1216) (-3.7733,4.0944) (-3.8044,5.1822) (-3.5596,6.2583) (-3.4004,6.7477)};

\draw[
     black,postaction={decorate} ]
plot[smooth, tension=.7] coordinates {(-3.3998,6.7417) (-2.8059,7.5785) (-1.6559,7.1538) (-0.4763,5.3782) (-0.4843,3.042) (-3.3587,1.8581)};

\draw[
     black, postaction={decorate} ]
plot[smooth, tension=.7] coordinates {(-3.3543,1.856) (-4.0011,1.6713) (-4.3445,1.1154) (-3.976,0.5807) (-3.2158,0.7831) (-3.3568,1.8548)};

\draw (-6,-1.5) -- (-4,-1.5);
\draw[fill = black]  (-4.1876,0.741) circle (0.05);
\draw (-6,-1.35) -- (-6,-1.65);
\draw (0.5,-1.35) -- (0.5,-1.65);
\node at (-6,-2) {$0$};
\node at (-5,0.5) {$f'(0)$};
\node at (0.5,-2) {$1$};
\node at (-4,-2) {$x$};
\draw  plot[smooth, tension=.7] coordinates {(-2,-1) (-1.5,0) (-1.5,1)};
\draw (-1.6031,0.8162) -- (-1.5,1) -- (-1.3248,0.8616);
\node at (-1.1102,-0.0649) {$f'$};
\end{scope} 

\draw (9.5,-1.5) -- (12.5,-1.5);
\node at (9.5,-2) {$y$};

\end{tikzpicture}
\caption{Removing a sub-loop $[x,y]$ of $f$. Thick lines represent the sub-loop and its image under $f$. On the right we present the function $f'$ with removed sub-loop $[x,y]$.}
\label{Figure:Rmeoving_a_sub_loop}
\end{figure}
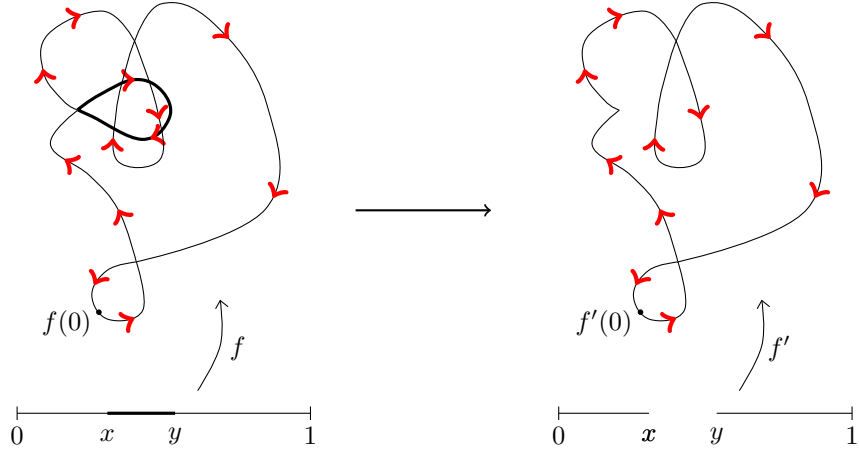

For our later use, we want to keep some information about the sub-loops that were removed. Therefore, we introduce the next definition.

\begin{defi}[bridge points]
Let $f$ be an admissible function. Let $f_0, f_1, f_2, \dots f_l$ be a sequence of functions from $[0,1] \to \cay$ and let $\{[x_i, y_i] \}_{0 \leq i \leq l-1}$ be a sequence of segments in $[0,1]$ such that:

\begin{itemize}
\item $f_0 := f$,
\item for every $0 \leq i \leq l-1$ the segment $[x_i,y_i]$ is a sub-loop of $f_i$,
\item for every $1 \leq i \leq l$ the function $f_i$ is $f_{i-1}$ with removed sub-loop $[x_{i-1}, y_{i-1}]$.
\end{itemize}

We construct the sequence of sets $\mathcal{B}_0, \mathcal{B}_1, \dots, \mathcal{B}_{l}$ in the following way:
\begin{itemize}
\item $\mathcal{B}_0 = \emptyset$
\item $\mathcal{B}_i =  \left(\mathcal{B}_{i-1}  \setminus [x_{i-1},y_{i-1}]\right)  \cup \{x_{i-1}\}$, for $1 \leq i \leq l$ 
\end{itemize}
We say that $\mathcal{B}_i$ is the set of \textit{bridge points of $f_i$}.
\end{defi}

\begin{defi}
We say that an intra-segment $\lambda$ is \textit{admissible} if: $\lambda$ has at least two edges, no edge of $\lambda$ has identified ends, no three edges of $\lambda$ intersect at a single point and every vertex of $\lambda$ belongs to at most two edges of $\lambda$.
\end{defi}

\begin{defi}[inductive construction of a graph of loops]\label{Definition:Graph_of_loops}
Let $\lambda$ be an admissible intra-segment in $\cay$. Suppose that the first and the last vertex of $\lambda$ coincide and call this vertex $v$. Let $\pi : [0,1] \to \cay$ be a continuous function such that: $\pi$ is an immersion, $\pi([0,1]) = \varphi(\lambda)$ and $\pi(0) = \pi(1) = \varphi(v)$. Note, that it means that the point $\pi(0)$, which is the same as $\pi(1)$, lies in the middle of a 2-cell in $\cay$, since an intra-segment has its ends in the middles of 2-cells.

By the fact that $\lambda$ is admissible, $\pi$ satisfies the assumptions of Definition \ref{Definition:f}. 

First, we inductively construct a finite sequence of functions $\pi_0, \pi_1, \pi_2, \dots$ together with a sequence of segments $s_0, s_1, s_2, \dots, \subseteq [0,1]$  in the following way:

\begin{itemize}
\item We define $\pi_0 := \pi$.
\item For $i \geq 0$ let $s_i$ be any minimal element of $\mathcal{I}(\pi_i)$.
\item For $i \geq 0$ the function $\pi_{i+1}$ is defined as $\pi_i$ with removed sub-loop $s_i$. 
\item If for some $i \geq 0$ the set $\mathcal{I}(\pi_i)$ consists of only one element $[0,1]$, then we end the sequence $\pi_0, \pi_1, \dots$ and also we end the sequence $s_0, s_1, \dots$. In this case we also fix the length of both sequences $l:=i$. 
\end{itemize}

For $0\leq j \leq l$ let $\mathcal{B}_i$ be the set of bridge points of $\pi_i$.

In other words: we construct a sequence of functions in which in every step we remove a minimal sub-loop up to the moment when there is nothing to remove, apart from the sub-loop $[0,1]$, and we also remember all the sub-loops we removed (as bridge points). 

Now we construct a graph $\mathcal{T}$, starting from $V = \emptyset$ and $E = \emptyset$, which are two sets to which we will inductively add vertices of $\mathcal{T}$ and edges of $\mathcal{T}$ respectively. Also, let $\mathcal{P} = \emptyset$ be the set to which we will inductively add simple cycles occurring in the construction.

We perform consecutive steps of construction, described below, for $n = 0, 1, \dots, l$.

\vspace{0.5cm}
\textbf{$n$-th Step of the construction of graph of loops}.

For $s \subseteq [0,1]$, the interior of $s$ will be denoted by $\text{int}(s)$. We consider two situations:
\begin{itemize}
\item{Case where $\mathcal{B}_n \cap \text{int}(s_n) \neq \emptyset$} 

Let $x_n$ and $y_n$ be two ends of $s_n$ and suppose $x_n < y_n$. Let $x_n < b_1 < b_2 < \dots < b_k < y_n$ be the bridge points of $\pi_n$ that lie in the interior of $s_n$. We add a vertex to $V$ labeled $x_n$. 

Note, that bridge points in the interior of $s_n$ were created in some previous steps of construction, so in the $n$-th step, for $1 \leq j \leq k$ there exists a vertex in $V$ labeled by $b_k$. 

In this step, for $1 \leq j < k$ we add an edge to $E$ joining $b_{j}$ with $b_{j+1}$. We also add two edges: one joining $x_n$ with $b_1$, and one joining $x_n$ with $b_k$. 

Edges added in this case form a simple cycle  of length $k+1$ on vertices $(x_n, b_1, b_2, \dots, b_k)$. We add this simple cycle to $\mathcal{P}$.

\item{Case where $\mathcal{B}_n \cap \text{int}(s_n) = \emptyset$} 

If $\mathcal{I}(s_n)$ is empty then we add a vertex $v$ to $V$ labeled by the smallest end of $s_n$. We also add an edge to $E$ joining $v$ with itself. In this case, newly added edge forms a simple cycle of length $1$, which we add to $\mathcal{P}$.

\end{itemize}
We define $\mathcal{T}$ to be the graph with the set of vertices $V$ and the set of edges $E$. The graph $\mathcal{T}$ together with the set $\mathcal{P}$ is called the \textit{graph of loops (of $\pi$ or of $\lambda$)}. 

By the construction, every vertex $v$ of $\mathcal{T}$ corresponds to a double point $p$ of $\pi$ ($p$ lies in $[0,1]$). The \textit{coordinates} of a vertex $v$ of $\mathcal{T}$ is the pair of numbers $x, y \in [0,1]$ such that $\pi(x) = \pi(y) = \pi(p)$.
\end{defi}

The construction of graph of loops described in Definition \ref{Definition:Graph_of_loops} is illustrated in Figure \ref{Figure:Graph_of_loops}.

\begin{figure}[h]
\centering
\begin{tikzpicture}[scale=0.5]

\tikzset{ decoration={markings,
 mark=at position .280  with {\arrow[red,line width=2pt]{>}},
 mark=at position .660  with {\arrow[red,line width=2pt]{>}},
 } }

 \node[above right] at (-3.3609,1.8557) {$A$};
 \node[above] at (-4.655,5.2156) {$B$};
 \node[left] at (-3.4004,6.7516) {$C$};

  \draw[
     black,postaction={decorate} ]
    plot[smooth, tension=.7] coordinates {(-3.3633,1.8707) (-3.7513,2.9965) (-4.3945,3.7699) (-5.2601,4.4551) (-4.6557,5.2201)};

  \draw[
     black, very thick,postaction={decorate} ]
    plot[smooth, tension=.7] coordinates {(-4.6607,5.2272) (-4.1703,5.5624) (-3.3793,5.8996) (-2.7767,5.6282) (-2.6319,5.0368) (-3.2303,4.573) (-4.3426,5.0824) (-4.6502,5.2248)};

\draw[
     black,postaction={decorate} ]
plot[smooth, tension=.7] coordinates {(-4.6599,5.2234) (-4.9804,5.4327) (-5.4145,6.0628) (-5.3059,6.9122) (-4.0982,7.4156) (-3.3994,6.7554)};

\draw[
     black,postaction={decorate} ]
 plot[smooth, tension=.7] coordinates {(-3.394,6.7468) (-2.8992,5.252) (-2.8371,4.1216) (-3.7733,4.0944) (-3.8044,5.1822) (-3.5596,6.2583) (-3.4004,6.7477)};

\draw[
     black,postaction={decorate} ]
plot[smooth, tension=.7] coordinates {(-3.3998,6.7417) (-2.8059,7.5785) (-1.6559,7.1538) (-0.4763,5.3782) (-0.4843,3.042) (-3.3587,1.8581)};

\draw[
     black, postaction={decorate} ]
plot[smooth, tension=.7] coordinates {(-3.3543,1.856) (-4.0011,1.6713) (-4.3445,1.1154) (-3.976,0.5807) (-3.2158,0.7831) (-3.3568,1.8548)};

\draw plot[smooth, tension=.7] coordinates {(2.5,3.5) (1.4949,3.0735) (0.9382,3.278) (0.9899,3.8219) (1.5156,4.0783) (2.1191,3.9048) (2.5049,3.4956)};

\node[above] at (2.6,3.5) {$B$};
\node at (-2,1) {$\pi_0$};

\draw[fill=black]  (-4.1456,0.6861) circle (0.05);
\node at (-5.0,0.4754) {$\pi(0)$};

\node at (3,1) {$\mathcal{T}$};

\begin{scope}[shift={(16,0)}]

 \node[above right] at (-3.3609,1.8557) {$A$};
 \node[above] at (-4.655,5.2156) {$B$};
 \node[left] at (-3.4004,6.7516) {$C$};

  \draw[
     black,postaction={decorate} ]
    plot[smooth, tension=.7] coordinates {(-3.3633,1.8707) (-3.7513,2.9965) (-4.3945,3.7699) (-5.2601,4.4551) (-4.6557,5.2201)};
    
\draw[
     black,postaction={decorate} ]
plot[smooth, tension=.7] coordinates {(-4.6599,5.2234) (-4.9804,5.4327) (-5.4145,6.0628) (-5.3059,6.9122) (-4.0982,7.4156) (-3.3994,6.7554)};

\draw[
     black, very thick, postaction={decorate} ]
 plot[smooth, tension=.7] coordinates {(-3.394,6.7468) (-2.8992,5.252) (-2.8371,4.1216) (-3.7733,4.0944) (-3.8044,5.1822) (-3.5596,6.2583) (-3.4004,6.7477)};

\draw[
     black,postaction={decorate} ]
plot[smooth, tension=.7] coordinates {(-3.3998,6.7417) (-2.8059,7.5785) (-1.6559,7.1538) (-0.4763,5.3782) (-0.4843,3.042) (-3.3587,1.8581)};

\draw[
     black, postaction={decorate} ]
plot[smooth, tension=.7] coordinates {(-3.3543,1.856) (-4.0011,1.6713) (-4.3445,1.1154) (-3.976,0.5807) (-3.2158,0.7831) (-3.3568,1.8548)};

\draw  plot[smooth, tension=.7] coordinates {(2.5,3.5) (1.4949,3.0735) (0.9382,3.278) (0.9899,3.8219) (1.5156,4.0783) (2.1191,3.9048) (2.5049,3.4956)};

\draw plot[smooth, tension=.7] coordinates {(4.5,4.5) (5.071,4.7578) (5.4596,5.0876) (5.4001,5.4009) (5.0664,5.5189) (4.6903,5.263) (4.4997,4.5019)};

\node[above] at (2.6,3.5) {$B$};
\node[below right] at (4.5,4.5) {$C$};

\node at (-2,1) {$\pi_1$};
\node at (3,1) {$\mathcal{T}$};

\draw[fill=black]  (-4.1456,0.6861) circle (0.05);
\node at (-5.0,0.4754) {$\pi(0)$};
\end{scope}

\begin{scope}[shift={(0,-10)}]

 \node[above right] at (-3.3609,1.8557) {$A$};
 \node[above] at (-4.655,5.2156) {$B$};
 \node[left] at (-3.4004,6.7516) {$C$};

  \draw[black, very thick, postaction={decorate} ]
    plot[smooth, tension=.7] coordinates {(-3.3633,1.8707) (-3.7513,2.9965) (-4.3945,3.7699) (-5.2601,4.4551) (-4.6557,5.2201)};
    
\draw[
     black, very thick, postaction={decorate} ]
plot[smooth, tension=.7] coordinates {(-4.6599,5.2234) (-4.9804,5.4327) (-5.4145,6.0628) (-5.3059,6.9122) (-4.0982,7.4156) (-3.3994,6.7554)};

\draw[
     black, very thick, postaction={decorate} ]
plot[smooth, tension=.7] coordinates {(-3.3998,6.7417) (-2.8059,7.5785) (-1.6559,7.1538) (-0.4763,5.3782) (-0.4843,3.042) (-3.3587,1.8581)};

\draw[
     black, postaction={decorate} ]
plot[smooth, tension=.7] coordinates {(-3.3543,1.856) (-4.0011,1.6713) (-4.3445,1.1154) (-3.976,0.5807) (-3.2158,0.7831) (-3.3568,1.8548)};

\draw(4,1.5) -- (4.5,4.5) -- (2.5,3.5) -- cycle;

\draw  plot[smooth, tension=.7] coordinates {(2.5,3.5) (1.4949,3.0735) (0.9382,3.278) (0.9899,3.8219) (1.5156,4.0783) (2.1191,3.9048) (2.5049,3.4956)};

\draw  plot[smooth, tension=.7] coordinates {(4.5,4.5) (5.071,4.7578) (5.4596,5.0876) (5.4001,5.4009) (5.0664,5.5189) (4.6903,5.263) (4.4997,4.5019)};

\node[right] at (4,1.5) {$A$};
\node[above] at (2.6,3.5) {$B$};
\node[below right] at (4.5,4.5) {$C$};

\node at (-2,1) {$\pi_2$};
\node at (3,1) {$\mathcal{T}$};

\draw[fill=black]  (-4.1456,0.6861) circle (0.05);
\node at (-5.0,0.4754) {$\pi(0)$};
\end{scope}

\begin{scope}[shift={(16,-10)}]

 \node[above right] at (-3.3609,1.8557) {$A$};
 \draw[
     black, very thick, postaction={decorate} ]
plot[smooth, tension=.7] coordinates {(-3.3543,1.856) (-4.0011,1.6713) (-4.3445,1.1154) (-3.976,0.5807) (-3.2158,0.7831) (-3.3568,1.8548)};
\draw (4,1.5) -- (4.5,4.5) -- (2.5,3.5) -- cycle;

\draw  plot[smooth, tension=.7] coordinates {(4,1.5) (3.903,0.8362) (3.9413,0.1702) (4.261,0.0009) (4.4553,0.3553) (4.3855,0.9438) (3.9927,1.5079)};

\draw  plot[smooth, tension=.7] coordinates {(2.5,3.5) (1.4949,3.0735) (0.9382,3.278) (0.9899,3.8219) (1.5156,4.0783) (2.1191,3.9048) (2.5049,3.4956)};

\draw  plot[smooth, tension=.7] coordinates {(4.5,4.5) (5.071,4.7578) (5.4596,5.0876) (5.4001,5.4009) (5.0664,5.5189) (4.6903,5.263) (4.4997,4.5019)};

\node[right] at (4,1.5) {$A$};
\node[above] at (2.6,3.5) {$B$};
\node[below right] at (4.5,4.5) {$C$};

\node at (-2,1) {$\pi_3$};
\node at (3,1) {$\mathcal{T}$};

\draw[fill=black]  (-4.1456,0.6861) circle (0.05);
\node at (-5.0,0.4754) {$\pi(0)$};
\end{scope}

\draw[dotted] (0.5,-9.5) -- (0.5,-2.5);
\draw[dotted] (0.5,0.5) -- (0.5,7.5);
\draw[dotted] (16.5,0.5) -- (16.5,7.5);
\draw[dotted] (16.5,-9.5) -- (16.5,-2.5);
\draw[->, very thick] (6.5,4.5) -- (9.5,4.5);
\draw[->, very thick] (9.5,0.5) -- (6.5,-2);
\draw[->, very thick] (6.5,-5.5) -- (10,-5.5);
\node at (0.5,0) {$n=0$};
\node at (16.5,0) {$n=1$};
\node at (0.5,-10) {$n=2$};
\node at (16.5,-10) {$n=3$};

\draw[fill=black]  (20.3567,-9.9222) circle (0.05);
\node at (20.7267,-10.0466) {$0$};

\end{tikzpicture}
\caption{Consecutive steps of the construction of a graph of loops for a given function $\pi$. Thick lines represent the minimal sub-loops that are being removed. The way of constructing $\mathcal{T}$ is, in general, not unique, since the choice of a minimal sub-loop may not be unique. In the last ($n=3$) step of the construction we add two edges joining $A$ with the vertex labeled $0$ (which corresponds to the first and last vertex of $\lambda$.}
\label{Figure:Graph_of_loops}
\end{figure}
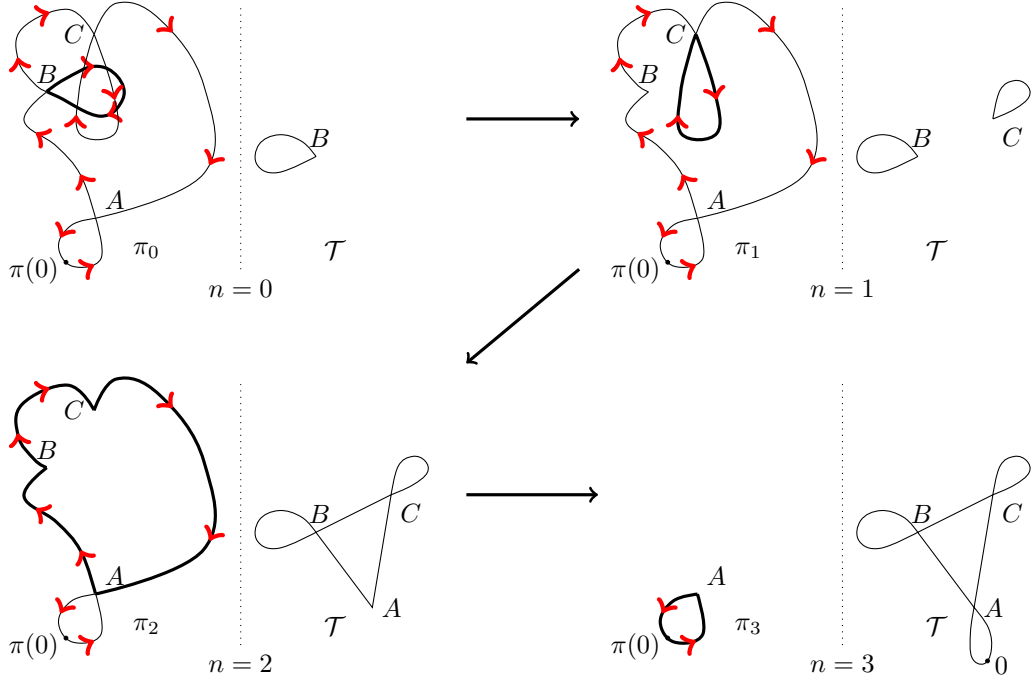

\begin{rem}\label{Remark:map_0_1_to_tree_of_loops}
There exists a continuous immersion $\psi$ from $[0,1]$ onto the simplicial realisation of $\mathcal{T}$, such that every edge of the simplicial realisation of $\mathcal{T}$ is a homeomorphic image of a unique segment of $[0,1]$.
\end{rem}

\begin{proof}
Let $v$ be a vertex of $\mathcal{T}$. Every vertex of $\mathcal{T}$ is created by identifying ends $x$ and $y$ of a sub-loop of function $\pi_i$, for some $i \geq 0$, where $\pi_i$ is as in Definition \ref{Definition:Graph_of_loops}. Removing a sub-loop does not create any new sub-loops, so every sub-loop of $\pi_i$ equals to some sub-loop of $\pi_0$. 

Let $\text{Coord}(\mathcal{T})$ be the set of coordinates of all vertices of $\mathcal{T}$. Observe, that $\text{Coord}(\mathcal{T})$ may not be equal to the set of all double points of $\pi_0$. Moreover, note that $0,1 \in \text{Coord}(\mathcal{T})$.

Note that, in every step of the construction of $\mathcal{T}$, we add an edge $e_{t, u}$ between two consecutive elements $t$ and $u$ of $\text{Coord}(\mathcal{T})$.  

The desired map $\psi$ sends each subsegment of $[0,1]$, bounded by two consecutive points $t$ and $u$ of $\text{Coord}(\mathcal{T})$, to the edge $e_{t,u}$ of $\mathcal{T}$ with the ends $t$ and $u$ (here we identify edges of $\mathcal{T}$ with edges of its simplicial realisation). 
\end{proof}

\begin{prop}\label{Proposition:Graph_of_loops_is_tree}
The graph of loops of $\pi$ is a tree of loops. 
\end{prop}

\begin{proof}
The graph $\mathcal{T}$ is connected since, by Remark \ref{Remark:map_0_1_to_tree_of_loops}, there exists a continuous map from $[0,1]$ to the simplicial realisation of $\mathcal{T}$. Hence, $\mathcal{T}$ satisfies Condition (\ref{Def:Tree_of_loops1}) of Definition \ref{Def:Tree_of_loops}.

Note, that in every step of the construction of the graph of loops we add exactly one new simple cycle, which shares at most one vertex with each of the already existing simple cycles. Therefore, the graph of loops satisfies Assertion (\ref{Def:Tree_of_loops2}) of Definition \ref{Def:Tree_of_loops}.

We need to prove that the dual graph $D_{\mathcal{T}}$ of $\mathcal{T}$ is a tree. By the construction of $\mathcal{T}$, graph $D_{\mathcal{T}}$ does not contain any double edges. Hence, we need to prove that there are no cycles of length $>2$ in $D_{\mathcal{T}}$. 

Suppose, on the contrary, that such a cycle exists and let $\{u_1, u_2, \dots, u_k \}$ be its vertices. For $1 \leq i \leq k$ let $\mathcal{P}_i$ be the simple cycle of $\mathcal{T}$ corresponding to the vertex $u_i$ of $D_{\mathcal{T}}$. Without loss of generality, we can suppose that, in the inductive construction of $\mathcal{T}$, the simple cycle $\mathcal{P}_1$ was added to the graph of loops as the first element of the set $\mathcal{P}_{set} := \{P_1, P_2, \dots, P_k \}$. Suppose that $\mathcal{P}_2$ is the second simple cycle of $\mathcal{P}_{set}$ added to $\mathcal{T}$ and $\mathcal{P}_k$ is the last one added to $\mathcal{T}$. By Definiton \ref{Definition:Graph_of_loops}, in the step in which we create the simple cycle $\mathcal{P}_2$, we also add an edge in $D_{\mathcal{T}}$ joining $u_1$ and $u_2$. To create in $D_{\mathcal{T}}$ an edge joining $u_k$ and $u_1$, the simple cycle $\mathcal{P}_k$ must share a vertex with $\mathcal{P}_1$. However, in the construction of $\mathcal{T}$ after adding $\mathcal{P}_2$ we remove the sub-loop associated to $\mathcal{P}_2$, so there are no longer the bridge points in $[0,1]$ corresponding to vertices of $\mathcal{P}_1$. Therefore, at the moment of adding $\mathcal{P}_k$ it is impossible to create an edge in $D_{\mathcal{T}}$ between $v_k$ and $v_1$. Hence, Assertion  (\ref{Def:Tree_of_loops3}) of Defnition \ref{Def:Tree_of_loops} is satisfied.
\end{proof}

By Proposition \ref{Proposition:Graph_of_loops_is_tree} we can always use the notion of a tree of loops instead of a graph of loops.

\begin{defi}[hypergraph realization of a simple cycle of tree of loops]
Let $\lambda$, $\pi$ and $\mathcal{T}$ be as in Definition \ref{Definition:Graph_of_loops} and let $\psi$ be as in the statement of Remark \ref{Remark:map_0_1_to_tree_of_loops}. 

We define the \textit{hypergraph realization} of an edge $e$ of $\mathcal{T}$ as an intra-segment $\lambda$, such that $\varphi(\lambda) = \pi(\psi^{-1}(e))$. 

For a simple cycle $\mathcal{C}$ in $\mathcal{T}$ we define its \textit{hypergraph realization} $\mathcal{C}_{\lambda}$ as the sequence of hypergraph realizations of all consecutive edges of $\mathcal{C}$. Every vertex $v$ of $\mathcal{C}$ corresponds to a point $v_{\lambda}$ where two intra-segments of $\mathcal{C}_{\lambda}$ meet. We will say that a 2-cell $c$ of $\cay$ \textit{contains} the vertex $v$ of $\mathcal{C}$ (which is also a vertex of $\mathcal{T}$) if $v_{\lambda}$ is the midpoint of $c$. If a 2-cell $c'$ of a 2-dimensional complex (diagram) is mapped, under the natural map $\varphi$, onto a 2-cell of $\cay$ that contains $v$, we also say that the 2-cell $c'$ \textit{contains} the vertex $v$.

The union of all intra-segments of $C_{\lambda}$ is also denoted by $C_{\lambda}$. 
\end{defi}

\begin{lem}\label{Lemma:Hypergraph_realisation_is_n_cycle}
Let $\lambda$, $\pi$ and $\mathcal{T}$ be as in Definition \ref{Definition:Graph_of_loops}. Let $\mathcal{C}$ be a simple cycle of $\mathcal{T}$ and let $\mathcal{C}_{\lambda}$ be its hypergraph realization. Then $C_{\lambda}$ is a multi-cycle of intra-segments. 
\end{lem}

\begin{proof}
We need to prove that $\varphi(\mathcal{C}_{\lambda})$ is an injected circle into $\cay$. 

Let $i$ be the number of the step of construction of $\mathcal{T}$ (described in Definition \ref{Definition:Graph_of_loops}) in which the simple cycle $\mathcal{C}$ is being added to $\mathcal{T}$. Let $s_i$, $\pi$ and $\pi_i$ be as in Definition \ref{Definition:Graph_of_loops}. By the definition of the hypergraph realization, we know that $\varphi(C_{\lambda}) = \pi_i(s_i)$. Suppose, on the contrary, that there exist a self-intersection of $C_{\lambda}$. The existence of such a self-intersection means that there exists a sub-loop of $\pi_i$ that lies entirely inside $s_i$. This contradicts the fact that $s_i$ is a minimal sub-loop of $\pi_i$.
\end{proof}

\begin{rem}\label{Remark:Maximally_two_twins}
Let $\lambda$ be an admissible intra-segment in $\cay$. Suppose that $\lambda$ cycles, namely that the first and the last vertex of $\lambda$ coincide in $\cay$. Let $\mathcal{T}$ be a tree of loops of $\lambda$. Then a vertex of $\mathcal{T}$ can be contained in 2-cells of maximally two diagrams collared by the hypergraph realizations of simple-cycles of $\mathcal{T}$.
\end{rem}

\begin{proof}
Suppose, on the contrary, that a vertex $v$ of $\mathcal{T}$ belongs to at least three simple cycles. It means that there are at least six edges of $\mathcal{T}$ ending in $v$. Hence, at least three edges of $\lambda$ intersect in a single point. This contradicts with the assumption that $\lambda$ is admissible. 
\end{proof}

\begin{defi}[twins]
Let $\lambda$ be as in the statement of Remark \ref{Remark:Maximally_two_twins}. Let $\mathcal{T}$ be a tree of loops of $\lambda$. A \textit{twin pair (of $\mathcal{T}$)} is a pair of corners $\{c_1, c_2\}$ of diagrams collared by hypergraph realisations of simple cycles of $\mathcal{T}$, that contain the same vertex of $\mathcal{T}$. By Remark \ref{Remark:Maximally_two_twins} twin pairs are separate sets. A \textit{twin} is an element of a twin pair. The \textit{twin partner} of a twin $c$ is the 2-cell $c'$ forming a twin pair with $c$.
\end{defi} 

Now, we will present the construction of a tree of diagrams, which can be viewed as a diagram collared by a tree of loops.

\begin{defi}[tree of diagrams]\label{Definition:Tree_of_diagrams}
Let $\lambda$ be an admissible intra-segment, and suppose that the first and the last vertex of $\lambda$ coincide. Let $\mathcal{T}$ be a tree of loops of $\lambda$ and denote by $\mathcal{HR}(\mathcal{T})$ the set of hypergraph realizations of the simple cycles of $\mathcal{T}$. For every $\mathcal{P} \in \mathcal{HR}(\mathcal{T})$ let us denote by $\mathcal{D}_{\mathcal{P}}$ a reduced diagram collared by $\mathcal{P}$ (such diagram exists by Lemma \ref{Lemma:Collared_n_cycles_reduced_exists}). 

We define the \textit{tree of diagrams $\mathcal{D}_{\mathcal{T}}$ (of $\mathcal{T}$)} in the following way:

\begin{equation}\label{identification_map}
\mathcal{D}_{\mathcal{T}} := \coprod_{\mathcal{P} \in \mathcal{HR}(\mathcal{T})} \mathcal{D}_{\mathcal{P}} \slash \sim,
\end{equation}

where $c_1 \sim c_2$ iff 2-cells $c_1$ and $c_2$ are \textbf{shell} corners of diagrams collared by elements of $\mathcal{HR}(\mathcal{T})$ and $c_1$ and $c_2$ are twins (recall, that a corner is called shell if at least half of its edges are external edges). In other words: we take the disjoint union of all diagrams collared by elements of $\mathcal{HR}(\mathcal{T})$ and identify all pairs of shell corners that are twins. We say that $\mathcal{D}_{\mathcal{T}}$ is \textit{collared} by $\mathcal{T}$ or by $\lambda$.

A \textit{component} of a tree of diagrams is its connected component. The \textit{main 2-cell} of $\mathcal{D}_{\mathcal{T}}$ is the 2-cell containing the first and the last edge of $\lambda$.  A \textit{corner} of the tree of diagrams $\mathcal{D}_{\mathcal{T}}$ is any corner of a diagram collared by a hypergraph realization of a simple cycle of $\mathcal{T}$.

The \textit{ladder of a component $\mathcal{E}$ of} $\mathcal{D}_{\mathcal{T}}$ is the subcomplex of $\mathcal{E}$ consisting of all 2-cells that are images under the identification map defined in (\ref{identification_map}) of 2-cells belonging to the ladders of diagrams $\mathcal{D}_{\mathcal{P}}$ for  $\mathcal{P} \in\mathcal{HR}(\mathcal{T})$.

For simplicity, we will often refer to the 2-cells of $\mathcal{D}_{\mathcal{T}}$ as to the 2-cells of diagrams $\mathcal{D}_{\mathcal{P}}$ for $\mathcal{P} \in \mathcal{HR}(\mathcal{T})$, for example: we will say that a 2-cell $c$ of $\mathcal{D}_{\mathcal{T}}$ \textit{belongs} to a diagram $\mathcal{D}_{\mathcal{P}}$ for some $\mathcal{P} \in \mathcal{HR}(\mathcal{T})$ if $c$ is the image in $\mathcal{D}_{\mathcal{T}}$ of a 2-cell belonging to $\mathcal{D}_{\mathcal{P}}$ under the projection map defined in (\ref{identification_map}).
\end{defi}

Definition \ref{Definition:Tree_of_diagrams} is illustrated in Figure \ref{Figure:Tree_of_diagrams} and in Figure \ref{Figure:Tree_of_diagrams_boundary_figure}.

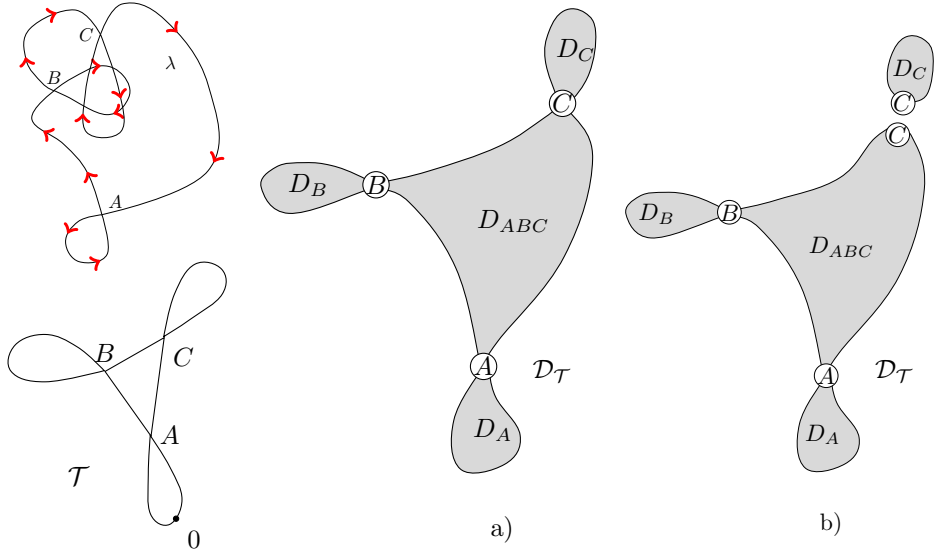
\begin{figure}[h]
\centering
\begin{tikzpicture}[scale=0.70]

\tikzset{ decoration={markings,
 mark=at position .280  with {\arrow[red,line width=2pt]{>}},
 mark=at position .660  with {\arrow[red,line width=2pt]{>}},
 } }
 
\draw[fill=gray!30]  plot[smooth cycle, tension=.7] coordinates {(4.1975,2.0909) (3.6324,1.0672) (3.7571,0.1406) (4.5974,0.2303) (4.8915,0.7859) (4.4248,1.513) (4.2798,2.1007)};
\draw[fill=gray!30]  plot[smooth cycle, tension=.7] coordinates {(5.8227,7.0989) (6.204,7.6614) (6.3261,8.4405) (6.1227,8.8124) (5.7261,8.8499) (5.4415,8.4619) (5.3703,7.9044) (5.6143,7.1621) (5.6906,7.1469)};
\draw[fill=gray!30]  plot[smooth cycle, tension=.7] coordinates {(2.0527,5.5573) (1.8229,5.3884) (0.5,5) (-0.0268,5.5007) (0.2641,5.8845) (0.8037,5.9639) (1.3035,5.9107) (1.7851,5.7353)};
\draw[fill=gray!30]  plot[smooth cycle, tension=.7] coordinates {(2.2783,5.5943) (3.1674,5.8078) (4.4869,6.2733) (5.3533,6.8182) (5.6365,7.0458) (5.7377,7.0407) (5.7983,6.9851) (6.2687,6.0393) (5.7302,4.1811) (4.4755,2.5788) (4.2659,2.2158) (4.2457,2.0388) (4.1294,1.9984) (4.1174,2.2147) (3.6673,3.9684) (2.7771,5.2176) (2.1316,5.4982) (2.3081,5.647) (2.1975,5.5909)};
\draw[fill = white]  (4.1975,2.0909) circle (0.25);
\draw[fill = white]  (2.1549,5.543) circle (0.25);
\draw[fill = white]  (5.6975,7.0909) circle (0.25);
\node at (2.1593,5.5384) {$B$};
\node at (4.1975,2.0909) {$A$};
\node at (5.6975,7.0909) {$C$};

\begin{scope}[shift={(-1.45,0)}]
\node at (6.2155,4.8479) {$D_{ABC}$};
\node at (2.3172,5.5478) {$D_{B}$};
\node at (5.8217,0.8911) {$D_{A}$};
\node at (7.4004,8.1052) {$D_{C}$};
\node at (6,-1) {a)};
\end{scope}

\begin{scope}[shift={(-1,5.25)}, transform canvas={scale=0.7}]

 \node[above right] at (-3.3609,1.8557) {$A$};
 \node[above] at (-4.655,5.2156) {$B$};
 \node[left] at (-3.4004,6.7516) {$C$};
 \node at (-1.5,6) {$\lambda$};

  \draw[
     black,postaction={decorate} ]
    plot[smooth, tension=.7] coordinates {(-3.3633,1.8707) (-3.7513,2.9965) (-4.3945,3.7699) (-5.2601,4.4551) (-4.6557,5.2201)};
    
  \draw[
     black, postaction={decorate} ]
    plot[smooth, tension=.7] coordinates {(-4.6607,5.2272) (-4.1703,5.5624) (-3.3793,5.8996) (-2.7767,5.6282) (-2.6319,5.0368) (-3.2303,4.573) (-4.3426,5.0824) (-4.6502,5.2248)};

\draw[
     black,postaction={decorate} ]
plot[smooth, tension=.7] coordinates {(-4.6599,5.2234) (-4.9804,5.4327) (-5.4145,6.0628) (-5.3059,6.9122) (-4.0982,7.4156) (-3.3994,6.7554)};

\draw[
     black,postaction={decorate} ]
 plot[smooth, tension=.7] coordinates {(-3.394,6.7468) (-2.8992,5.252) (-2.8371,4.1216) (-3.7733,4.0944) (-3.8044,5.1822) (-3.5596,6.2583) (-3.4004,6.7477)};

\draw[
     black,postaction={decorate} ]
plot[smooth, tension=.7] coordinates {(-3.3998,6.7417) (-2.8059,7.5785) (-1.6559,7.1538) (-0.4763,5.3782) (-0.4843,3.042) (-3.3587,1.8581)};

\draw[
     black, postaction={decorate} ]
plot[smooth, tension=.7] coordinates {(-3.3543,1.856) (-4.0011,1.6713) (-4.3445,1.1154) (-3.976,0.5807) (-3.2158,0.7831) (-3.3568,1.8548)};

\end{scope}
\begin{scope}[shift={(6,0)}, transform canvas={scale=0.9}]
\draw[fill=gray!30]  plot[smooth cycle, tension=.7] coordinates {(5.8938,2.127) (5.3287,1.1033) (5.4534,0.1767) (6.2937,0.2664) (6.5878,0.822) (6.1211,1.5491) (5.9761,2.1368)};
\draw[fill=gray!30]  plot[smooth cycle, tension=.7] coordinates {(7.6273,7.8834) (8.0276,8.0809) (8.1497,8.86) (7.9463,9.2319) (7.5497,9.2694) (7.2651,8.8814) (7.1939,8.3239) (7.431,7.9642) (7.6273,7.8834)};
\draw[fill=gray!30]  plot[smooth cycle, tension=.7] coordinates {(3.749,5.5934) (3.5192,5.4245) (2.1963,5.0361) (1.6695,5.5368) (1.9604,5.9206) (2.5,6) (2.9998,5.9468) (3.4814,5.7714)};
\draw[fill=gray!30]  plot[smooth cycle, tension=.7] coordinates {(3.9746,5.6304) (4.8637,5.8439) (6.1832,6.3094) (6.9278,7.152) (7.4634,7.4457) (7.4076,7.2201) (7.694,7.1895) (7.965,6.0754) (7.4265,4.2172) (6.1718,2.6149) (5.9622,2.2519) (5.942,2.0749) (5.8257,2.0345) (5.8137,2.2508) (5.3636,4.0045) (4.4734,5.2537) (3.8279,5.5343) (4.0044,5.6831) (3.8938,5.627)};
\draw[fill = white]  (5.8938,2.127) circle (0.25);
\draw[fill = white]  (3.8512,5.5791) circle (0.25);
\draw[fill = white]  (7.4032,7.2158) circle (0.25);
\node at (3.8556,5.5745) {$B$};
\node at (5.8938,2.127) {$A$};
\node at (7.4017,7.2112) {$C$};

\node at (6.2155,4.8479) {$D_{ABC}$};
\node at (2.3172,5.5478) {$D_{B}$};
\node at (5.8217,0.8911) {$D_{A}$};
\node at (7.6703,8.5994) {$D_{C}$};
\draw[fill = white]  (7.5174,7.8844) circle (0.25);
\node at (7.5079,7.8833) {$C$};
\node at (6,-1) {b)};
\end{scope}

\draw (-2.1178,0.7664) -- (-1.868,2.6377) -- (-2.9915,1.9953) -- cycle;

\draw  plot[smooth, tension=.7] coordinates {(-2.1337,0.7664) (-2.1701,0.0351) (-2.1284,-0.7141) (-1.7635,-0.9116) (-1.5299,-0.4408) (-1.7569,0.1708) (-2.1422,0.7902)};

\draw  plot[smooth, tension=.7] coordinates {(-2.9977,2.0136) (-4.1258,1.7983) (-4.7761,1.9973) (-4.6986,2.4976) (-4.1203,2.6981) (-3.5154,2.4595) (-3.0087,2.0119)};

\draw  plot[smooth, tension=.7] coordinates {(-1.8816,2.6553) (-1.3709,2.9875) (-0.7677,3.4621) (-0.7752,3.9368) (-1.2842,4.0334) (-1.6721,3.4677) (-1.8776,2.647)};

\node[right] at (-2.1314,0.7803) {$A$};
\node[above] at (-3,2) {$B$};
\node[below right] at (-1.8735,2.6147) {$C$};

\node at (-3.5,0) {$\mathcal{T}$};
\node at (5.5,2) {$\mathcal{D}_{\mathcal{T}}$};
\node at (12,2) {$\mathcal{D}_{\mathcal{T}}$};

\draw[fill=black]  (-1.6424,-0.8035) circle (0.05);
\node at (-1.2998,-1.2002) {$0$};

\end{tikzpicture}
\caption{A segment $\lambda$ of hypergraph, its tree of loops $\mathcal{T}$ and two of many possible shapes of the tree of diagrams $\mathcal{D}_{\mathcal{T}}$ of $\lambda$. Gray areas represent diagrams collared by multi-cycles of intra-segments and white circles are single 2-cells, that are corners of these diagrams. In situation a) we present $\mathcal{D}_{\mathcal{T}}$ in case where all corners of diagrams $D_A$, $D_B$, $D_C$ and $D_{ABC}$ are shell. In situation b) we present $\mathcal{D}_{\mathcal{T}}$ in case where the 2-cell $C$ is not a shell corner of $D_{ABC}$ and all other corners are shell (in this case $\mathcal{D}_{\mathcal{T}}$ is not connected). Note, that $\mathcal{D}_{\mathcal{T}}$ may contain many 2-cells labeled by the same relator (and belonging to different diagrams collared by simple cycles of $\mathcal{T}$).}
\label{Figure:Tree_of_diagrams}
\end{figure}

Note, that a priori a tree of diagrams may not be connected -- this is the case where at least one of the corners of diagrams collared by elements of $\mathcal{HR}(\mathcal{T})$ (as in Definition \ref{Definition:Tree_of_diagrams}) is not shell. However, we will show in the upcoming sections that in the hexagonal model at density $<\frac{1}{3}$ and the square model at density $< \frac{3}{8}$ w.o.p. every tree of diagrams is connected.

\begin{rem}\label{Remark:The_main_is_not_twin}
The main 2-cell of a tree of diagrams is not a twin and every corner of a tree of diagrams, different than the main 2-cell, is a twin.
\end{rem}

\begin{proof}
Let $\mathcal{D}_\mathcal{T}$ be a tree of diagrams collared by a tree of loops $\mathcal{T}$. Let $v$ be the vertex of $\mathcal{T}$ contained in the main 2-cell of $\mathcal{D}_{\mathcal{T}}$. Note that the vertex $v$ is created in the last step of construction of a tree of diagrams since $v$ is created at the moment when the sub-loop being the whole segment $[0,1]$ is the only remaining sub-loop. Therefore, $v$ belongs to only one simple cycle of $\mathcal{T}$, so it is contained in only one 2-cell, which is also the main 2-cell of $\mathcal{D}_{\mathcal{T}}$. This proofs that the main-cell of $\mathcal{D}_{\mathcal{T}}$ cannot be a twin.
 
In the construction of the tree of loops every vertex $v$, different than the one corresponding to the sub-loop $[0,1]$, occurs twice in the inductive procedure: first, when it is being created and second time when we consider a sub-loop containing the bridge point associated with $v$. In further steps we remove this sub-loop, together with the bridge point corresponding to $v$, or we end the construction. Hence, such $v$ belongs to exactly two simple cycles of $\mathcal{T}$.
\end{proof}

\section{Properties of trees of diagrams in the hexagonal model}\label{Section:Trees_hexagonal}

In this section, we will analyze how some specific properties of the hexagonal model affect the properties of trees of diagrams. Until the end of Section \ref{Section:Bent_walls_hexagonal} a random group which we consider comes from the hexagonal model at density $d < \frac{1}{3}$.

First, we will explain how the number $\frac{1}{3}$ arises as a natural candidate for the density threshold for Property (T) in the hexagonal model. Observe, that if $d > \frac{1}{3}$ the diagram presented in Figure \ref{Figure:Long_cycle_hex} cannot be excluded using the Isoperimetric Inequality (Theorem \ref{Theorem:Generalized_Isoperimetric_Inequality}). It means that there may be arbitrarily long cycles in hypergraphs, thus, hypergraphs are far from being trees. For density $< \frac{1}{3}$ hypergraphs may not be embedded trees, but we have some control over their behavior; it is explained in the further part of this paper.

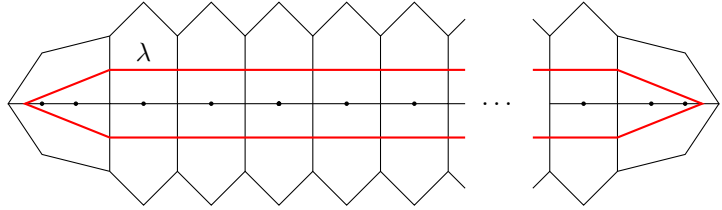
\begin{figure}[h]
\centering
\begin{tikzpicture}[scale=0.45]

\draw (-12.5,6) -- (-2.5,6) -- (-2.5,8) -- (-3.5,9) -- (-4.5,8) -- (-5.5,9) -- (-6.5,8) -- (-7.5,9) -- (-8.5,8) -- (-9.5,9) -- (-10.5,8) -- (-11.5,9) -- (-12.5,8) -- (-12.5,6.5) -- cycle;
\draw (-10.5,6) -- (-10.5,8);
\draw (-8.5,6) -- (-8.5,8);
\draw (-6.5,6) -- (-6.5,8);
\draw (-4.5,6) -- (-4.5,8);
\draw (-12.5,6) -- (-12.5,4) -- (-11.5,3) -- (-10.5,4) -- (-9.5,3) -- (-8.5,4) -- (-7.5,3) -- (-6.5,4) -- (-5.5,3) -- (-4.5,4) -- (-3.5,3) -- (-2.5,4) -- (-2.5,6);

\node at (-1,6) {$\dots$};

\draw (-2,6) -- (-2.5,6);
\draw (-2,8.5) -- (-2.5,8);

\draw[fill=black]  (4.5,6) circle (0.05);
\draw[fill=black]  (3.5,6) circle (0.05);
\draw[fill=black]  (1.5,6) circle (0.05);
\draw[fill=black]  (-11.5,6) circle (0.05);
\draw[fill=black]  (-9.5,6) circle (0.05);
\draw[fill=black]  (-5.5,6) circle (0.05);
\draw[fill=black]  (-7.5,6) circle (0.05);
\draw[fill=black]  (-7.5,6) circle (0.05);
\draw[fill=black]  (-7.5,6) circle (0.05);
\draw[fill=black]  (-7.5,6) circle (0.05);

\draw[fill=black]  (-3.5,6) circle (0.05);

\draw[fill=black]  (-14.5,6) circle (0.05);
\draw[fill=black]  (-13.5,6) circle (0.05);

\draw (-12.5,6) -- (-15.5,6) -- (-14.5,7.5) -- (-12.5,8);
\draw (-15.5,6) -- (-14.5,4.5) -- (-12.5,4);
\draw (-10.5,6) -- (-10.5,4);
\draw (-8.5,6) -- (-8.5,4);
\draw (-6.5,6) -- (-6.5,4);
\draw (-4.5,6) -- (-4.5,4);
\draw (-2.5,4) -- (-2,3.5);
\draw (0,8.5) -- (0.5,8) -- (1.5,9) -- (2.5,8) -- (2.5,6) -- (2.5,4) -- (1.5,3) -- (0.5,4) -- (0,3.5);
\draw (0.5,8) -- (0.5,4);
\draw (0.5,6) -- (2.5,6);
\draw (2.5,6) -- (3.5,6) -- (4.5,6) -- (5.5,6) -- (4.5,7.5) -- (2.5,8);
\draw (2.5,4) -- (4.5,4.5) -- (5.5,6);

\draw[red, thick] (-2,7) -- (-12.5,7) -- (-15,6) -- (-12.5,5) -- (-2,5);
\draw[red, thick] (0,7) -- (2.5,7) -- (5,6) -- (2.5,5) -- (0,5);
\node[above] at (-11.5,7) {$\lambda$};

\end{tikzpicture}
\caption{In the hexagonal model at density $d > \frac{1}{3}$ there may exist an arbitrarily long cycle of a hypergraph.}
\label{Figure:Long_cycle_hex}
\end{figure}

\begin{lem}\label{Lemma:Gluing_simple_cycles_hex}
Let $\lambda$ be an admissible intra-segment, and suppose that the first and the last vertex of $\lambda$ coincide. Let $\mathcal{T}$ be the tree of loops of $\lambda$ and $\mathcal{D}_{\mathcal{T}}$ be the tree of diagrams of $\lambda$.  
Let $\mathcal{C}_1$ and $\mathcal{C}_2$ be two simple cycles of $\mathcal{T}$ and let $\mathcal{D}_1$ and $\mathcal{D}_2$ be diagrams collared by hypergraph realisations of $\mathcal{C}_1$ and $\mathcal{C}_2$ respectively. Let $c_1$ and $c_2$ be shell corners of $\mathcal{D}_1$ and $\mathcal{D}_2$ respectively. Suppose that $c_1$ and $c_2$ are twins (it means, that $c_1$ and $c_2$ are mapped onto the same 2-cell of $\cay$, thus, are labeled by the same relator). Let $\gamma_1$ be the boundary edge-path of $c_1$ along which $c_1$ is glued to $\mathcal{D}_1 - c_1$. Analogously, let $\gamma_2$ be the boundary edge-path of $c_2$ along which $c_2$ is glued to $\mathcal{D}_2 - c_2$. Let $c_{1,2}$ be the image of $c_1$ (and also $c_2$) under the map $\phi$ that sends each 2-cell of $\mathcal{D}_1 \cup \mathcal{D}_2$ to its image in $\mathcal{D}_{\mathcal{T}}$.

Then $\phi(\gamma_1)$ and $\phi(\gamma_2)$ are two edge-paths on the boundary of $c_{1,2}$ with distinct sets of edges (see Figure \ref{Figure:Gluing_simple_cycles}), and moreover:

$$|\phi(\gamma_1)| + |\phi(\gamma_2)| \geq 4.$$ 
\end{lem}

\begin{figure}[h]
\centering
\begin{tikzpicture}[scale=0.85]

\draw[fill = gray!30] (-12,7) -- (-10,7) -- (-9.0582,8.2717) -- (-10,9.5) -- (-12,9.5) -- (-12.8449,8.2547) -- cycle;

\draw  plot[smooth, tension=.7] coordinates {(-12.8389,8.2487) (-14.9958,7.0944) (-16.6147,5.8556) (-16.9386,4.7464) (-16.423,4.0252)};
\draw  plot[smooth, tension=.7] coordinates {(-10.0084,6.9934) (-9.4585,5.4563) (-8.7259,4.4318) (-8.2365,3.9571)};
\draw  plot[smooth, tension=.7] coordinates {(-11.9853,9.4975) (-12.3913,11.054) (-12.6098,12.0145) (-12.731,12.8883)};
\draw  plot[smooth, tension=.7] coordinates {(-9.0417,8.2673) (-5.9204,9.3888) (-4.9315,11.0822) (-5.446,12.7183)};
\node at (-12.4743,3.9494) {$\dots$};
\node at (-8.9631,12.7794) {$\dots$};
\node at (-10.3044,7.6775) {$c_{1,2}$};
\draw[very thick, blue] (-12.3904,7.563) -- (-11.0043,8.1452) node (v1) {} -- (-10.9727,6.9835);
\draw[very thick, blue]  plot[smooth, tension=.7] coordinates {(-12.4029,7.5656) (-13.1111,7.187) (-15.6309,5.7938) (-16.2381,4.9359) (-15.8342,4.4247)};
\draw[very thick, blue]  plot[smooth, tension=.7] coordinates {(-10.9729,6.9754) (-10.38,5.751) (-10.0683,5.2329) (-9.6768,4.6636) (-9.2689,4.1907) (-8.8622,3.7765)};

\draw[very thick, red] (-11,9.5) -- (-11.0044,8.1739) -- (-9.5272,8.8736);
\draw[very thick, red]  plot[smooth, tension=.7] coordinates {(-9.5253,8.8642) (-7.4476,9.3231) (-6.1879,10.3071) (-6,11.5) (-6.2615,12.5692)};
\draw[very thick, red]  plot[smooth, tension=.7] coordinates {(-11.0024,9.4972) (-11.4882,11.6574) (-11.6522,12.4274) (-11.7055,12.9841)};
\node at (-12.4318,4.7103) {$\mathcal{D}_1$};
\node at (-9,12) {$\mathcal{D}_2$};
\node at (-10.9136,11.9239) {$\mathcal{C}_2$};
\node at (-14.9189,5.3833) {$\mathcal{C}_1$};
\draw[fill = black]  (-11.0057,8.1303) circle (0.08);
\node[left] at (-11.0234,8.2517) {$v$};

\draw[purple, thick] (-12.8371,8.2547) -- (-11.9992,6.9995) -- (-10.0055,6.9956);
\draw[purple, thick] (-12,9.5) -- (-10,9.5) -- (-9.0508,8.2821);

\draw[fill = black]  (-12.3708,7.5729) circle (0.05);
\draw[fill = black]  (-10.9665,7.0203) circle (0.05);
\draw[fill = black]  (-9.5016,8.8643) circle (0.05);
\draw[fill = black]  (-11.0019,9.5006) circle (0.05);
\node at (-12.3536,7.7796) {$x$};
\node at (-11.1819,7.2535) {$y$};
\node at (-9.6644,8.5513) {$x^a$};
\node at (-11.2863,9.2322) {$y^a$};
\node[below] at (-12,7) {$\phi(\gamma_1)$};

\node[above ] at (-10,9.5) {$\phi(\gamma_2)$};

\end{tikzpicture}
\caption{The common corner $c_{1,2}$ of two diagrams $\mathcal{D}_1$ and $\mathcal{D}_2$ collared by hypergraph realisations of simple-cycles $\mathcal{C}_1$ and $\mathcal{C}_2$ of a tree of loops $\mathcal{T}$.}
\label{Figure:Gluing_simple_cycles}
\end{figure}
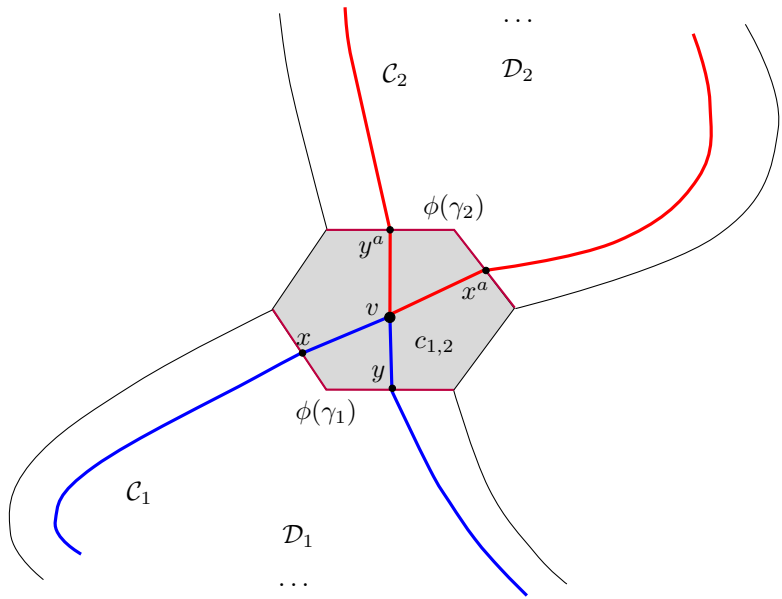

\begin{proof}
Let $v$ be the vertex of $\mathcal{T}$ contained in the 2-cells $c_{1}$ and $c_2$. There are exactly two edges $\{e_1, e_2\}$ of $\lambda$ that are contained in $c_{1,2}$, and $v$ is their point of intersection (three edges of $\lambda$ cannot intersect in a single point, by our assumption on $\lambda$). Let $x$ and $y$ be the two endpoints of edges $\{e_1, e_2\}$ that belong to the hypergraph realization of $\mathcal{C}_1$. Then the boundary points $x^a$ and $y^a$ that are antipodal to $x$ and $y$ respectively, belong to the hypergraph realization of $\mathcal{C}_2$ (see Figure \ref{Figure:Gluing_simple_cycles}). By the fact that $c_1$ and $c_2$ are shell corners, both paths $\phi(\gamma_1)$ and $\phi(\gamma_2)$ have length at most 3. 

First, consider the case where $x$ and $y$ are not midpoints of two consecutive edges of $c_{1,2}$. In that case $|\phi(\gamma_1)| = |\phi(\gamma_2)| = 3$ and edge-paths $\phi(\gamma_1)$ and $\phi(\gamma_2)$ are antipodal. Hence $\phi(\gamma_1)$ and $\phi(\gamma_2)$ have disjoint sets of edges.

Now consider the case where $x$ and $y$ are midpoints of two consecutive edges of $c_{1,2}$. Let $e_x$ and $e_y$ be the boundary edges of $c_{1,2}$ containing $x$ and $y$ respectively. Let $L_1$ be the ladder of the diagram $\mathcal{D}_1$. Let $P_1$ be the edge-path in $L_1$ along which $L_1$ is glued to the disc basis of $\mathcal{D}_1$. By the definition of the ladder, in the complex $L_1$ the 2-cell $c_1$ is joined with the rest of $L_1$ along only two edges: $e_x$ and $e_y$. Hence, $P_1 \cap c_1$ must have its ends in the vertices of $e_x \cup e_y$. Therefore, $P_1$ cannot contain any boundary edges of $c_1$ different then $e_x$ and $e_y$, since otherwise it would have length larger than three, meaning that $c_{1,2}$ is not a shell corner. Hence, $\gamma_1$ consists of two edges: $e_x$ and $e_y$. Analogously, we prove that $\gamma_2$ consists of two  edges: $e_{x^a}$ and $e_{y^a}$ that contain $x^a$ and $y^a$. Hence, $\phi(\gamma_1)$ and $\phi(\gamma_2)$ have disjoint sets of edges.

The 2-cell $c_{1,2}$ is glued to the each of the diagrams $\phi(\mathcal{D}_1)$ and $\varphi(\mathcal{D}_2)$ along at least two edges, so $|\phi(\gamma_1)| + |\phi(\gamma_2)| \geq 4.$
\end{proof}

\begin{cor}\label{Corollary:Component_is_with_small_legs_hex}
Each component of the tree of diagrams in the hexagonal model is a reduced diagram with 36-small hull.
\end{cor}

\begin{proof}
Let $\mathcal{D}_{\mathcal{T}}$ be a tree of diagrams collared by a tree of loops $\mathcal{T}$ and let $\mathcal{D}$ be a component of $\mathcal{D}_{\mathcal{T}}$. Moreover, let $\mathcal{HR}(\mathcal{D})$ be the set of hypergraph realizations of all simple cycles of $\mathcal{T}$ that are collaring diagrams forming $\mathcal{D}$.

First, we will prove that $\mathcal{D}$ is a reduced diagram. By the construction, every diagram collared by an element of $\mathcal{HR}(\mathcal{D})$ is reduced. If two such diagrams: $\mathcal{D}_1$ and $\mathcal{D}_2$ share a 2-cell $c$, then by Lemma \ref{Lemma:Gluing_simple_cycles_hex} the complex $(\mathcal{D}_1 - c) \cap (\mathcal{D}_2 - c)$ contains no edges, so there is no reduction pair in $\mathcal{D}$.  

Now, we will construct a disc basis of $\mathcal{D}$. For every element $\mathcal{C} \in \mathcal{HR}(\mathcal{D})$ let $A_{\mathcal{C}}$ be the disc basis of the diagram collared by $\mathcal{C}$. If two elements $\mathcal{C}_1$ and $\mathcal{C}_2$ of $\mathcal{HR}(\mathcal{D})$ intersect, then there exists in $\mathcal{D}$ a common 2-cell $c_{(\mathcal{C}_1, \mathcal{C}_2)}$ of diagrams collared by $\mathcal{C}_1$ and $\mathcal{C}_2$. Let $\gamma_{(\mathcal{C}_1, \mathcal{C}_2)}$ be any simple edge-path on the boundary of $c_{(\mathcal{C}_1, \mathcal{C}_2)}$ joining $A_{\mathcal{C}_1}$ with $A_{\mathcal{C}_2}$.

We define the following diagram. 

$$A := \left( \bigcup_{\mathcal{C} \in \mathcal{HR}(\mathcal{D})} A_{\mathcal{C}} \right) \cup \left( \bigcup_{ \mathcal{C}_i, \mathcal{C}_j \in \mathcal{HR}(\mathcal{D}); \mathcal{C}_i \cap \mathcal{C}_j \neq \emptyset  } \gamma_{(\mathcal{C}_i, \mathcal{C}_j)} \right).$$

In other words: $A$ is formed as a union of all disc basis of diagrams contained in $\mathcal{D}$ and short edge-paths joining them (see Figure \ref{Figure:Disc_basis_tree_of_diagrams}).

\begin{figure}[h]
\centering
\begin{tikzpicture}[scale=0.95]

\draw[fill = gray!30]  plot[smooth cycle, tension=.7] coordinates {(-1.3261,-1.2163) (-2.8221,-1.3582) (-3.5393,-0.2461) (-2.2115,0.8832) (-0.9321,-0.0735)};
\draw  plot[smooth, tension=.7] coordinates {(-1.5506,0.6536) (-0.8302,1.0062) (-0.6491,1.3952)};
\draw[fill = gray!30]  plot[smooth cycle, tension=.7] coordinates {(-0.9584,1.7315) (-1.1803,3.0361) (-0.2508,3.6684) (0.7495,3.1173) (0.7438,1.6575) (-0.1405,1.1692)};
\draw[fill = gray!30]  plot[smooth cycle, tension=.7] coordinates {(-4,1.5) (-3.0238,2.3735) (-3.3273,3.7481) (-5.1378,4.2862) (-5.9253,2.7629) (-5.263,1.5781)};
\draw  plot[smooth, tension=.7] coordinates {(-4.255,1.4208) (-3.5137,0.9811) (-3.2378,0.2914)};
\draw[fill = gray!30]   plot[smooth cycle, tension=.7] coordinates {(-3.3362,4.8378) (-3.5,5.5) (-2.8516,5.952) (-2.2672,5.4599) (-2.5929,4.747)};
\draw  plot[smooth, tension=.7] coordinates {(-3.9507,4.111) (-3.9165,4.6674) (-3.4813,5.0482)};
\draw[fill = gray!30]  plot[smooth cycle, tension=.7] coordinates {(-5.6572,5.0562) (-5.5602,5.7382) (-6.2373,6.1693) (-6.8243,5.5854) (-6.5027,4.8475)};
\draw  plot[smooth, tension=.7] coordinates {(-5.5,4) (-5.5462,4.5648) (-6.0187,4.8428)};
\node at (-4.6257,2.794) {$A_{\mathcal{C}_1}$};
\node at (-6.215,5.5834) {$A_{\mathcal{C}_2}$};
\node at (-2.9081,5.4087) {$A_{\mathcal{C}_3}$};
\node at (-2.2115,-0.4884) {$A_{\mathcal{C}_4}$};
\node at (-0.1333,2.4393) {$A_{\mathcal{C}_5}$};
\node at (-6,4.5) {$\gamma_{\mathcal{C}_1, \mathcal{C}_2}$};
\node at (-3.4644,4.3934) {$\gamma_{\mathcal{C}_1, \mathcal{C}_3}$};
\node at (-4,1) {$\gamma_{\mathcal{C}_1, \mathcal{C}_4}$};
\node at (-0.4027,0.7498) {$\gamma_{\mathcal{C}_4, \mathcal{C}_5}$};

\end{tikzpicture}
\caption{The diagram $A$, being the disc basis of $\mathcal{D}$.}
\label{Figure:Disc_basis_tree_of_diagrams}
\end{figure}
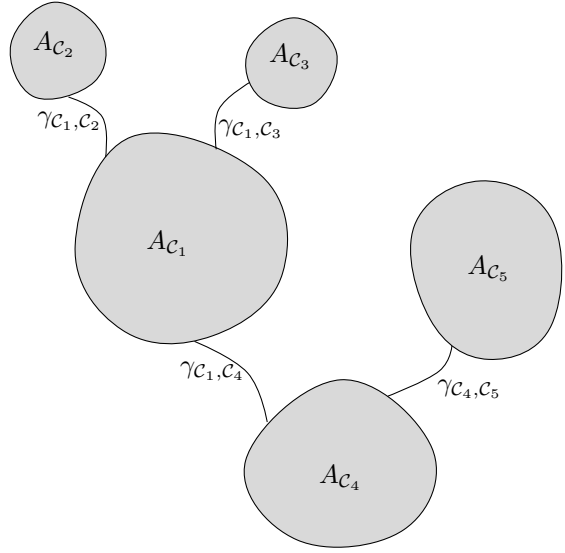

Now, we will prove that $A$ is a disc diagram. Let $\mathcal{D}_{\mathcal{T}}$ denote the dual graph to $\mathcal{T}$. For every $\mathcal{C} \in \mathcal{HR}(\mathcal{D})$ let us contract $A_{\mathcal{C}}$ to a single point. This will result in a space that is homeomorphic to the simplicial realization of a subgraph of $\mathcal{D}_{\mathcal{T}}$, thus to a tree. Therefore, $A$ is homotopically equivalent to a tree, so $A$ is a disc diagram. 

Now, we will analyze the ladder $L$ of the diagram $\mathcal{D}$. 
Let $\gamma$ be any edge-path with ends on the boundary of $A$, cutting $A$ into two pieces: $A'$ and $A''$. We define $\mathcal{D}'$ to be the subcomplex of $\mathcal{D}$ consisting of $A'$ and all 2-cells $c$ of $L$ such that $c \cap A' \neq \emptyset$. We define $\mathcal{D}''$ to be the subcomplex of $\mathcal{D}$ consisting of $A''$ and all 2-cells $c$ of $L$ such that $c \cap A' = \emptyset$.

To estimate $|\parti \mathcal{D}'| + |\parti \mathcal{D}''| - |\parti \mathcal{D}|$ we will analyze the projection map $p: \mathcal{D}' \cup \mathcal{D}'' \to \mathcal{D}$. If an edge $e$ of $\mathcal{D}$ satisfies $|p^{-1}(e)| > 1$ it means that $e$ belongs to $\gamma$ or $e$ belongs to a 2-cell $c$ of $L$ such that $c \cap A' \neq \emptyset$ and $c$ shares an edge with $A''$. Each such 2-cell $c$ shares a vertex with $\gamma$, contains a boundary edge of $A''$ and a shares a vertex with $A'$. Therefore, $c$ contains a boundary edge of $A''$ that is adjacent to the common vertex of $A'$ and $A''$. Moreover, at most two 2-cells of $L$ can share a boundary edge of $A$. Summing up, there is $|\gamma| + 1$ vertices of gamma, thus at most $|\gamma| + 1$ boundary edges of $A''$ that share a vertex with $\gamma$. Hence, there is at most $2|\gamma | + 2$ such 2-cells $c$ (introduced above). Therefore,  we can estimate the number of edges $e$ such that $|p^{-1}(e)| > 1$, by $2|\gamma| + 6(2|\gamma|+2) \leq 6(3|\gamma| + 2)$.  

By the construction of tree of diagrams there are no edges $e$ of $\mathcal{D}$, such that $|p^{-1}(e)| > 2$, so we continue estimation:

$$|\parti \mathcal{D}| \geq |\parti \mathcal{D}'| + |\parti \mathcal{D}''| - 12(3|\gamma| +2) \geq  |\parti \mathcal{D}'| + |\parti \mathcal{D}''| - 36|\gamma| - 36,$$
which shows that $\mathcal{D}$ is a diagram with 36-small hull if we define its hull to be $L$, and its disc basis to be $A$.
\end{proof}

\begin{theo}\label{Theorem:Tree_of_diagrams_boundary_hex}
Let $\lambda$ be an admissible intra-segment and suppose that the first and last vertex of $\lambda$ coincide. Let $\mathcal{D}_{\mathcal{T}}$ be a tree of diagrams of $\lambda$. Then w.o.p. $\mathcal{D}_{\mathcal{T}}$ is a reduced connected diagram with 36-small hull and

$$|\tilde{\partial} \mathcal{D}_{\mathcal{T}}| \leq 2|\mathcal{D_{\mathcal{T}}}| + 2. $$

Moreover, the main 2-cell of $\mathcal{D}_{\mathcal{T}}$ is a shell corner with exactly four external edges and also every 2-cell of $\mathcal{D}_{\mathcal{T}}$, different than the main 2-cell, contributes at most 2 to the generalized boundary length of $\mathcal{D}_{\mathcal{T}}$ (see Figure \ref{Figure:Tree_of_diagrams_boundary_figure}).
\end{theo}

\begin{figure}[h]
\centering
\begin{tikzpicture}[scale=0.45]

\draw[fill = gray!30] (0.1624,-0.1301) -- (0.1465,1.0415) -- (1.007,1.5178) -- (1.853,1.0031) -- (1.856,-0.3089) -- (0.8886,-0.9289) -- cycle;
\node at (0.5,7.5) {};
\node at (4,12.5) {};

\node at (11.5,9) {};

\draw[very thick, blue]  plot[smooth, tension=.7] coordinates {(1.6768,1.4088) (5.9836,4.6734) (18.2056,5.1237) (19.3691,2.2456) (15.5113,4.6338) (11.6841,7.2669) (10.735,10.9104) (13.7049,13.7272) (15.8175,15.9316) (13.9192,16.8502) (12.7557,15.258) (14.5928,13.7272) (19.9508,13.2067) (21.3592,14.5538) (18.8486,14.7988) (17.7157,10.6654) (11.2248,10.3286) (1.7028,10.1449) (-0.716,12.0126) (1.5497,13.6047) (2.5907,11.2165) (-1,7) (1.3635,-0.6429)};
\draw[very thick, blue] (0.5932,-0.5826) -- (1.6826,1.4355);
\draw (13.1545,5.6571) -- (13.5992,6.1408) -- (14.6953,5.7921) -- (14.9726,5.1601) -- (14.5292,4.8122) -- (13.3501,5.1444) -- cycle;
\draw (9.9334,10.6221) -- (10.4475,10.9734) -- (11.2582,10.7685) -- (11.1441,10.0732) -- (10.7499,9.7353) -- (9.9335,9.7494) -- cycle;
\draw (13.4563,13.6511) -- (13.4298,13.9958) -- (13.8806,14.3007) -- (14.517,14.1284) -- (14.57,13.85) -- (14.0794,13.5318) -- cycle;
\draw (17.9367,13.3621) -- (18.3785,13.6724) -- (18.9888,13.567) -- (19.0708,13.0205) -- (18.7155,12.6652) -- (17.9913,12.7745) -- cycle;
\draw (0.8634,9.8039) -- (0.7274,10.15) -- (1.4444,10.6073) -- (2.2355,10.459) -- (2.3714,10.1747) -- (1.9353,9.5744) -- cycle;
\draw  plot[smooth, tension=.7] coordinates {(1.8367,1.0306) (5.5714,3.8331) (13.3725,5.1607)};
\draw  plot[smooth, tension=.7] coordinates {(14.5274,4.7903) (18.6104,1.6892) (20.7247,3.7661) (19.5032,5.32) (17.0584,5.9276) (14.6743,5.7763)};
\draw  plot[smooth, tension=.7] coordinates {(13.5997,6.1418) (11.8682,7.7074) (10.7503,9.7306)};
\draw  plot[smooth, tension=.7] coordinates {(11.1485,10.0869) (16.047,9.8918) (18.5731,10.7357) (18.6998,12.6731)};
\draw  plot[smooth, tension=.7] coordinates {(19.0744,13.0177) (21.2311,13.1402) (21.9997,14.4725) (20.6163,15.5229) (18.63,15.2861) (18.37,13.6823)};
\draw  plot[smooth, tension=.7] coordinates {(17.9302,13.3701) (16.257,13.5715) (14.5668,13.8417)};
\draw  plot[smooth, tension=.7] coordinates {(14.5229,14.128) (16.1161,15.7794) (15.1123,17.0423) (13.1037,16.8899) (12.3988,15.1123) (13.4221,13.992)};
\draw  plot[smooth, tension=.7] coordinates {(10.443,10.9741) (11.3108,12.3211) (13.4544,13.6487)};
\draw  plot[smooth, tension=.7] coordinates {(9.925,10.6246) (5.9557,10.3501) (2.3919,10.1939)};
\draw  plot[smooth, tension=.7] coordinates {(2.2439,10.4845) (3.2261,12.0793) (1.6701,13.9465) (-0.761,13.0324) (-0.7708,10.8832) (0.7397,10.1424)};
\draw  plot[smooth, tension=.7] coordinates {(0.8517,9.7944) (-1.0409,8.4661) (-1.1597,4.9544) (0.1467,1.0449)};
\node at (5.772,9.1247) {$\lambda$};
\node at (2.1799,-1.1493) {$c_{main}$};
\node at (2.5,9) {$c_1$};
\node at (11.5,11) {$c_2$};
\node at (14.2803,13.0384) {$c_3$};
\node at (19.4446,12.3891) {$c_4$};
\node at (14.3513,6.574) {$c_5$};

\end{tikzpicture}
\caption{An example shape of a tree of diagrams $\mathcal{D}_{\mathcal{T}}$ of $\lambda$ in the hexagonal model. The main 2-cell $c_{main}$ is marked with gray and all other corners are labeled as $c_1, c_2, \dots, c_5$.}
\label{Figure:Tree_of_diagrams_boundary_figure}
\end{figure}
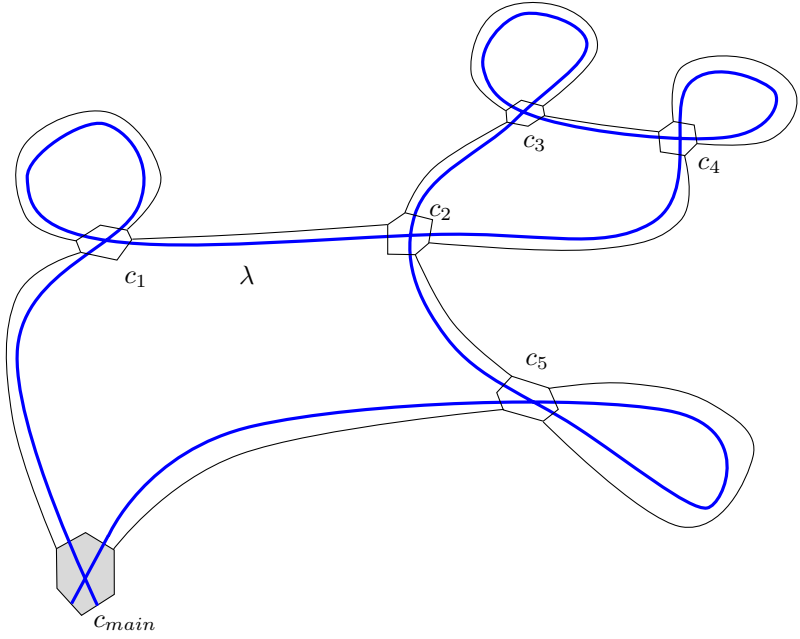

Before we prove Theorem \ref{Theorem:Tree_of_diagrams_boundary_hex} we will introduce the following definition. 

\begin{defi}[boundary deviation for hexagonal model]
Let $\mathcal{D}$ be a 2-dimensional complex diagram. For a 2-cell $c$ of $\mathcal{D}$ we define its \textit{boundary deviation} to be
$$\text{number of external edges of } c  - 2.$$

We define the \textit{boundary deviation} of $\mathcal{D}$ in the following way: 
$$\text{dev}(\mathcal{D}) := |\tilde{\partial} \mathcal{D}| - 2|\mathcal{D}|.$$

For a collection of diagrams $\{\mathcal{D}_1, \mathcal{D}_2, \dots, \mathcal{D}_n \}$ we define its \textit{total boundary deviation} to be 
$$\text{dev}_{T} := \sum_{i=1}^n \text{dev}(\mathcal{D}_i).$$ 
\end{defi}

The motivation to define the boundary deviation comes from the fact that in the hexagonal model, for densities $< \frac{1}{3}$ all diagrams with the boundary deviation not larger than 0 violate the Generalized Isoperimetric Inequality (Theorem \ref{Theorem:Generalized_Isoperimetric_Inequality}).

\begin{proof}[Proof of Theorem \ref{Theorem:Tree_of_diagrams_boundary_hex}]
Let $\{ \mathcal{D}_1, \mathcal{D}_2, \dots, \mathcal{D}_k \}$ be components of $\mathcal{D}_{\mathcal{T}}$.

First, we will show, that the total boundary deviation of $\mathcal{D}_{\mathcal{T}}$ satisfies 
\begin{equation}\label{Eq:total_deviaton}
\text{dev}_T = \sum_{i=1}^{k} \text{dev}(\mathcal{D}_i) \leq 2.
\end{equation}

We will prove (\ref{Eq:total_deviaton}) by the analysis of the boundary deviation of each 2-cell of $\mathcal{D}_{\mathcal{T}}$. Clearly, the total boundary deviation of a component of a tree of diagrams is not larger than the sum of the boundary deviations of its 2-cells. Let $c$ be a 2-cell of $\mathcal{D}$. Then we have 2 possibilities: 

\begin{enumerate}[a)]
\item \label{Case:Not_in_ladder} $c$ does not belong to any of the ladders of the diagrams $\{ \mathcal{D}_1, \mathcal{D}_2, \dots, \mathcal{D}_k \}$.
\item \label{Case:In_ladder} $c$ belongs to the ladder of one of the diagrams $\{ \mathcal{D}_1, \mathcal{D}_2, \dots, \mathcal{D}_k \}$.
\end{enumerate}

Case \ref{Case:Not_in_ladder}). If $c$ does not belong to any of these ladders then it is an internal 2-cell, thus, $c$ contributes at most -2 to the total boundary deviation of $\mathcal{D}_{\mathcal{T}}$.

Case \ref{Case:In_ladder}). By Remark \ref{Remark:The_main_is_not_twin} the main 2-cell of $\mathcal{D}_{\mathcal{T}}$ is not a twin and every corner different then the main 2-cell is a twin. Therefore, every 2-cell $c$ of $\mathcal{D}_{\mathcal{T}}$ that lies in the ladder of a connected component of $\mathcal{D}_{\mathcal{T}}$ fits in exactly one of the following categories:
\begin{enumerate}[I.]

\item \label{Case:I} $c$ is the main 2-cell of $\mathcal{D}_{\mathcal{T}}$
\item \label{Case:II} $c$ is not a corner of a component of $\mathcal{D}_{\mathcal{T}}$ and not the main 2-cell of $\mathcal{D}_{\mathcal{T}}$
\item \label{Case:III} $c$ is a twin and $c$ is  identified with its twin partner.
\item \label{Case:IV} $c$ is a twin and $c$ is \textbf{not} identified with its twin partner.
\end{enumerate}

Now, we will analyze how a 2-cell from each category contributes to the total boundary deviation of $\mathcal{D}_{\mathcal{T}}$.

Case \ref{Case:I}. Let $c$ be the main 2-cell of $\mathcal{D}_{\mathcal{T}}$. Such a 2-cell gives contribution at most four to the generalized boundary length of a component of $\mathcal{D}_{\mathcal{T}}$ to which it belongs. Therefore, it contributes at most 2 to the total boundary deviation of $\mathcal{D}_{\mathcal{T}}$.

Case \ref{Case:II}. If $c$ is not a corner of a component of $\mathcal{D}_{\mathcal{T}}$ then a hypergraph intra-segment collaring  $\mathcal{D}_{\mathcal{T}}$ joins its antipodal points, so $c$ contributes at most 2 to the generalized boundary length of the component of $\mathcal{D}_{\mathcal{T}}$ to which it belongs. Therefore, $c$ contributes at most 0 to the total boundary deviation of $\mathcal{D}_{\mathcal{T}}$.

Case \ref{Case:III}. If a 2-cell $c$ is identified with its twin partner, then by Lemma \ref{Lemma:Gluing_simple_cycles_hex} it may have at most 2 external edges, since it is glued to the diagram along two edge-path, that do not share an edge and each of them has length at least 2 (these edge-path are called: $\phi(\gamma_1)$ and $\phi(\gamma_2)$ in the statement of Lemma \ref{Lemma:Gluing_simple_cycles_hex}). Hence, the 2-cell $c$ gives contribution at most 2 to the generalized boundary length of the component of $\mathcal{D}_{\mathcal{T}}$ containing $c$. Therefore, $c$ contributes at most 0 to the total boundary deviation of $\mathcal{D}_{\mathcal{T}}$.

Case \ref{Case:IV}. We will consider $c$ simultaneously with its twin partner $c'$. If $c$ is not identified with its twin partner, then $c$ and $c'$ belong to different components of $\mathcal{D}_{\mathcal{T}}$. Moreover, at least one element of the pair $\{c, c'\}$ is not a shell corner of its component of $\mathcal{D}_{\mathcal{T}}$, say this element is $c$.

Let $\mathcal{D}_c$ and $\mathcal{D}_{c'}$ be the connected components of $\mathcal{D}_{\mathcal{T}}$ containing $c$ and $c'$ respectively. Let $\lambda_c$ and $\lambda_{c'}$ be the multi-cycles of intra-segments collaring $\mathcal{D}_c$ and $\mathcal{D}_{c'}$ respectively. Let $x$ and $y$ be the two boundary points of $c$ joint by $\lambda_c$.  We will consider two situations:

\textbf{Situation, where $x$ and $y$ are midpoints of two consecutive edges of $c$}. In this case, by the fact that $c$ is not a shell corner, we know
 that $c$ has no external edges (see Figure \ref{Figure:Siatuation_x_y_consecutive}). 

\begin{figure}[h]
\centering
\begin{tikzpicture}[scale=0.95]

\draw (-0.5,-1.5) -- (1,-1.5) -- (1.5,-0.5) -- (1,0.5) -- (-0.5,0.5) -- (-1,-0.5)--cycle;
\draw  plot[smooth, tension=.7] coordinates {(1,0.5) (0.4504,1.9103) (-0.6751,2.6405) (-2.1728,1.8336) (-3.5304,-0.6633) };
\draw  plot[smooth, tension=.7] coordinates {(1,0.5) (2.2966,1.1619) (3.3124,0.0626) (3.1828,-1.6359) };
\draw[very thick, blue]  plot[smooth, tension=.7] coordinates {(0.2494,0.4851) (0.0442,1.5973) (-1,2) (-2.1475,0.7394) (-2.6947,-0.4398)};
\draw[very thick, blue] (0.2766,-0.4981) -- (0.2446,0.4691);
\draw[very thick, blue]  plot[smooth, tension=.7] coordinates {(1.2129,0.038) (2.3151,0.5467) (2.9393,-0.1624) (2.816,-1.2029)};
\draw[very thick, blue] (0.2754,-0.5117) -- (1.1936,0.0356);
\node at (0,-1) {$c$};
\draw[fill=black]  (0.2506,0.5011) circle (0.05);
\draw[fill=black]  (1.2236,0.044) circle (0.05);
\node[above left] at (0.2516,0.483) {$x$};
\node[right] at (1.4655,-0.0549) {$y$};
\node at (-1.8616,0.3335) {$\lambda_{c}$};
\node at (-1.5,-1) {$\mathcal{D}_{c}$};
\node at (-3,-1) {$\dots$};
\node at (2.5,-1.5) {$\dots$};

\end{tikzpicture}
\caption{If $x$ and $y$ are midpoints of consecutive edges of $c$ then $c$ has no external edges.}
\label{Figure:Siatuation_x_y_consecutive}
\end{figure}
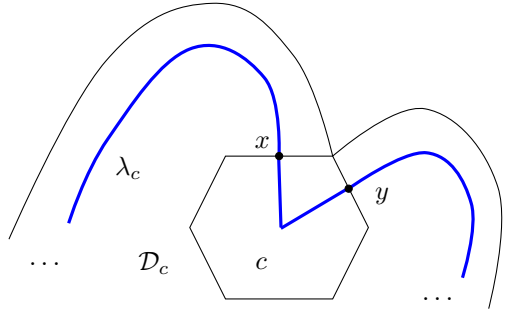
 
Therefore, the contribution to the generalized boundary length from the 2-cell $c$ is at most 0, so the boundary deviation of $c$ is at most -2. 

Note, that every 2-cell of a diagram collared by a multi-cycle of intra-segments cannot have more than four external edges: for 2-cell in the ladder it results from the fact that collaring hypergraph intra-segment is an injected circle and for internal 2-cells it results by the fact, that every edge of it is glued to the ladder or other internal 2-cells. This means that the boundary deviation of every 2-cell of a diagram collared by a multi-cycle of intra-segments is at most 2. 
 
Therefore, the joint contribution to the total boundary deviation of the pair $\{c, c'\}$ is at most 0 (which also means 0 on average for this pair). 

\textbf{Situation, where $x$ and $y$ are not midpoints of two consecutive edges of $c$}. 
Since $c$ is a corner, $x$ and $y$ cannot be antipodal points on the boundary of $c$. Therefore, the edges containing $x$ and $y$ are separated by exactly one boundary edge of $c$. By the fact that $c$ is not a shell corner, we know that $c$ has at most one external edge, so $c$ contributes at most -1 to the total boundary deviation of $\mathcal{D}_{\mathcal{T}}$ (see Figure \ref{Figure:Siatuation_x_y_not_consecutive} a)).

\begin{figure}[h]
\centering
\begin{tikzpicture}[scale=0.95]

\draw (-0.5,-1.5) -- (1,-1.5) -- (1.5,-0.5) -- (1,0.5) -- (-0.5,0.5) -- (-1,-0.5)--cycle;
\draw  plot[smooth, tension=.7] coordinates {(-0.5,0.5) (-1.3629,1.3087) (-2.3659,1.2296) (-2.9774,0.5422) (-3.5304,-0.6633) };
\draw  plot[smooth, tension=.7] coordinates {(1,0.5) (2.2966,1.1619) (3.3124,0.0626) (3.1828,-1.6359) };
\draw[very thick, blue]  plot[smooth, tension=.7] coordinates {(-0.7556,0.0331) (-1.284,0.7584) (-2.1252,0.829) (-2.67,0.1949) (-3.0165,-0.6487)};
\draw[very thick, blue] (0.2766,-0.4981) -- (-0.7402,0.0173);
\draw[very thick, blue]  plot[smooth, tension=.7] coordinates {(1.2129,0.038) (2.3151,0.5467) (2.9393,-0.1624) (2.816,-1.2029)};
\draw[very thick, blue] (0.2754,-0.5117) -- (1.1936,0.0356);
\node at (0,-1) {$c$};
\draw[fill=black]  (-0.7376,-0.0041) circle (0.05);
\draw[fill=black]  (1.2236,0.044) circle (0.05);
\node[left] at (-0.7389,-0.0117) {$x$};
\node at (1.4655,-0.0549) {$y$};
\node at (-2.1077,0.3719) {$\lambda_{c}$};
\node at (-1.5,-1) {$\mathcal{D}_{c}$};
\node at (-3,-1) {$\dots$};
\node at (2.5,-1.5) {$\dots$};

\begin{scope}[shift={(-4,-3)}]

\draw (-0.5,-1.5) -- (1,-1.5) -- (1.5,-0.5) -- (1,0.5) -- (-0.5,0.5) -- (-1,-0.5)--cycle;
\draw  plot[smooth, tension=.7] coordinates {(-1,-0.5) (-1.584,-0.484) (-2.2369,-0.98) (-2.4894,-1.6653) (-2.4373,-2.2044) };
\draw  plot[smooth, tension=.7] coordinates {(1.5,-0.5) (2.1544,-0.3809) (3,-1) (3.2685,-1.9132) };
\draw[very thick, red]  plot[smooth, tension=.7] coordinates {(-0.7755,-0.9685) (-1.3067,-0.9344) (-1.8043,-1.2811) (-2.0389,-1.7972) (-2.078,-2.2062)};
\draw[very thick, red] (0.2766,-0.4981) -- (1.2302,-1.0194);
\draw[very thick, red]  plot[smooth, tension=.7] coordinates {(1.2361,-1.025) (1.9144,-0.7652) (2.5957,-1.2141) (2.8459,-2.0007)};
\draw[very thick, red] (0.2754,-0.5117) -- (-0.7823,-0.9738);
\node at (0,-1) {$c'$};
\draw[fill=black]  (1.2454,-1.0125) circle (0.05);
\draw[fill=black]  (-0.7678,0.0144) circle (0.05);
\draw[fill=black]  (1.2439,-0.0043) circle (0.05);
\draw[fill=black]  (-0.7785,-0.9687) circle (0.05);
\node at (1.4255,-1.2022) {$x'_a$};
\node at (-0.94,-1.2291) {$y'_a$};
\node at (-1.7301,-1.928) {$\lambda_{c'}$};
\node at (0.0248,-2.1469) {$\mathcal{D}_{c'}$};
\node at (-2.1782,-2.5112) {$\dots$};
\node at (3.1036,-2.3046) {$\dots$};
\node at (-0.9987,0.1594) {$x'$};
\node at (1.4702,0.2236) {$y'$};
\end{scope}

\begin{scope}[shift={(4,-4.)}]

\draw (-0.5,-1.5) -- (1,-1.5) -- (1.5,-0.5) -- (1,0.5) -- (-0.5,0.5) -- (-1,-0.5)--cycle;
\draw  plot[smooth, tension=.7] coordinates {(-0.5,-1.5) (-1.313,-1.5919) (-1.9327,-1.2503) (-2.3679,-0.5238) (-2.5125,0.4721) };
\draw  plot[smooth, tension=.7] coordinates {(1,-1.5) (1.8738,-1.4249) (2.6617,-0.7146) (3,0.5) };
\draw[very thick, red]  plot[smooth, tension=.7] coordinates {(-0.7755,-0.9685) (-1.2882,-1.2272) (-1.8221,-0.8804) (-2.167,-0.1513) (-2.2315,0.4672)};
\draw[very thick, red] (0.2766,-0.4981) -- (1.2302,-1.0194);
\draw[very thick, red]  plot[smooth, tension=.7] coordinates {(1.2361,-1.025) (1.7425,-1.1296) (2.2428,-0.7481) (2.6531,0.4598)};
\draw[very thick, red] (0.2754,-0.5117) -- (-0.7823,-0.9738);
\node at (0,-1) {$c'$};
\draw[fill=black]  (1.2454,-1.0125) circle (0.05);
\draw[fill=black]  (-0.7678,0.0144) circle (0.05);
\draw[fill=black]  (1.2439,-0.0043) circle (0.05);
\draw[fill=black]  (-0.7785,-0.9687) circle (0.05);
\node at (1.123,-0.6521) {$x'_a$};
\node at (-1.1342,-0.8728) {$y'_a$};
\node at (-1.798,-0.0518) {$\lambda_{c'}$};
\node at (0.322,1.0027) {$\mathcal{D}_{c'}$};
\node at (-2.3138,0.8338) {$\dots$};
\node at (2.7485,0.7834) {$\dots$};
\node at (-0.9987,0.1594) {$x'$};
\node at (1.4702,0.2236) {$y'$};
\end{scope}

\node at (0.2121,-1.8634) {a)};
\node at (-3.6515,-6.2337) {b)};
\node at (4.2475,-6.274) {c)};

\end{tikzpicture}
\caption{The case, where $x$ and $y$ are not midpoints of consecutive edges of $c$: a) the 2-cell $c$ with the collaring multi-cycle of intra-segments $\lambda_c$, b) and c) are two possibilities of how $c'$ is glued to the diagram $\mathcal{D}_{c'}$. In both cases: b) and c), the 2-cell $c'$ has at most 3 external edges.}
\label{Figure:Siatuation_x_y_not_consecutive}
\end{figure}
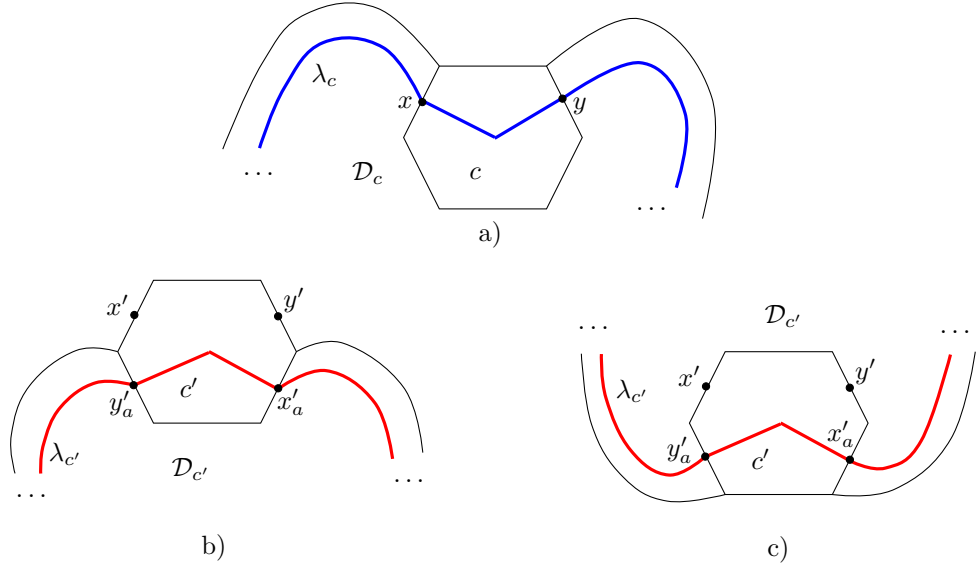

Now, let us analyze the twin partner of $c$. The 2-cell $c'$ is mapped  onto the same 2-cell of $\mathcal{D}_{\mathcal{T}}$ as $c$ under the natural combinatorial map, so let $x'$ and $y'$ be the points on the boundary of $c'$ corresponding to $x$ and $y$ respectively. Note, that the hypergraph multi-cycle $\lambda_{c'}$ collaring $\mathcal{D}_{c'}$ crosses the boundary of $c'$ in points $x'_a$ and $y'_a$ which are the antipodes of $x'$ and $y'$ respectively. Since $x$ and $y$ are not the midpoints of the consecutive edges, so are $x'_a$ and $y'_a$. Therefore, $c'$ can have at most 3 external edges, so its boundary deviation is at most 1 (see Figure \ref{Figure:Siatuation_x_y_not_consecutive} b) and c)). Therefore, the joint contribution to the total boundary deviation of the pair $\{c, c' \}$ is 0 (so 0 on average for this pair). This ends the proof of Inequality (\ref{Eq:total_deviaton}). 

Suppose, on the contrary, that $\mathcal{D}_{\mathcal{T}}$ is not connected, thus, it has at least two components. Note, that the boundary deviation of every diagram is an even number, since by Lemma \ref{Lemma:Pairity_of_the_boundary} the generalized boundary length is even. Then, by Inequality (\ref{Eq:total_deviaton}), we know that one of the components of $\mathcal{D}_{\mathcal{T}}$ must have a boundary deviation at most 0. However, this means, that such a component w.o.p. does not exist, since it violates Theorem \ref{Theorem:Generalized_Isoperimetric_Inequality} (Generalized Isoperimetric Inequality). Therefore, $\mathcal{D}_{\mathcal{T}}$ is w.o.p. connected.

Moreover, if the main 2-cell does not have four external edges, it contributes less than 2 to the total boundary deviation, which means that the total deviation of $\mathcal{D}_{\mathcal{T}}$ is 0 since it must be even. In this case by Theorem \ref{Theorem:Generalized_Isoperimetric_Inequality} (Generalized Isoperimetric Inequality)  we conclude that such $\mathcal{D}_{\mathcal{T}}$ w.o.p. cannot exists.

By the fact that $\mathcal{D}_{\mathcal{T}}$ is connected and by Corollary \ref{Corollary:Component_is_with_small_legs_hex} we conclude that $\mathcal{D}_{\mathcal{T}}$ is a diagram with 36-small hull.
\end{proof}

\begin{cor}\label{Corollary:Bounded_size_of_tree_of_diagrams_hex}
For any fixed $\varepsilon > 0$ a tree of diagrams in the hexagonal model at density $d$ has not more than $\frac{1}{2(1-3 d - \varepsilon)}$ 2-cells. 
\end{cor}

\begin{proof}
Let $\mathcal{D}_{\mathcal{T}}$ be a tree of diagrams. By combining Theorem \ref{Theorem:Generalized_Isoperimetric_Inequality} (Generalized Isoperimetric Inequality) with Theorem \ref{Theorem:Tree_of_diagrams_boundary_hex} we obtain the following inequality:

$$6(1-2d-\varepsilon)|\mathcal{D}_{\mathcal{T}}| \leq 2 |\mathcal{D}_{\mathcal{T}}| + 2 .$$ The statement of Corollary \ref{Corollary:Bounded_size_of_tree_of_diagrams_hex} results by solving this inequality.
\end{proof}

\section{Bent walls in the hexagonal model at density $<\frac{1}{3}$} \label{Section:Bent_walls_hexagonal}

\begin{theorem}\label{Theorem:Properties_of_hyp_segment_hex}
Let $\lambda$ be a hypergraph segment in the Cayley complex $\cay$ of a random group in the hexagonal model at density $d < \frac{1}{3}$. Then w.o.p.:  

\begin{enumerate}
\item \label{Assertion:no_three_intersect} No three edges of $\lambda$ intersect,
\item \label{Assertion:not_0cycle} $\lambda$ is not a hypergraph cycle, i.e. the first and the last vertex of $\lambda$ do not coincide.
\end{enumerate}
\end{theorem}

\begin{proof}
We will prove the statement by induction on the length of $\lambda$. First, we will show the proof of Assertion (\ref{Assertion:no_three_intersect}) and then, using it, we will prove Assertion (\ref{Assertion:not_0cycle}).

For $|\lambda|=1$ the only possibility that $\lambda$ is a hypergraph cycle is that $\lambda$ is a loop, meaning that the ends of its one edge coincide. Then by Lemma \ref{Lemma:Collared_n_cycles_reduced_exists} there exists a reduced diagram $\mathcal{D}_{\lambda}$ collared by $\lambda$. The segment $\lambda$ consists of one edge, which joins the antipodal points of a 2-cell, therefore $\mathcal{D}_{\lambda}$ has the generalized boundary length at most 2. Since it has at least one 2-cell, it violates Theorem \ref{Theorem:Generalized_Isoperimetric_Inequality} (Generalized Isoperimetric Inequality). Therefore, $\lambda$ cannot be a loop, so Assertions (\ref{Assertion:no_three_intersect}) and (\ref{Assertion:not_0cycle}) are satisfied for $|\lambda|=1$.

Suppose, that for every $\lambda$ of length $\leq k$ Assertions (\ref{Assertion:no_three_intersect}) and (\ref{Assertion:not_0cycle}) hold. We will prove that for $\lambda$ of length $k+1$ these assertions are also satisfied.

\textbf{Proof of Assertion (\ref{Assertion:no_three_intersect})}
Suppose, that three edges of $\lambda$ intersect in one point, call it $x$. Let $c$ be the 2-cell of $\cay$ such that $x$ is its middle. Then there exist two shorter hypergraph segments $\Lambda_a \subset \lambda$ and $\Lambda_b \subset \lambda$ such that the first and the last edge of both $\Lambda_a$ and $\Lambda_b$ is contained in $c$ and moreover $\Lambda_a$ and $\Lambda_b$ have a common edge $e$. Let $\lambda_a$ and $\lambda_b$ be the intra-segments of $\Lambda_a$ and $\Lambda_b$ respectively. Note that the first and the last  vertex of $\lambda_a$ and $\lambda_b$ is $x$, and $\lambda_b$ prolongs $\lambda_a$ (see Figure \ref{Figure:lambda_a_lambda_b}). Since both $\Lambda_a$ and $\Lambda_b$ are strictly shorter than $\lambda$, by the inductive assumptions, we know that Assertion (\ref{Assertion:no_three_intersect}) is satisfied for $\Lambda_a$ and $\Lambda_b$. Hence, $\lambda_a$ and $\lambda_b$ are admissible, so there exist two trees of diagrams: $\mathcal{D}_a$ and $\mathcal{D}_b$ collared by intra-segments: $\lambda_a$ and $\lambda_b$ respectively. Moreover, $\mathcal{D}_a$ and $\mathcal{D}_b$ share the main 2-cell, which is $c$ (see Figure \ref{Figure:lambda_a_lambda_b}). Consider the diagram $D_{a\cup b}$ that is the defined as the identification of $\mathcal{D}_a$ and $\mathcal{D}_b$ along the 2-cell $c$. 

\begin{figure}[h]
\centering
\begin{tikzpicture}[scale=0.80]

\draw (-1.5,0) -- (-1.5,1) -- (-0.5,2) -- (0.5,1) node (v2) {} -- (0.5,0) -- (-0.5,-1) node (v1) {} -- cycle;
\draw[very thick, purple] (-1.5,0.5) -- (0.4883,0.4832);
\draw[very thick, blue] (-1.0539,1.4326) -- (-0.0107,-0.5136);
\draw[very thick, red] (-1.0748,-0.4462) -- (0.0163,1.5069);
\draw[very thick, red]  plot[smooth, tension=.7] coordinates {(-1.0681,-0.4146) (-6,-0.5) (-6.2031,-3.1809) (-7.3021,-0.8732) (-5.5,-1.5) (-3.0978,0.5709) (-1.5,0.5)};
\draw[very thick, blue]  plot[smooth, tension=.7] coordinates {(0,-0.5) (0.9766,-0.9716) (3.2442,-1.2704) (3.8004,-0.1604) (2.0456,0.4818) (0.4929,0.485)};

\draw  plot[smooth, tension=.7] coordinates {(-1.5,1)  };
\node at (-5.5547,-0.0185) {$\lambda_a$};
\node at (3,0.5) {$\lambda_b$};
\draw[fill=black]  (-0.5588,0.4909) circle (0.05);
\node at (-0.535,0.059) {$x$};
\draw  plot[smooth, tension=.7] coordinates {(-1.4857,0.98)  };
\draw  plot[smooth, tension=.7] coordinates {(v1)  };

\draw  plot[smooth, tension=.7] coordinates {(v2)  };
\node at (-3,-1) {$\mathcal{D}_a$};
\node at (2.0187,-0.4117) {$\mathcal{D}_b$};
\draw  plot[smooth, tension=.7] coordinates {(-0.5,-1)  };
\draw  plot[smooth, tension=.7] coordinates {(-1.5,1) (-3,1) (-4.5,0.5) (-7,0) (-8,-1) (-6.3965,-3.827) (-4.8923,-1.9665) (-1.4068,-1.2817) (-0.4992,-1.0267)};
\draw  plot[smooth, tension=.7] coordinates {(v2) (2.6012,1.1415) (4.3311,0.2304) (4.4926,-1.4073) (3.1433,-2.1223) (0.6061,-1.6033) (-0.5126,-1.0267)};
\node at (-0.5,1.5) {$c$};
\node at (0.0809,0.6777) {$e$};

\end{tikzpicture}
\caption{A hypergraph segment with three edges intersecting in a single point.}
\label{Figure:lambda_a_lambda_b}
\end{figure}
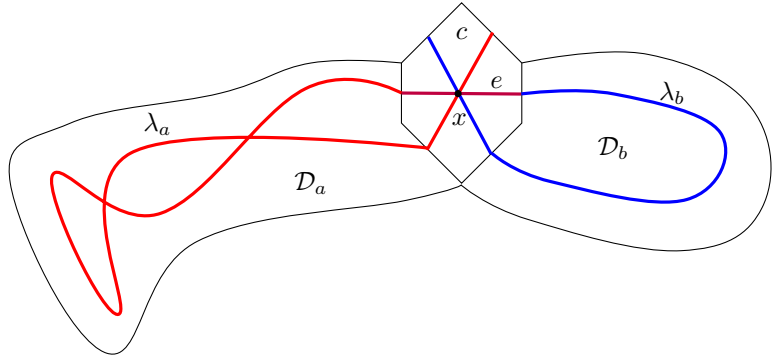

By Theorem \ref{Theorem:Tree_of_diagrams_boundary_hex} we know that the main 2-cell of a tree of diagrams, has four external edges, so we know that $c$ is glued to the diagram $(\mathcal{D}_a - c)$ along two consecutive edges $\{e_a^1, e_a^2\}$ of $c$. The same holds for the diagram $(\mathcal{D}_b - c)$: it is glued to $c$ along two consecutive edges $\{e_b^1, e_b^2 \}$. Note that, one element of the set $\{e_a^1, e_a^2\}$ contains one end of the edge $e$ and one element of $\{e_b^1, e_b^2 \}$ contains the other end of $e$. Ends of $e$ are antipodal on $c$, so sets $\{e_a^1, e_a^2\}$ and $\{e_b^1, e_b^2 \}$ are disjoint. The diagram $\mathcal{D}_{a \cup b}$ is reduced, by the fact that both $\mathcal{D}_a$ and $\mathcal{D}_b$ are reduced and there is no common edge of $c$ in $\mathcal{D}_a - c$ and $\mathcal{D}_b - c$.  Note, that the 2-cell $c$ contributes at most 2 to the generalized boundary length of $\mathcal{D}_{a \cup b}$. We also know, by Theorem \ref{Theorem:Tree_of_diagrams_boundary_hex} that every 2-cell of $\mathcal{D}_a$ and $\mathcal{D}_b$, different than the main 2-cell, contributes at most 2 to the generalized boundary length of $\mathcal{D}_{a \cup b}$. Therefore, we know that $|\tilde{\partial}\mathcal{D}_{a \cup b}| \leq 2 |\mathcal{D}_{a \cup b}|$. 

Now, by Corollary \ref{Corollary:Bounded_size_of_tree_of_diagrams_hex} the diagram $\mathcal{D}_{a \cup b}$ has bounded number of 2-cells (the bound depends only on the density of the hexagonal model), so by the local version of the Generalized Isoperimetric Inequality (Lemma \ref{Lemma:Local_generalized_isoperimetric}), we conclude that w.o.p. there is no such diagram $\mathcal{D}_{a \cup b}$. This ends the proof of Assertion (\ref{Assertion:no_three_intersect}).

\textbf{Proof of Assertion (\ref{Assertion:not_0cycle})}. Suppose, that $\lambda$ is a hypergraph cycle, i.e.: its first and last vertex coincide. By Assertion (\ref{Assertion:no_three_intersect}) no three edges of $\lambda$ intersect, and by the inductive assumption, all segments of length $\leq k$ are not cycles, so every vertex of $\lambda$ belongs to at most two edges of $\lambda$. Therefore there exists an admissible intra-segment $\lambda'$ such $\varphi(\lambda') = \varphi(\lambda)$. Hence, there exists a tree of diagrams collared by $\lambda'$. By the fact that $\lambda$ is a hypergraph cycle, the main 2-cell of the tree of diagrams collared by $\lambda'$ has at most two external edges (since $\lambda$ joins its antipodal boundary points). It contradicts Theorem \ref{Theorem:Tree_of_diagrams_boundary_hex}. 
\end{proof}

Now, we will define a new system of hypergraphs, called \textbf{bent hypergraphs}, that provides a structure of a space with walls on $\cay$. Theorem \ref{Theorem:Properties_of_hyp_segment_hex} holds with overwhelming probability for a random group in the hexagonal model at density $d < \frac{1}{3}$. Let $G$ be a random group in the hexagonal model for which the statement of Theorem \ref{Theorem:Properties_of_hyp_segment_hex} is satisfied. The upcoming definitions are suitable only for such a group $G$. Again, let $\cay$ be the Cayley complex of $G$. 

\begin{cor}\label{Corrolary:Maximally_two_times}
With overwhelming probability every 2-cell of $\cay$ contains at most two edges of one hypergraph.
\end{cor}

\begin{proof}
The statement results simply by Assertion (\ref{Assertion:no_three_intersect}) of Theorem \ref{Theorem:Properties_of_hyp_segment_hex}.
\end{proof}

\begin{defi}\label{Definition:Crossing}
A 2-cell of $\cay$ that contains exactly two edges of one standard hypergraph is called a \textit{crossing}. A 2-cell of $\cay$ that contains edges of a three distinct standard hypergraphs is called a \textit{regular} 2-cell.
\end{defi}
By Corollary \ref{Corrolary:Maximally_two_times} every 2-cell in the Cayley complex of $G$ is either a crossing or a regular 2-cell.

\begin{defi}[\textbf{bent hypergraph}]\label{Definition:Bent_hypergraph_hex}
We will define a graph $\Gamma_{b}$. The vertices of $\Gamma_b$ is the set $V$ of midpoints of edges of $\cay$. We join $x,y \in V$ by an edge:
\begin{itemize}
\item  if $x$ and $y$ correspond to the antipodal midpoints of edges of a \textbf{regular} 2-cell, or
\item if $x$ and $y$ are two midpoints of edges of a \textbf{crossing} $c$ such that $x$ and $y$ lie on one standard hypergraph, are \textbf{not antipodal} and are \textbf{not midpoints of consecutive edges of $c$} (see Figure \ref{Figure:bent_hypergraph_hex}).
\end{itemize}
A connected component of $\Gamma$ is called a \textit{bent hypergraph}. Every edge of bent hypergraph that connects not antipodal midpoints of edges of a 2-cell is called a \textit{bent edge}.

There exists  a natural combinatorial map $\varphi: \Gamma_b \to \cay$ sending each vertex to a corresponding midpoint of 1-cell of $\cay$ and each edge $[v_1, v_2]$ of $\Gamma_b$ to a segment in $\cay$ joining $\varphi(v_1)$ with $\varphi(v_2)$.
\end{defi}

\begin{figure}[h]
\centering
\begin{tikzpicture}[scale=0.80]

\draw (-3,0) -- (-1,0) -- (0,1.5) -- (-1,3) -- (-3,3) -- (-4,1.5) -- cycle;
\draw[dashed] (-3.5337,0.7919) -- (-0.5311,2.2461);
\draw[dashed] (-3.4928,2.2256) -- (-0.5015,0.751);
\draw[very thick, blue] (-3.5317,0.8047) -- (-0.5065,0.7509);
\draw[very thick, blue] (-3.4907,2.2463) -- (-0.5341,2.2544);
\node at (-2.0566,0.3673) {$c_{cross}$};
\draw (0,1.5) -- (2,1.5) -- (3,3) -- (2,4.5) -- (0,4.5) -- (-1,3) -- cycle;
\draw[very thick, blue] (-0.509,2.2531) -- (2.4503,3.7755);
\node at (1,2) {$c_{reg}$};
\node at (1,3.5) {$\Lambda_{bent}$};

\end{tikzpicture}
\caption{Thick segments represent edges of a bent hypergraph $\Lambda_{bent}$ in a crossing $c_{cross}$ and in a regular 2-cell $c_{reg}$. For comparison, dashed lines are edges of a standard hypergraph.}
\label{Figure:bent_hypergraph_hex}
\end{figure}
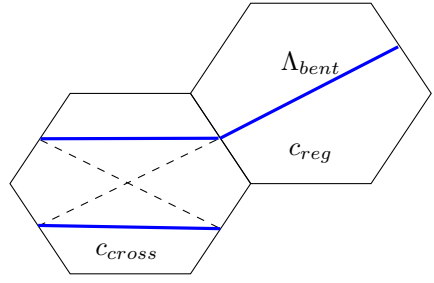

\begin{rem}\label{Remark:Connected_componentsC1} Two midpoints of edges of $\cay$ are connected by a bent hypergraph iff are connected by a standard hypergraph. 
\end{rem}

\begin{proof}
Let $x$ and $y$ be the points connected by a standard hypergraph $\Lambda$ and let $x = x_0, x_1, \dots, x_n = y$ be the vertices of an edge-path in $\Lambda$ connecting $x$ with $y$. Then, for every $1 \leq i \leq n$, points $x_{i-1}$ and $x_{i}$ are antipodal boundary points of a 2-cell $c_i$ of $\cay$. For $1 \leq i \leq n$, let $e_i$ be the edge in $\Lambda$ joining $x_{i-1}$ with $x_i$. Let $\Lambda_b$ be a bent hypergraph containing $x$. 

We will prove, by induction, that for every $1 \leq i \leq n$ the point $x_i$ belongs to $\Lambda_b$.  Consider the point $x_i$. If the edge $e_i$ belongs to $\Lambda_b$ there is nothing to prove. Otherwise, $x_i$ is the end of a bent edge of $\Lambda_b$ and $c_i$ is a crossing. Then, by Definition  \ref{Definition:Crossing} and Definition \ref{Definition:Bent_hypergraph_hex}  there exists a segment of a bent hypergraph $\Lambda_b$ joining $x_{i-1}$ with $x_i$.
\end{proof}

\begin{rem}\label{Remark:Connected_componentsC2} If $x$ and $y$ are two ends of a bent edge $e$ contained in a 2-cell $c$, then there exists a \textbf{unique} segment $\lambda_c$ of a standard hypergraph joining $x$ with $y$ (see Figure \ref{Figure:ear_hex}). Moreover, $c$ contains  only one edge of $\lambda_c$.
\end{rem}

\begin{proof}
Bent edge must be contained in a crossing, so there exists a segment $\lambda_0$ of a standard hypergraph, such that its first and last edge are contained in $c$. We obtain $\lambda_c$ such that its first edge is contained in $c$ and that $x$ and $y$ are ends of $\lambda_c$, by removing the appropriate  one of the edges of $\lambda_0$ contained in $c$ (the first or the last one). 

We will prove now, that such $\lambda_c$ is unique. Suppose, on the contrary, that there exists a segment $\lambda_c' \neq \lambda_c$ of a standard hypergraph such that the ends of $\lambda_c'$ are points $x$ and $y$. Then the concatenation $\lambda_{loop}:= \lambda_c \cup (\lambda_c')^{-1}$ is a closed edge-path of a standard hypergraph. It may contain back-tracks, which can be reduced in a standard way. By the fact that $\lambda_c \neq \lambda_c'$ the procedure of removing back-tracks in $\lambda_{loop}$ results in a non-trivial cycle of a standard hypergraph. By Assertion (\ref{Assertion:not_0cycle}) of Theorem \ref{Theorem:Properties_of_hyp_segment_hex} a segment of a standard hyperghraph w.o.p. cannot form a cycle, thus we obtain a contradiction. \end{proof}

\begin{figure}[h]
\centering
\begin{tikzpicture}[scale=0.80]

\draw (-3,0) -- (-1,0) -- (0,1.5) -- (-1,3) -- (-3,3) -- (-4,1.5) -- cycle;
\draw[dashed] (-3.5337,0.7919) -- (-0.5311,2.2461);
\draw[very thick, blue] (-3.5317,0.8047) -- (-0.5065,0.7509);
\node at (-2.5,2.5) {$c$};
\draw[fill = black]  (-0.4951,0.7445) circle (0.05);
\draw[fill = black]  (-3.5238,0.8073) circle (0.05);
\node at (-3.4985,1.0581) {$x$};
\node at (-0.5472,1.037) {$y$};
\draw[dashed]  plot[smooth, tension=.7] coordinates {(-0.5244,2.2388) (1.8776,3.0387) (3.5348,2.5936) (3.3088,0.8522) (-0.4658,0.7434)};
\node at (1.5,3.5) {$\lambda_c$};
\node at (-2,0.5) {$e$};
\draw[dotted] (-0.4625,0.7084) -- (-3.4353,2.306);

\end{tikzpicture}
\caption{The dashed line is the segment of a standard hypergraph with the same ends as a bent edge (an ear of $c$). The dotted segment is the edge of $\lambda_0$ that was removed to create $\lambda_c$.}
\label{Figure:ear_hex}
\end{figure}
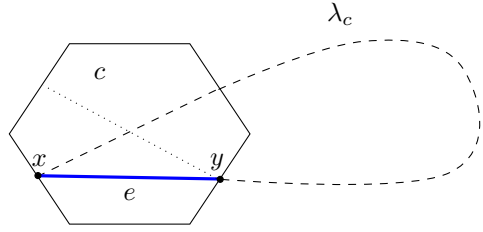

\begin{defi}[ear]
Such $\lambda_c$ as in the statement of Remark \ref{Remark:Connected_componentsC2} will be called an \textit{ear (of $c$)}. 
\end{defi}

By  Remark \ref{Remark:Connected_componentsC2} every crossing has exactly one ear.

\begin{rem}\label{Remark:No_self_intersections}
There are no two edges of a bent hypergraph that intersect.
\end{rem}

\begin{proof}
The statement results immediately from Definition \ref{Definition:Bent_hypergraph_hex}.
\end{proof}

\begin{theorem}\label{Theorem:Bent_hypergraph_hexs_are_embedded_trees_hex}
Every bent hypergraph is an embedded tree into $\cay$.
\end{theorem}

Before we prove Theorem \ref{Theorem:Bent_hypergraph_hexs_are_embedded_trees_hex} we need to determine how crossings are organized in $\cay$.

\begin{lem}\label{Lemmma:Possible_ear_intersections_hex}
Let $c_1$ and $c_2$ be two crossings of the same standard hypergraph. Let $\lambda_1$ be the ear of $c_1$ and $\lambda_2$ be the ear of $c_2$. Then w.o.p. \textbf{exactly} one of the following holds:
\begin{enumerate}
\item $\lambda_1$ does not enter $c_2$ and $\lambda_2$ does not enter $c_1$
\item $\lambda_1$ enters $c_2$ and $\lambda_2$ does not enter $c_1$
\item $\lambda_2$ enters $c_1$ and $\lambda_1$ does not enter $c_2$.
\end{enumerate} 
\end{lem}

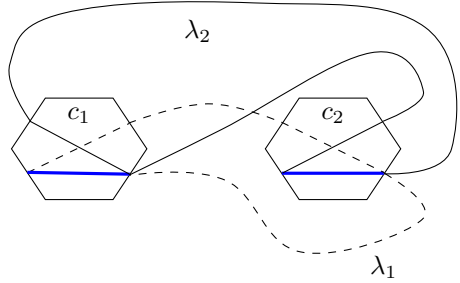
\begin{figure}[h]
\centering
\begin{tikzpicture}[scale=0.45]

\draw (-3,0) -- (-1,0) -- (0,1.5) -- (-1,3) -- (-3,3) -- (-4,1.5) -- cycle;
\draw[dashed] (-3.5337,0.7919) -- (-0.5311,2.2461) node (v2) {};
\draw[very thick, blue] (-3.5317,0.8047) -- (-0.5065,0.7509);
\node at (-2,2.5) {$c_1$};

\draw (4.5,0) -- (6.5,0) -- (7.5,1.5) -- (6.5,3) -- (4.5,3) -- (3.5,1.5)--cycle;
\draw[very thick, blue] (3.9887,0.7787) -- (7.0012,0.7787);
\draw[dashed]  (3.9724,2.2906) -- (7.0172,0.795);

\draw[dashed]   plot[smooth, tension=.7] coordinates {(-0.5468,2.257) (1.8,2.8313) (3.9621,2.2904)};
\draw[dashed]   plot[smooth, tension=.7] coordinates {(7.0277,0.7711) (8.1356,-0.5843) (4.5417,-1.5476) (2.3271,0.5667) (-0.5019,0.7173)};
\node at (7,-2) {$\lambda_1$};
\draw (-3.4407,2.3096) -- (-0.4776,0.7321);
\draw (4.0122,0.7881) -- (6.965,2.2998);
\node at (5.5,2.5) {$c_2$};
\draw  plot[smooth, tension=.7] coordinates {(7.0098,0.7623) (9,1.5) (8.2205,5.0932) (3.1631,5.8146) (-2.4825,5.5056) (-4.047,4.0092) (-3.4287,2.2989)};
\draw  plot[smooth, tension=.7] coordinates {(-0.5149,0.7281) (2.2222,2.0798) (6.6103,4.3385) (8.1705,3.1872) (6.9761,2.3157)};
\node at (1.5,5) {$\lambda_2$};

\end{tikzpicture}
\caption{An impossible way of an intersection of two ears.}
\label{Figure:two_ears_hex}
\end{figure}

\begin{proof}
It suffices to prove that w.o.p. it cannot happen that $\lambda_1$ enters $c_2$ and $\lambda_2$ enters $c_1$. Suppose, on the contrary, that $\lambda_1$ enters $c_2$ and $\lambda_2$ enters $c_1$ (see Figure \ref{Figure:two_ears_hex}). 

Note, that an ear $\lambda$ of a crossing $c$ can be prolonged to a segment $\lambda_{long}$ of standard hypergraph such that its first and last edge intersect in a 2-cell $c$. Let $\lambda_{ear}$ be the intra-segment of $\lambda_{long}$. A tree of diagrams collared by $\lambda_{ear}$ will be called the \textit{ear diagram of $\lambda$}. Note, that $c$ is the main 2-cell of the ear diagram of $\lambda$.  

\begin{figure}[h]
\centering
\begin{tikzpicture}[scale=0.45]

\draw (-3,0) -- (-1,0) -- (0,1.5) -- (-1,3) -- (-3,3) -- (-4,1.5) -- cycle;
\draw[dashed] (-1.986,1.4922) -- (-0.5311,2.2461) node (v2) {};
\draw[very thick, blue] (-3.5,0.7749) -- (-0.5065,0.7509);
\node at (-2,2.5) {$c$};

\draw[dashed]  (-0.4827,0.7407) -- (-1.9868,1.5136);

\draw[dashed]  plot[smooth, tension=.7] coordinates {(-0.4927,0.749) (2.8477,0.3206) (5.9198,1.0591) (6.5,3) (3,4) (-0.5221,2.2614)};
\draw  plot[smooth, tension=.7] coordinates {(-1,0) (2.8791,-0.5953) (7.037,0.2853) (7.6613,3.6406) (2.8234,5.0229) (-1.0001,2.9784)};
\node at (2,2.5) {$\lambda_{ear}$};
\draw[fill=black]  (-2.0074,1.5167) circle (0.05);
\node at (4.5,2) {$\mathcal{D}_{ear}$};

\end{tikzpicture}
\caption{The ear diagram of $\lambda$.}
\label{Figure:ear_diagram_hex}
\end{figure}
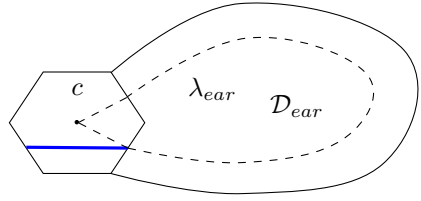 

Let $\mathcal{D}_1$ and $\mathcal{D}_2$ be the ear diagrams of $\lambda_1$ and $\lambda_2$ respectively. Let $\mathcal{D}_{sum}$ be the image of $\mathcal{D}_1 \cup \mathcal{D}_2$ in $\cay$ under the natural combinatorial map $\varphi : \mathcal{D}_1 \cup \mathcal{D}_2 \to \cay$. Let $c_1^s$ and $c_2^s$ denote the images of $c_1$ and $c_2$ in $\mathcal{D}_{sum}$ respectively. 

By Theorem \ref{Theorem:Tree_of_diagrams_boundary_hex}  we know that, apart from the crossing $c$, every other 2-cell of an ear diagram contributes at most $2$ to the generalized boundary length of it. Therefore, every 2-cell of $\mathcal{D}_{sum}$, possibly apart from images of $c_1$ and $c_2$ under $\varphi$, contribute at most 2 to its generalized boundary length (the natural combinatorial map can only identify edges of a diagram, it cannot tear glued faces apart).

Note that $c_1$ is not the main 2-cell of $\mathcal{D}_2$ and also $c_2$ is not the main 2-cell of $\mathcal{D}_1$. Therefore, by Theorem \ref{Theorem:Tree_of_diagrams_boundary_hex}, $c_1$ contributes at most 2 to the generalized boundary of $\mathcal{D}_2$ and analogously $c_2$ contributes at most 2 to the generalized boundary of $\mathcal{D}_2$. Since $\varphi$ can only create more identifications of 2-cells and edges, we conclude that every 2-cell of $\mathcal{D}_{sum}$ contributes at most 2 to $|\tilde{\partial} \mathcal{D}_{sum}|$. Therefore, $|\tilde{\partial} \mathcal{D}_{sum}| \leq 2|\mathcal{D}_{sum}|$. 

The diagram $\mathcal{D}_{sum}$ may not have a structure of a diagram with $K$-small hull (for some known $K$) since $\varphi$ can identify many 2-cells and edges of $\mathcal{D}_1$ and $\mathcal{D}_2$ in a complicated way. However, by Corollary \ref{Corollary:Bounded_size_of_tree_of_diagrams_hex} the diagram $\mathcal{D}_{sum}$ has a number of 2-cells bounded by a number depending only on the density of the hexagonal model. Therefore, we can apply Lemma \ref{Lemma:Local_generalized_isoperimetric} (local version of the Generalized Isoperimetric Inequality) to obtain that w.o.p. such $\mathcal{D}_{sum}$ cannot exist in $\cay$. This ends the proof.

\end{proof}

Lemma \ref{Lemmma:Possible_ear_intersections_hex} allows us to introduce a  partial order on the set of crossings in $\cay$:

\begin{defi}\label{Definition:Order_on_crossings}
We say that a crossing $c$ is \textit{greater} then a crossing $c'$ iff the ear of $c$ enters $c'$ (which by Lemma \ref{Lemmma:Possible_ear_intersections_hex} means that the ear of $c'$ does not enter $c$).
\end{defi}

Now we can provide the proof of Theorem \ref{Theorem:Bent_hypergraph_hexs_are_embedded_trees_hex}.

\begin{proof}
Suppose, on the contrary, that there exists a bent hypergraph $\Lambda$ that is not an embedded tree. It means that either there is a segment $\lambda$ of $\Lambda$ that is a cycle in hypergraph or there is a segment $\lambda'$ of $\Lambda$ such that its first and last edge intersect (but not coincide). By Remark \ref{Remark:No_self_intersections} such $\lambda'$ cannot exist. Hence, we are left with the proof that w.o.p. $\lambda$ being a hypergraph cycle cannot exist in $\cay$. 

Suppose, on the contrary, that such $\lambda$ exists. Without loss of generality, we can assume that $\lambda$ has no repetition of edges or vertices (otherwise we could split $\lambda$ into shorter hypergraph cycles). 

If there is no bent edge in $\lambda$ then $\lambda$ is not a hypergraph cycle by Theorem \ref{Theorem:Properties_of_hyp_segment_hex}. Therefore, we are left with the case where $\lambda$ contains a bent edge.

Let $D := \{d_1, d_2, \dots d_i \}$ be the set of crossings containing two bent edges of $\lambda$. Let $C: = \{c_1, c_2, \dots, c_j\}$ be the set of crossings containing only one bent edge of $\lambda$. 

We will now describe a procedure of removing bent edges from $\lambda$, that will turn it into a hypergraph cycle of a standard hypergraph. 

For every pair $e_a$ and $e_b$ of bent edges contained in one element of $D$  we ,,tie'' them, that is we replace the pair $\{e_a, e_b\}$ with the pair of edges $e_a'$ and $e_b'$ of standard hypergraph, such that the pair $\{ e_a', e_b' \}$ has the same set of ends as the pair $\{e_a, e_b\}$ (see Figure \ref{Figure:tie_edges_hex}).

\begin{figure}[h]
\centering
\begin{tikzpicture}[scale=0.95]

\draw (-3,0) -- (-1,0) -- (0,1.5) -- (-1,3) -- (-3,3) -- (-4,1.5) -- cycle;
\draw[very thick, blue] (-3.5317,0.8047) -- (-0.4449,0.7986);
tes {(-0.5244,2.2388) (3.4156,3.0415) (5.3696,2.5681) (2,0.5) (-0.4658,0.7434)};
\draw[very thick, blue] (-3.4706,2.2661) -- (-0.5248,2.2482);

\begin{scope}[shift={(6,0)}]
\draw (-3,0) -- (-1,0) -- (0,1.5) -- (-1,3) -- (-3,3) -- (-4,1.5) -- cycle;
\draw[very thick] (-3.5337,0.7919) -- (-0.5311,2.2461) node (v2) {};
tes {(-0.5244,2.2388) (3.4156,3.0415) (5.3696,2.5681) (2,0.5) (-0.4658,0.7434)};
\draw[very thick]  (-0.4827,0.7407) -- (-3.4628,2.2504);
\end{scope}

\node at (-2,0.5) {$e_a$};
\node at (-2,2.5) {$e_b$};
\node at (3.3336,2.3072) {$e_a'$};
\node at (3.2231,0.8672) {$e_b'$};
\draw[->, very thick] (0.5,1.5) -- (1.5,1.5);

\end{tikzpicture}
\caption{The modification of edges of $\lambda$ in a 2-cell belonging to $D$.}
\label{Figure:tie_edges_hex}
\end{figure}
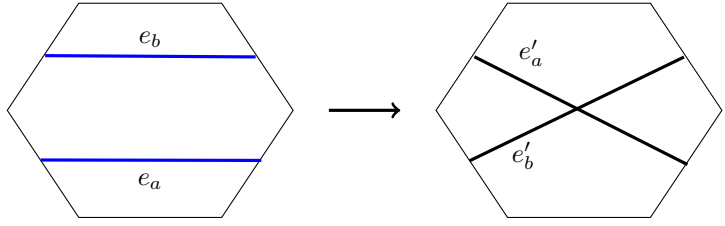 

For every bent edge that is contained in a crossing $c \in C$ we replace such edge by the ear of $c$. This procedure results in a loop $\lambda_{st}$ of a standard hypergraph, however $\lambda_{st}$ may be non-reduced, i.e., it may contain backtrackings. In the worst case scenario reducing back-trackings could make $\lambda_{st}$ trivial. To prove that $\lambda_{st}$ stays non-trivial after the reduction, we will use the order on crossings introduced in Definition \ref{Definition:Order_on_crossings}. 

Let $c_{max}$ be a maximum crossing of $C$. The segment $\lambda$ has only one edge contained in $c_{max}$, since $\lambda$ has no repetition of edges and vertices and $c_{max} \in C$. By the fact that $c_{max}$ is a maximum crossing, none of other ears of elements of $C$ enter $c_{max}$.  Therefore, $\lambda_{st}$ cannot be turned into trivial loop by the reduction procedure, since the edge of $\lambda_{st}$ contained in $c_{max}$ is unique. This contradicts Theorem \ref{Theorem:Properties_of_hyp_segment_hex}.
\end{proof}

\section{Lack of Property (T) in the hexagonal model}\label{Section:Lack_of_T_hex}

\begin{lem}\label{Lemma:Two_components}\emph{(Reformulation of {\cite[Lemma 2.3]{ow11}})}
Let $\lambda$ be a bent hypergraph. Then $\cay - \lambda$ consists of exactly two connected components.
\end{lem}

By Lemma \ref{Lemma:Two_components} for a bent hypergraph $\Lambda$ there are exactly two connected components of $\cay - \Lambda$.  Denote them as $\Lambda^+$ and $\Lambda^-$ (we choose arbitrarily which component is  
$\Lambda^+$). 

\begin{lem}\label{Lemma:Sides_change}
Let $G$ be a random group in the hexagonal model at density $\frac{1}{6} < d < \frac{1}{3}$. Then w.o.p. there exists in $\cay$ a bent hypergraph $\Lambda$ and a group element $g \in  G$ such that $g(\Lambda^+) = \Lambda^-$ and $g(\Lambda^-)=\Lambda^+$. 
\end{lem}

\begin{proof}
First, we will prove that there exists a 2-cell $c$ in $\cay$ such that two of its antipodal edges are labeled by the same letter $x$. Up to a multiplicative factor, there is $(2n-1)^5$ cyclically reduced words, such that its antipodal letters are the same. An idea behind the proof is that this set has density $\frac{5}{6}$, and the set of relators has density strictly larger than $\frac{1}{6}$, so they sum up to more then 1, thus they must w.o.p. intersect. Formally, we estimate the probability of the existence of such 2-cell $c$ below.

Let $m := (2n-2)$. Then the set of relators consisting of 5 distinct letters has size larger then $m^5$. Therefore, up to a multiplicative factor, the probability $P_{\sim c}$, that there is no such 2-cell $c$ in $\cay$ can be estimated in the following way (all non-integer numbers are rounded):

$$P_{\sim c} < \frac{ \binom{m^6 - m^5}{m^{6d}} }{ \binom{m^6}{m^{6d}} } < \left( \frac{m^6 - m^5}{m^6} \right)^{m^{6d}} = \left(1 - \frac{1}{m}\right)^{m^{6d}} \to 0,$$

as $m \to \infty$, since $6 d > 1$.

Now, we will prove that w.o.p. a 2-cell $c$ having the same letters on two of its antipodal edges cannot contain a bent edge. Suppose, on the contrary, that $c$ contains a bent edge. Then $c$ is a crossing, so there exists a segment of a standard hypergraph, with the first and the last edge contained in $c$. Hence, there exists a tree of diagrams $\mathcal{D}$, such that $c$ is the main 2-cell of $\mathcal{D}$. By Theorem \ref{Theorem:Tree_of_diagrams_boundary_hex}, we know that every 2-cell of $\mathcal{D}$, different then $c$ contributes at most 2 to the generalized boundary length of $\mathcal{D}$, and $c$ contributes four to $|\parti D|$. Since two edges of $c$ are labeled with the same letter, we can identify them, obtaining a new diagram $\mathcal{D}'$. The diagram $\mathcal{D}'$ may not be reduced, so consider the image $\varphi(\mathcal{D}')$ of $\mathcal{D}'$ in $\cay$ under the natural combinatorial map. Identifying two edges in a diagram reduces the generalized boundary length by 2, so we conclude that: 

$$|\tilde{\partial} \varphi(\mathcal{D}')| \leq 2 |\varphi(\mathcal{D}')|.$$

By Corollary \ref{Corollary:Bounded_size_of_tree_of_diagrams_hex}, the diagram $\mathcal{D}$ has the number of 2-cells bounded by a number depending only on the density. Therefore, $\varphi(\mathcal{D}')$ satisfies the same bound on the number of 2-cells, so we can apply Lemma \ref{Lemma:Local_generalized_isoperimetric} (local version of the Generalized Isoperimetric Inequality) to conclude, that such $\varphi(\mathcal{D}')$ w.o.p. cannot exist in $\cay$. This ends the proof that $c$ does not contain any bent edge. 

Let $e$ and $e'$ be two antipodal edges of $c$, that are labeled by the same letter, and let this letter be called $x$. Without loss of generality, we can suppose that the begging of $e$ is the vertex corresponding to the neutral element of $G$. Let $y$ and $z$ be the two letters on two consecutive edges of $c$ between $e$ and $e'$ according to the orientation. We define $\Lambda$ to be the bent hypergraph containing the edge joining the middle of $e$ with the middle of $e'$. The desired group element $g$ can be given by formula $g=xyz$. Note, that the action of $g$ on $\cay$ satisfies $ge=e'$, so by the fact that $c$ does not contain any bent edges the action of $g$ on $\cay$ preserves the bent hypergraph $\Lambda$. Moreover, the edges $e$ and $e'$ cross $\Lambda$ in different directions thus $g(\Lambda^+) = \Lambda^-$ and $g(\Lambda^-)=\Lambda^+$.
\end{proof}

The following proof is mimicking of the proof of \cite[Theorem 1.1]{przyt}. 

\begin{proof}[\textbf{Proof of Theorem \ref{MainTheorem:Sharp_hexagonal}}]
Let $g$ and $\Lambda$ be as in Lemma \ref{Lemma:Sides_change}. Let $H$ be the stabilizer of $\Lambda$. We know, that $H$ acts cocompactly on $\Lambda$, since the set of hypergraphs is $G$-invariant. 

First, we will prove that both components: $\Lambda^+$, $\Lambda^-$ of $\cay - \Lambda$ are not within a finite distance from $\Lambda$. Suppose, on the contrary, that $\Lambda^+$ or $\Lambda^-$ is at a finite distance from $\Lambda$. By the fact that $\Lambda^+$ and $\Lambda^-$ are exchanged by $g$, it means that both components $\Lambda^+$ and $\Lambda^-$ are then at a finite distance from $\Lambda$. Therefore, $H$ act cocompactly on $\Lambda$, so $G$ is quasi-isometric to $H$, which means that $G$ is also quasi-isometric to $\Lambda$. By Theorem \ref{Theorem:Bent_hypergraph_hexs_are_embedded_trees_hex} $\Lambda$ is a tree, and by  Corollary \ref{Corollary:Torsionfree} $G$ is torsion-free. Therefore, by Stallings Theorem \cite{sta68}, the group $G$ is free. This contradicts the fact that the Euler characteristic $\chi(G)$ of $G$ is positive.

Let $H' \subset H$ be the index 2 subgroup stabilizing $\Lambda^+$ and $\Lambda^-$. By the previous observations, the number of relative ends $e(G, H') > 1$. Therefore, by \cite{nr98} the action of $G$ on a CAT(0) cube complex, provided by Sageev construction (see \cite[Theorem 3.1]{sag95}) is nontrivial, so $G$ does not have Kazhdan's Property (T). 

We have shown that a random group in the hexagonal model at density $\frac{1}{6} < d < \frac{1}{3}$ w.o.p. does not have Property (T). By the fact that random groups at higher densities are w.o.p. quotients of random groups at lower densities, we can easily show that for every density $< \frac{1}{3}$ a random group in the hexagonal model w.o.p. does not have Property (T); the formal approach with all details is the same as in the proof of \cite[Lemma 5.15]{odrz} (which is the analogous fact for the square model).

The results above, combined with Proposition \ref{Proposition:T_in_hex_positive}, give the proof of Theorem \ref{MainTheorem:Sharp_hexagonal}.
\end{proof}

\section{Bent walls in the square model at density $< \frac{3}{8}$} \label{Section:Bent_walls_square}

The goal of this section is to prove Theorem \ref{MainTheorem:T_square}. First, we will explain why the density $\frac{3}{8}$ is the limit for our methods. In the square model we will need the fact that hypergraphs do not contain cycles. Observe, that for densities $> \frac{3}{8}$ the diagram presented in Figure \ref{Figure:long_cycle_square} cannot be w.o.p. excluded, so there may exist arbitrarily long cycles in hypergraphs.  

\begin{figure}[h]
\centering
\begin{tikzpicture}[scale=0.95]

\draw  (-2.5,1.5) rectangle (-1.5,0.5);
\draw  (-0.5,0.5) rectangle (-1.5,1.5);
\draw  (0.5,0.5) rectangle (-0.5,1.5);
\draw  (1.5,0.5) rectangle (0.5,1.5);
\draw  (-1.5,-0.5) rectangle (-2.5,0.5);
\draw  (-0.5,-0.5) rectangle (-1.5,0.5);
\draw  (0.5,-0.5) rectangle (-0.5,0.5);
\draw  (1.5,-0.5) rectangle (0.5,0.5);
\node at (2.5,0.5) {$\dots$};
\draw (-2.5,-0.5);
\draw (-2.5,-0.5) -- (-4,0.5) -- (-2.5,1.5);
\draw (-4,0.5) -- (-2.5,0.5);
\draw[fill=black]  (-3.1587,0.5146) circle (0.05);
\draw (1.5,1.5) -- (2,1.5);
\draw (1.5,-0.5) -- (2,-0.5);
\draw (1.5,0.5) -- (2,0.5);
\draw  (3.5,1.5) rectangle (4.5,0.5);
\draw  (4.5,-0.5) rectangle (3.5,0.5);
\draw (3,1.5) -- (3.5,1.5);
\draw (3,0.5) -- (3.5,0.5);
\draw (3,-0.5) -- (3.5,-0.5);
\draw (4.5,-0.5) -- (6,0.5) -- (4.5,1.5);
\draw (4.5,0.5) -- (6,0.5);
\draw[fill=black]  (5.0785,0.4949) circle (0.05);
\draw[thick, red] (2,1) -- (-2.5,1) -- (-3.5111,0.4973) -- (-2.5,0) -- (2,0);
\draw[thick, red] (3,1) -- (4.5,1) -- (5.5,0.5) -- (4.5,0) -- (3,0);

\end{tikzpicture}
\caption{In the square model at density $d > \frac{3}{8}$ there may exist an arbitrarily long cycle of a hypergraph.}
\label{Figure:long_cycle_square}
\end{figure} 

In this section, we will analyze how the notion of a tree of diagrams works in the case of the square model. We will perform the same steps as in the case of hexagonal model, however, all proofs will be easier (or even omitted), since the square model is, in general, simpler than the hexagonal model. The following lemma is analogous to Lemma \ref{Lemma:Gluing_simple_cycles_hex}. 

\begin{lem}\label{Lemma:Gluing_simple_cycles_square}
Let $\lambda$ be an admissible intra-segment in the square model and suppose that the first and the last vertex of $\lambda$ coincide. Let $\mathcal{T}$ be the tree of loops of $\lambda$ and $\mathcal{D}_{\mathcal{T}}$ be the tree of diagrams collared by $\mathcal{T}$. Let $\mathcal{C}_1$ and $\mathcal{C}_2$ be two simple cycles of $\mathcal{T}$ and let $\mathcal{D}_1$ and $\mathcal{D}_2$ be diagrams collared by hypergraph realizations of $\mathcal{C}_1$ and $\mathcal{C}_2$ respectively. Let $c_1$ and $c_2$ be shell corners of $\mathcal{D}_1$ and $\mathcal{D}_2$ respectively. Suppose that $c_1$ and $c_2$ are twins. Let $\gamma_1$ be the boundary edge-path of $c_1$ along which $c_1$ is glued to $\mathcal{D}_1 - c_1$. Analogously, let $\gamma_2$ be the boundary edge-path of $c_2$ along which $c_2$ is glued to $\mathcal{D}_2 - c_2$. Let $c_{1,2}$ be the image of $c_1$ (and also $c_2$) under the map $\phi$ that sends each 2-cell of $\mathcal{D}_1 \cup \mathcal{D}_2$ to its image in $\mathcal{D}_{\mathcal{T}}$.

Then $\phi(\gamma_1)$ and $\phi(\gamma_2)$ are two edge-paths on the boundary of $c_{1,2}$ with distinct sets of edges and moreover:

$$|\phi(\gamma_1)| + |\phi(\gamma_2)| = 4.$$
\end{lem}

\begin{proof}
The proof is analogous to the proof of Lemma \ref{Lemma:Gluing_simple_cycles_hex}: By the fact that $c_1$ and $c_2$ are shell corners, both paths have length at most 2 and by the same observations as in the proof of Lemma \ref{Lemma:Gluing_simple_cycles_hex}, we know that $\phi(\gamma_1)$ is antipodal to $\phi(\gamma_2)$. Two antipodal paths of length at most two on the boundary of a square must be disjoint.
\end{proof}

The only difference between Lemma \ref{Lemma:Gluing_simple_cycles_square} and Lemma  \ref{Lemma:Gluing_simple_cycles_hex} is that in the case of the square model the 2-cell $c_{1,2}$ has no external edges, instead of possibly 2, as in the case of the hexagonal model.

\begin{cor}\label{Corollary:Component_is_with_small_legs_square}
Each component of the tree of diagrams in the square model is a diagram with 24-small hull.
\end{cor}

\begin{proof}
The proof is analogous to the proof of Corollary \ref{Corollary:Component_is_with_small_legs_hex}.
\end{proof}

\begin{theo}\label{Theorem:Tree_of_diagrams_boundary_square}
Let $\lambda$ be an admissible intra-segment in the square model and suppose that the first and the last vertex of $\lambda$ coincide. Let $\mathcal{D}_{\mathcal{T}}$ be a tree of diagrams of $\lambda$. Then w.o.p. $\mathcal{D}_{\mathcal{T}}$ is a reduced connected diagram with 24-small hull and

$$|\tilde{\partial} \mathcal{D}_{\mathcal{T}}| \leq |\mathcal{D_{\mathcal{T}}}| + 1. $$

Moreover, the main 2-cell of $\mathcal{D}_{\mathcal{T}}$ is a shell corner with exactly two external edges and also every 2-cell of $\mathcal{D}_{\mathcal{T}}$, different than the main 2-cell, contributes at most 1 to the generalized boundary length of $\mathcal{D}_{\mathcal{T}}$.
\end{theo}

\begin{proof} The proof is analogous to the proof of Theorem \ref{Theorem:Tree_of_diagrams_boundary_hex}; we just use Lemma \ref{Lemma:Gluing_simple_cycles_square} instead of Lemma \ref{Lemma:Gluing_simple_cycles_hex} and we consider boundary deviation of any diagram $\mathcal{D}$ defined as $|\tilde{\partial} \mathcal{D}| - |\mathcal{D}|$ (instead of $|\tilde{\partial} \mathcal{D}| - 2|\mathcal{D}|$ as in case of the hexagonal model).
\end{proof}

\begin{cor}\label{Corollary:Bounded_size_of_tree_of_diagrams_square}
For any fixed $\varepsilon > 0$ in the square model at density $d < \frac{3}{8}$ w.o.p. every tree of diagrams has not more than $\frac{1}{3 - 8d - \varepsilon}$ 2-cells. 
\end{cor}

\begin{proof}
The proof is analogous to the proof of Corollary \ref{Corollary:Bounded_size_of_tree_of_diagrams_hex}; we just use Theorem \ref{Theorem:Tree_of_diagrams_boundary_square} instead of Theorem \ref{Theorem:Tree_of_diagrams_boundary_hex}.
\end{proof}

\begin{theorem}\label{Theorem:Properties_of_hyp_segment_square}
Let $\lambda$ be a hypergraph segment in the Cayley complex of a random group in the square model at density $d < \frac{1}{3}$. Then w.o.p. $\lambda$ is not a cycle, i.e. the first and the last vertex of $\lambda$ do not coincide.
\end{theorem}

\begin{proof}
If $\lambda$ is a hypergraph cycle, then the main 2-cell of the tree of diagrams collared by the intra-segment of $\lambda$ has at most one external edge (since $\lambda$ joins two antipodal points on its boundary), and this contradicts Theorem \ref{Theorem:Tree_of_diagrams_boundary_square}.
\end{proof}

Theorem \ref{Theorem:Properties_of_hyp_segment_square} holds w.o.p. in the square model at density $d < \frac{3}{8}$. From now, until the end of Section \ref{Section:Bent_walls_square}, suppose that the random group $G$ in the square model is chosen such that the statement of Theorem \ref{Theorem:Properties_of_hyp_segment_square} holds and let $\cay$ be the Cayley complex of such a group $G$.

We define crossings and regular 2-cells in the same way as in the hexagonal model:
\begin{defi}
A 2-cell of $\cay$ that contains exactly two edges of one standard hypergraph is called a \textit{crossing}. A 2-cell of $\cay$ that contains edges of two distinct standard hypergraph is called a \textit{regular} 2-cell.
\end{defi}

\begin{lem}\label{Lemma:One_intra_segment_for_crossing_square}
Let $c$ be a crossing. Then there exists exactly one intra-segment $\lambda$ such that the first and the last vertex of $\lambda$ lies in the middle of $c$. Moreover, $\lambda$ crosses two consecutive edges of the boundary of $c$ (see Figure \ref{Figure:ear_square}).
\end{lem}

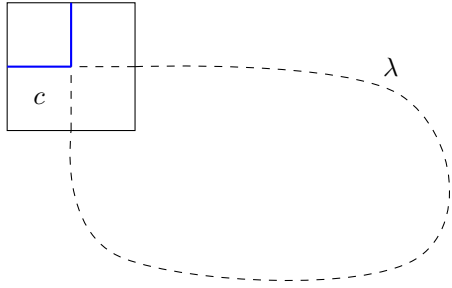
\begin{figure}[h]
\centering
\begin{tikzpicture}[scale=0.85]

\draw  (0.5,-0.5) rectangle (2.5,-2.5);
\draw[dashed] (2.5,-1.5) -- (1.5,-1.5);
\draw[dashed] (1.5,-2.5) -- (1.5,-1.5);
\draw[thick, blue] (0.5,-1.5) -- (1.5,-1.5);
\draw[thick, blue] (1.5,-0.5) -- (1.5,-1.5);
\draw[dashed]  plot[smooth, tension=.7] coordinates {(2.5,-1.5) (6.6823,-1.966) (6.8701,-4.5234) (2.3275,-4.5024) (1.5,-2.5)};
\node at (6.5,-1.5) {$\lambda$};
\node at (1,-2) {$c$};

\end{tikzpicture}
\caption{The dashed line represents an intra-segment $\lambda$ with the first and last vertex in the middle of the crossing $c$.}
\label{Figure:ear_square}
\end{figure} 

\begin{proof}
The existence of such $\lambda$ results by the definition of crossing. The intra-segment $\lambda$ cannot join the antipodal points on the boundary of $c$, since by Thorem \ref{Theorem:Properties_of_hyp_segment_square} a hypergraph cannot form a cycle.
Now, we need to prove the uniqueness of such $\lambda$. Suppose that there exist two intra-segments: $\lambda_1$ and $\lambda_2$ that 
have its first and last vertices in the middle of $c$. By the fact that 2-cells in $\cay$ are squares, we know that the concatenation $\lambda_{1,2} := \lambda_1 \cup \lambda_2$ is a cycle (it requires only to choose the right order of joining these two intra-segments, but they always prolong each other).  The edge-path $\lambda_{1,2}$ may contain back-tracking, but we can remove them in the standard way obtaining a reduced non-trivial path, that is a cycle. This path can be transformed into a cycle of hypergraph by removing the vertices correspoinding to the middle of $c$ and adding an appropriate edge joining the midpoints of the boundary edges of $c$. Therefore, we constructed a non-trivial cycle of hypergraph, which contradicts Theorem  \ref{Theorem:Properties_of_hyp_segment_square}.
\end{proof}

\begin{defi}
Such $\lambda$ as in the statement of Lemma \ref{Lemma:One_intra_segment_for_crossing_square} will be called the \textit{intra-ear} of the crossing $c$ (see Figure \ref{Figure:ear_square}).
\end{defi}

Now we will define \textbf{bent hypergraphs} in the square model. The definition is very similar to Definition \ref{Definition:Bent_hypergraph_hex}.

\begin{defi}[\textbf{bent hypergraph}]\label{Definition:Bent_hypergraph_square}
We define a graph $\Gamma_{b}$ in the following way: The vertices of $\Gamma_b$ is the set $V$ of midpoints of edges of $\cay$. We join $x,y \in V$ by an edge:
\begin{itemize}
\item  if $x$ and $y$ correspond to the antipodal midpoints of edges of a \textbf{regular} 2-cell, or
\item if $x$ and $y$ are two midpoints of edges of a \textbf{crossing} $c$ that  are \textbf{not antipodal} and \textbf{exactly one of them lie on the intra-ear of $c$} (see Figure \ref{Figure:bent_square}).
\end{itemize}
A connected component of $\Gamma$ is called a \textit{bent hypergraph (in the square model)}. Every edge of bent hypergraph that connects not antipodal midpoints of edges of a 2-cell is called a \textit{bent edge}.

There exists  a natural combinatorial map $\varphi: \Gamma_b \to \cay$ sending each vertex to a corresponding midpoint of 1-cell of $\cay$ and each edge $[v_1, v_2]$ of $\Gamma_b$ to a segment in $\cay$ joining $\varphi(v_1)$ with $\varphi(v_2)$.
\end{defi}

\begin{figure}[h]
\centering
\begin{tikzpicture}[scale=0.85]
\draw  (0.5,-0.5) rectangle (2.5,-2.5);
\draw[dashed] (2.5,-1.5) -- (1.5,-1.5);
\draw[dashed] (1.5,-2.5) -- (1.5,-1.5);
\draw[very thick, blue] (0.5,-1.5) -- (1.5,-2.5);
\draw[very thick, blue] (1.5,-0.5) -- (2.5,-1.5);
\draw[dashed]  plot[smooth, tension=.7] coordinates {(2.5,-1.5) (6.6823,-1.966) (6.8701,-4.5234) (2.3275,-4.5024) (1.5,-2.5)};
\node at (1.1204,-1.0113) {$c_{cross}$};
\draw  (-1.5,-2.5) rectangle (0.5,-0.5);
\draw[very thick, blue] (-1.5,-1.5) -- (0.5,-1.5);
\node at (-0.5,-1) {$c_{reg}$};
\end{tikzpicture}
\caption{Thick edges represent edges of a bent hypergraph in a crossing $c_{cross}$ and a regular 2-cell $c_{reg}$. The dashed line is an intra-ear of $c_{cross}$.}
\label{Figure:bent_square}
\end{figure}
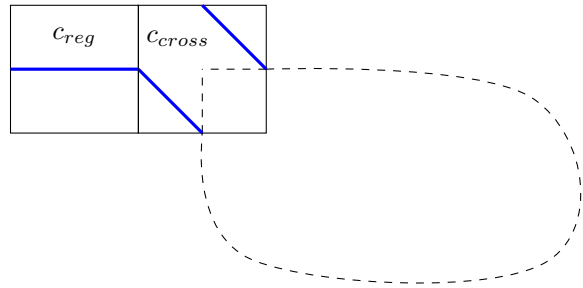 

\begin{defi}
Let $e$ be a bent edge. The segment of a standard hypergraph having the same ends as $e$ is called the \textit{ear} of $e$ (the uniqueness of such segment results simply by Lemma \ref{Lemma:One_intra_segment_for_crossing_square}).
\end{defi}

\begin{theorem}\label{Theorem:Bent_hypergraph_hexs_are_embedded_trees_square}
In the square model at density $d < \frac{3}{8}$ w.o.p. every bent hypergraph is an embedded tree into $\cay$.
\end{theorem}

\begin{proof}
The proof is analogous to the proof of Theorem \ref{Theorem:Bent_hypergraph_hexs_are_embedded_trees_hex}: first we introduce the order on crossings in the square model. In the statement and in the proof of Lemma \ref{Lemmma:Possible_ear_intersections_hex} we do not use the fact that we work specifically in the hexagonal model, but only the fact that trees of diagrams have bounded size. This holds as well in the square model, according to Corollary \ref{Corollary:Bounded_size_of_tree_of_diagrams_square}. Therefore, we can define the order on crossings in the same way as in Definition \ref{Definition:Order_on_crossings} and repeat the rest of the proof of Theorem \ref{Theorem:Bent_hypergraph_hexs_are_embedded_trees_hex}. 
\end{proof}

\begin{lem}\label{Lemma:Two_components_square}\emph{(Reformulation of {\cite[Lemma 2.3]{ow11}})}
Let $\lambda$ be a bent hypergraph in the square model. Then $\cay - \lambda$ consists of exactly two connected components: $\Lambda^+$ and $\Lambda^-$.
\end{lem}

\begin{lem}\label{Lemma:Sides_change_square}
Let $G$ be a random group in the square model at density $\frac{1}{4} < d < \frac{3}{8}$. Then w.o.p. exist in $\cay$  a bent hypergraph $\Lambda$ and group element $g \in  G$ such that $g(\Lambda^+) = \Lambda^-$ and $g(\Lambda^-)=\Lambda^+$. 
\end{lem}

\begin{proof} The proof is analogous to the proof of Lemma \ref{Lemma:Sides_change}. 
\end{proof}

\begin{proof}[Proof of Theorem \textbf{\ref{MainTheorem:T_square}}]
The proof is almost the same as the part of the proof of Theorem \ref{MainTheorem:Sharp_hexagonal} concerning the lack of Property (T) in the hexagonal model. The only difference is that now we use Lemma \ref{Lemma:Sides_change_square} instead of Lemma \ref{Lemma:Sides_change}.
\end{proof}

\end{document}